\documentclass[draft,a4paper,12pt,leqno,oneside]{amsart}

\usepackage{setspace} 
\usepackage{typearea} % European margins

\usepackage{amscd,amsmath,amssymb,verbatim}

\usepackage%[pdftex]
{graphicx} \setkeys{Gin}{width=\linewidth}

\usepackage{ifthen}

\usepackage{xr-hyper}
\usepackage{paralist,cite}
\usepackage[colorlinks,backref,final,bookmarksnumbered,bookmarks]{hyperref}
\usepackage[backrefs]{amsrefs}

\newcommand{\arxiv}[2][]{\ifthenelse{\equal{#1}{}}
{\href{http://arxiv.org/abs/#2}{\tt arXiv:#2}}
{\href{http://arxiv.org/abs/math/#2}{\tt arXiv:math.#1/#2}}}

\theoremstyle{plain}
\newtheorem{theorem}{Theorem}[section]
\newtheorem{lemma}[theorem]{Lemma}
\newtheorem{corollary}[theorem]{Corollary}
\newtheorem{proposition}[theorem]{Proposition}
\newtheorem{problem}[theorem]{Problem}
\newtheorem{addendum}[theorem]{Addendum}

\theoremstyle{definition}
\newtheorem{example}[theorem]{Example}

\newtheoremstyle{remark}
{}{}{}{}{\itshape}{}{ }{\thmname{#1}\thmnumber{ \itshape #2.}}
\theoremstyle{remark}
\newtheorem{remark}[theorem]{Remark}

\newtheoremstyle{concise}
{}{}{}{}{\bfseries}{}{ }{\thmnumber{#2.}\thmnote{ #3.}}
\theoremstyle{concise}
\newtheorem{definition}[theorem]{}

\newenvironment{roster}[1][0]
{

\begin{enumerate}\setcounter{enumi}{#1}}
{\end{enumerate}}

 \def\R{\mathbb{R}} \def\Z{\mathbb{Z}} 
\def\x{\times} \def\but{\setminus} \def\emb{\hookrightarrow} \def\incl{\subset}
 \def\phi{\varphi} \def\emptyset{\varnothing}
\def\xr#1{\xrightarrow{#1}} \def\xl#1{\xleftarrow{#1}} \renewcommand{\:}{\colon}
\DeclareMathOperator{\st}{st} \DeclareMathOperator{\lk}{lk}
\DeclareMathOperator{\Lk}{Lk} \DeclareMathOperator{\id}{id}
\DeclareMathOperator{\Fr}{Fr} \DeclareMathOperator{\Int}{Int}
\DeclareMathOperator{\hocolim}{hocolim}
\def\join{\mathop{\raisebox{-3pt}{\Huge$*$}}} \def\cojoin{\,\bar*\,}
\def\bydef{\mathrel{\mathop:}=}

\def\Cel{\lfloor} \def\Cer{\rfloor}
\def\Fll{\lceil} \def\Flr{\rceil}
\def\downscale#1{\mathchoice{\raisebox{1pt}{$\scriptstyle#1$}}
{\raisebox{1pt}{$\scriptstyle#1$}}{\raisebox{.5pt}{$\scriptscriptstyle#1$}}
{{\scriptscriptstyle#1}}}
\def\cel{\downscale\Cel} \def\cer{\downscale\Cer}
\def\fll{\downscale\Fll} \def\flr{\downscale\Flr}
 
\def\bigfl{\Fll} \def\bigfr{\Flr}

\begin{document}

\title{Combinatorics of combinatorial topology}
\author{Sergey A. Melikhov}
\address{Steklov Mathematical Institute of the Russian Academy of Sciences,
ul.\ Gubkina 8, Moscow, 119991 Russia}
\email{melikhov@mi.ras.ru}
\address{Delaunay Laboratory of Discrete and
Computational Geometry, Yaroslavl' State University,
14 Sovetskaya st., Yaroslavl', 150000 Russia}
\thanks{Supported by Russian Foundation for Basic Research Grant No.\ 11-01-00822,
Russian Government project 11.G34.31.0053 and Federal Program ``Scientific and
scientific-pedagogical staff of innovative Russia'' 2009--2013}

\begin{abstract} We develop a tighter implementation of basic PL topology,
which keeps track of some combinatorial structure beyond PL homeomorphism type.
With this technique we clarify some aspects of PL transversality and give
combinatorial proofs of a number of known results.

New results include a combinatorial characterization of collapsible polyhedra
in terms of constructible posets (generalizing face posets of constructible
simplicial complexes in the sense of Hochster).
The relevant constructible posets are also characterized in terms of Reading's
zipping of posets, which is a variation of edge contractions in simplicial
complexes.
\end{abstract}

\maketitle
%\tableofcontents

\part{INTRODUCTION}

This work was originally motivated by uniform homotopy theory (see \cite{M2},
which depends heavily on Part \ref{combinatorics} of the present paper) and,
independently, embedding theory (see \cite{M1}).
In short, the applications force a rethink of PL topology, which involves a shift
of focus from simplicial complexes to general posets (also known as van Kampen's
star complexes or Akin's cone complexes) and cell complexes (also known as regular
PL CW-posets), as well as from spaces to maps.

It gradually became clear to the author that the same kind of techniques should
also apply to improve his understanding of PL transversality and collapsing.
This is essentially the subject of Parts \ref{subdivision and collars}
and \ref{chapter-collapsing}.
The results about collapsing may sound more convincing, yet they are largely
due to a progress with transversality --- even if it is quite modest per se.

A fifth motivation is the search of purely combinatorial foundations of
PL topology, like in the early works by Alexander and Newman (see \cite{Li},
\cite[Chapter II]{Gl}), but which would be (at least) as efficient as
the standard foundations by Whitehead and Zeeman \cite{Wh}, \cite{Ze},
\cite{St}, \cite{RS}, based on arbitrary linear subdivisions (and hence on
affine geometry).
A well-known obstacle here is Alexander's problem (see \cite[p.\ 311]{Li}
and \cite[comments to Theorem B]{Mn} for a recent discussion): given two
PL homeomorphic simplicial complexes, do they have isomorphic stellar
subdivisions?
(Equivalently, is having isomorphic stellar subdivisions an equivalence relation?)

Apparently it would be good enough to prove a theorem of this kind about not
necessarily stellar subdivisions of not necessarily simplicial complexes,
provided that all the notions involved are purely combinatorial.
(See \cite{McP} for a slightly different approach, which has led to a rather active area of research.)
This suggests analyzing the combinatorial content of PL transversality theorems,
and getting one's hands dirty with exploring more general subdivisions of
more general complexes.
That sounds pretty close to what the present paper is about.

The discussion of some highlights of the paper just below and in
\S\ref{collapsing-intro} is followed by an actual introduction
in \S\ref{examples of posets}, which is supposed to communicate some geometric
intuition for posets and explain why they can (and perhaps should) replace
simplicial complexes in a combinatorially biased treatment of PL topology.

\section{Collapsible polyhedra and maps} \label{collapsing-intro}

It is well known (cf.\ \cite{Mc}, \cite{Bj}) that a simplicial complex,
and more generally a cell complex (=finite CW-complex whose attaching
maps are PL embeddings) can be reconstructed from the poset of its
{\it nonempty} faces.
Namely, if $P$ is the poset of nonempty cells of a cell complex $K$
(ordered by inclusion), then the cells of $K$ can be identified with
the subcomplexes of the order complex of $P$ that are the order
complexes of the principal order ideals of $P$.
We will therefore identify simplicial and cell complexes with their
posets of nonempty faces.

\subsection{Constructible poset}
We call a poset $P$ {\it constructible} if either $P$ has a greatest element or
$P=Q\cup R$, where $Q$ and $R$ are order ideals (that is, if $p\le q$ where $q\in Q$
then $p\in Q$; and similarly for $R$), each of $Q$, $R$ and $Q\cap R$
is constructible, and every maximal element of $Q\cap R$ is covered by a maximal
element of $Q$ and by a maximal element of $R$.

Constructible simplicial and polytopal complexes originate in the work of M. Hochster
on Cohen--Macaulay rings \cite{Ho} and have been widely studied in Topological
Combinatorics (see references in \cite{Ha0}, \cite{Ben}).
The above definition is supposed to be the topologist's version of the notion,
and has two somewhat different relations with the original combinatorialist's definition:

\begin{itemize}
\item If $P$ is (the poset of nonempty faces of) an {\it acyclic} simplicial or
polytopal complex, then $P$ is constructible in the above sense if and only if it is
constructible in the sense of Hochster (see Lemma \ref{Mayer-Vietoris}(c)).

\item If $P$ is the ``face poset'', in the combinatorialist's sense (that is,
including the empty face as well!) of a simplicial or polytopal complex $K$, then
$P$ is constructible in the above sense if and only if $K$ is constructible in
the sense of Hochster (see Lemma \ref{Hochster-lemma}).
\end{itemize}

Order complexes of constructible posets are clearly contractible
(by Mayer--Vietoris and Seifert--van Kampen).
The dunce hat of Zeeman \cite{Ze2} (also known as the Borsuk tube), that is,
the mapping cone of the map $\phi\:S^1\to S^1$ corresponding to the word $aa^{-1}a$,
is an example of a contractible polyhedron that cannot be triangulated by
a constructible simplicial complex \cite{Ha0}, \cite{Ha1}.

There is no simple relation between collapsible polyhedra and constructible simplicial
complexes.
Not all collapsible polyhedra can be triangulated by constructible simplicial complexes,
since a constructible simplicial complex is easily seen to be pure (i.e.\ its maximal
simplices all have the same dimension).
Conversely, Hachimori constructed a constructible simplicial complex triangulating
a non-collapsible polyhedron $P$ \cite{Ha0}, \cite{Ha1}.
Namely, $P$ the mapping cone of the map $\tilde\phi$ defined by the pullback diagram
$$\begin{CD}
S^1@>\tilde\phi>>S^1\\
@VVV@VVpV\\
S^1@>\phi>>S^1,
\end{CD}$$
where $p$ is the double covering and $\phi$ is as above.

\begin{theorem}\label{constructible-main}
Let $X$ be a compact polyhedron. The following are equivalent:

(i) $X$ is collapsible;

(ii) $X$ can be triangulated by a simplicial complex whose dual poset is constructible;

(iii) $X$ can be cellulated by a cell complex whose dual poset is constructible.
\end{theorem}

The characterization of collapsible polyhedra by constructible posets
that are dual to cell complexes may be seen as surprising.
In a collapse, one bites off a single simplex, or ball, at each step;
in the destruction of a constructible poset, one cuts into two parts at each step,
where each part (as well as their common boundary) may be arbitrarily complex in
the sense of the number of further cuts.

A PL map is called {\it collapsible} if its point-inverses are collapsible
(so in particular nonempty).
The proof of Theorem \ref{constructible-main} also yields a version for maps
(see Corollary \ref{collapsible-constructible map} and Theorem
\ref{constructible point-inverses}):

\begin{theorem} A PL map between compact polyhedra is collapsible if and only if
it can be triangulated by a simplicial map whose point-inverses have constructible
dual posets.
\end{theorem}

Note that the point-inverses of a simplicial map (other than the preimages
of the vertices) are usually not simplicial complexes.
It is not hard to see, however, that they are cubosimplicial complexes
(Lemma \ref{cubosimplicial}).

\subsection{Edge-zipping}\label{edge-zipping}
An {\it edge contraction} is a surjective simplicial map $f\:K\to L$ between
simplicial complexes such that every vertex of $L$ has only one preimage, apart
from one vertex, whose preimage is an edge, $v*w$ (written here as the join of
its vertices).
We call $f$ an {\it elementary edge-zipping} if any of the following equivalent
conditions hold:

\begin{itemize}
\item $\lk(v)\cap\lk(w)=\lk(v*w)$;
\item $v*w$ is not contained in any ``missing face'' of $K$, i.e.\ in
an isomorphic copy of $\partial\Delta^n$ in $K$ that does not extend to
an isomorphic copy of $\Delta^n$ in $K$;
\item the PL map triangulated by $f$ is collapsible.
\end{itemize}

We say that a simplicial complex $K$ {\it edge-zips} onto a simplicial complex $L$
if there exists a sequence of elementary edge-zippings $K=K_0\to\dots\to K_n=L$.

Every triangulation of $S^2$ edge-zips onto $\partial\Delta^3$ (Steinitz, 1934;
see \cite{Zo}, \cite{Su}).
In fact, if $K$ is a $2$-dimensional simplicial complex, then every simplicial
subdivision of $K$ edge-zips onto $K$ \cite{M1}.
On the other hand, there exists a simplicial subdivision of $\partial\Delta^4$ that
does not edge-zip onto $\partial\Delta^4$ (see \cite[Example 6.1]{Ne}, where
a stronger assertion is proved).

\begin{proposition}\label{edge-zipping-main}
A compact polyhedron is collapsible if and only if it can be triangulated by
a simplicial complex that edge-zips onto a point.
\end{proposition}

It should not be hard to prove Proposition \ref{edge-zipping-main} directly
(see Remark \ref{trivial-remark} for a sketch of proof of an assertion
that contains the ``if'' implication).
We will see in a moment that it follows from a more interesting result.

\subsection{Zipping}
Let $P$ be a poset.
Suppose that a $p\in P$ covers two incomparable elements $q,r\in P$
so that:

$\bullet$ if $s<p$ and $s\ne q,r$, then $s<q$ and $s<r$, and

$\bullet$ if $s>q$ and $s>r$, then $s\ge p$.

\noindent
In this situation $P$ is said to {\it elementarily zip} onto the quotient
of $P$ by the subposet $\{p,q,r\}$.
(See \S\ref{quotient poset} concerning quotient posets.)
A {\it zipping} of posets is a sequence of elementary zippings.

This definition was introduced by N. Reading \cite{Re}.

As observed in \cite[2.5(2)]{MN}, if a poset $P$ elementarily zips onto
a poset $Q$, then the order complex of $P$ edge-zips onto that of $Q$ in two steps.
The two edge contractions of course correspond to the chains $\{q<p\}$ and $\{r<p\}$.

Conversely:

\begin{lemma}\label{zips if edge-zips} If a simplicial complex $K$
edge-zips onto a simplicial complex $L$, then $K$ zips onto $L$.
\end{lemma}

The author learned from E. Nevo that he had independently been aware of this fact.

\begin{proof} It suffices to consider the case where $L$ is obtained from $K$
by an elementary edge-zipping; let $a*b$ be the edge being contracted.
Let $A_1,\dots,A_n$ be the simplices of the link of $a*b$ in $K$ arranged in
an order of increasing dimension, and let $L_k$ be the subcomplex
$A_1\cup\dots\cup A_{k-1}$ of $K$.
Let $f_k\:L_k*a*b\to L_k*c$ be the simplicial quotient map identifying
$a$ with $b$, and let $Q_k$ be the adjunction poset $K\cup_{f_k}L_k*c$
(see Lemma \ref{adjunction poset}).
Then $Q_1=K/(a*b)$ (the quotient poset), so $K$ zips onto $Q_1$.
Furthermore, for each $k=1,\dots,n$ the image of the simplex $A_k*a*b$ of $K$ in $Q_k$
is isomorphic to $(\partial(A_k*c))+(a*b)$, where the prejoin $X+Y$ of the posets $X$,
$Y$ consists of all the elements of $Y$ placed above all the elements of $X$,
retaining the original orders within $X$ and $Y$ (see \S\ref{operations}).
It is easy to check by induction that the image of $A_k*a*b$ is the only cell of $Q_k$
that contains both the images of $A_k*a$ and $A_k*b$, and it follows that $Q_k$
elementarily zips onto $Q_{k+1}$.
Thus $K$ zips onto $Q_n=L$.
\end{proof}

As promised, Proposition \ref{edge-zipping-main} is now a consequence of
Lemma \ref{zips if edge-zips}, Theorem \ref{constructible-main} and the following

\begin{theorem}\label{zipping-main}
A cell complex $K$ zips onto a point if and only if the dual poset $K^*$
is constructible.
\end{theorem}

One can interpret Theorems \ref{constructible-main} and \ref{zipping-main} as
providing two quite different combinatorial characterizations of collapsibility,
which only magically turn out to be equivalent combinatorially and not just
topologically.

Given the high flexibility of constructible posets and the very concrete
character of the zipping operation, one is tempted to think that it is
constructible posets and zipping that are the most adequate combinatorial
representation of collapsible polyhedra --- rather than, say, simplicially
collapsible or edge-contractible simplicial complexes.

\subsection{Scheme of the proof}
The proof of Theorems \ref{constructible-main} and \ref{zipping-main} in Part
\ref{chapter-collapsing} involves the following steps.

$\bullet$ If $K$ is a triangulation of a collapsible polyhedron, then some simplicial
subdivision of $K$ is simplicially collapsible (Whitehead \cite{Wh}; see also
\cite{Ze}, \cite{RS}, \cite{AB}).

$\bullet$ If a simplicial complex $K$ is simplicially collapsible, then trivially
it is collapsible as a poset (see definition in \S\ref{collapsing of posets}).

$\bullet$ If a poset $P$ is collapsible, then $(P^\flat)^*$ is shellable
(Corollary \ref{collapsible-shellable2}), where $P^\flat$ denotes the barycentric
subdivision (=the poset of nonempty faces of the order complex) and
$Q^*$ denotes the dual poset. (This step is a combinatorial abstraction of
a rather well-known construction in PL topology.)

$\bullet$ If a poset $P$ is shellable (see definition in \S\ref{shelling}), then
trivially it is constructible.

$\bullet$ If $K$ is a simplicial complex such that $K^*$ is constructible, then
$K$ zips onto a point (Theorem \ref{zips if constructible}(b)).

$\bullet$ If $K$ or $K^*$ is a cell complex, and $K$ zips onto a point, then
$K$ transversely zips onto a point (Theorem \ref{transverse zipping vs zipping}).

$\bullet$ If $K$ or $K^*$ is a cell complex, and $K$ transversely zips onto a point
(see definition in \S\ref{zipping-definition}), then $K^*$ is transversely constructible
(Theorem \ref{transversely constructible if zips2}(a)).

$\bullet$ If $K$ is transversely constructible (see definition in
\S\ref{constructible-definition}), then the polyhedron triangulated by $K$ is collapsible
(Lemma \ref{constructible-collapsible}).

\section{PL basics and PL transversality revisited}

Having discussed the results of Part \ref{chapter-collapsing}, let us now review the results
of Parts \ref{combinatorics} and \ref{subdivision and collars}.

Part \ref{combinatorics} develops a purely combinatorial and
self-contained theory,
without any homeomorphisms hidden in the background.
Most results of this chapter are valid for posets of any cardinality, and
serve as lemmas either for further chapters or for further papers by
the author (or both).

We start by developing a combinatorial star/link/join/prejoin technique
(\S\ref{operations}) which enables one to reduce some basic PL homeomorphisms
such as $X*Y=CX\x Y\cup_{X\x Y}X\x CY$ and $\lk((x,y),X\x Y)\cong\lk(x,X)*\lk(y,Y)$
to combinatorial isomorphisms (which is not really possible within the usual
simplicial theory).

This is complemented by a technique of canonical subdivision
(\S\ref{intervals and cubes}), which can be seen as a nearly universal alternative
to the familiar barycentric subdivision --- so that, for instance, it provides
an alternative combinatorial handle decomposition of a triangulated manifold
(\S\ref{handles}).
We find the canonical subdivision to be useful in PL transversality
(\S\ref{transversality}) and absolutely indispensable in the theory of uniform
polyhedra (\cite{M2}), where the barycentric subdivision is totally useless.

We also show equivalence of a local and a global definition of (possibly locally
infinite) simplicial complexes (Theorem \ref{2.10}) and compare two kinds of
combinatorial weakly infinite-dimensional cubes using a version of the
Davis--Januszkiewicz/Babson--Billera--Chan mirroring construction
(\S\ref{weak join} and Theorem \ref{3.5}).

Finally, we include a combinatorial treatment of Hatcher's homotopy colimit
decomposition of simplicial maps (Theorem \ref{hocolim}), which incidentally
turns out to have a cubical character (Lemma \ref{Hatcher is cubical}),
and observe that it implies Homma's theorem on factorization of simplicial
maps (Corollary \ref{factorization}).

\begin{definition}[Convention/disclaimer]
Starting with Part \ref{subdivision and collars}, we confine our attention to finite posets.
Additional related conventions are stated at the beginning of Part \ref{subdivision and collars}.
\end{definition}

Starting with Part \ref{subdivision and collars}, we also occasionally use known results from
PL topology.
Thus the remaining chapters are not entirely self-contained.
We do try to keep external references to a minimum.

Two key notions are introduced in Part \ref{subdivision and collars}: a subdivision map
and a stratification map.

A special case of subdivision maps where the domain and the range are cell complexes
was studied by R. Stanley \cite{Stan} (see also \cite{Ath}) and also by N. Mn\"ev, who calls 
them assembly or aggregation maps \cite{Mn}.
The opposite case, where the domain and the range are dual posets of simplicial
complexes was considered by E. Akin \cite{Ak2}.
The equivalence of a local and a global definition of a subdivision map (Theorem
\ref{subdivision map}) yields a simple proof of some of Akin's results.

In the case where the range is a cell complex, stratification maps are
a combinatorial abstraction of the notion of a polyhedral mock bundle of
Buocristiano, Rourke and Sanderson \cite{BRS}.
A characterization of stratification maps in terms of subdivision maps
is given in Theorem \ref{stratification map}.

We give a combinatorial proof of Whitehead's collaring theorem (Theorem
\ref{collaring theorem}) and extend it to a ``multi-collaring'' theorem somewhat
reminiscent of Grothendieck's reflections on ``tame topology''
(Theorem \ref{grothendieck}).
The multi-collaring theorem implies that a stratification map gives rise to
a PL variety filtration in the sense of D. Stone \cite{Sto}, and in particular
to a locally conelike TOP stratified set in the sense of Siebenmann \cite{Sie}.

We give combinatorial proofs of McCrory's characterization of
PL manifolds (Theorem \ref{manifolds}), M. M. Cohen's partition theorem
for simplicial maps (Theorem \ref{comanifolds}) and
the Buoncristiano--Rourke--Sanderson amalgamation theorem
(Theorem \ref{amalgamation2}).
The partition and amalgamation theorems --- along with a subdivision theorem
which we don't consider in this paper --- form the basis of the geometric
treatment of generalized homology and cohomology theories
(Buoncristiano--Rourke--Sanderson, \cite[Chapter II]{BRS})
and are the main ingredients of the PL transversality theorem (see
\cite[Chapter II]{BRS}).

The Buoncristiano--Rourke--Sanderson definition of PL transversality makes sense
for a pair of PL maps from polyhedra to a polyhedron \cite[pp.\ 23-24, 35]{BRS}.
McCrory has shown that PL transversality of a pair of embedded polyhedra in
a manifold is a symmetric relation \cite{Mc}.
We show that McCrory's symmetric definition of transversality for embeddings
does not extend to the case of maps (Example \ref{pullback fails}); however,
it does extend to a symmetric notion of ``semi-transversality'' of maps such that
semi-transverse maps can be made transverse by a canonical procedure, involving
no arbitrary choices (Corollary \ref{transversality1} and Theorem \ref{stretching}).

\section{Geometry of posets and preposets}\label{examples of posets}

\subsection{Face poset of a convex polytope}\label{polytopes}
Let $P$ be a convex polytope, that is, the convex hull of a finite subset of some
Euclidean space $\R^d$.
The relation of inclusion on the set $\mathcal{F}_P$ of all nonempty faces of $P$
is a partial order, so we have the {\it face poset} $F_P=(\mathcal F_P,\incl)$.
If two convex polytopes are affinely equivalent, clearly their face posets
are isomorphic.

If $\sigma$ is an affine simplex, that is, the convex hull of
a nonempty finite affinely independent set of points $S\subset\R^d$, then $F_\sigma$
is isomorphic to the poset $\Delta^S$ of all nonempty subsets of $S$ ordered
by inclusion.

If $\pi$ is a parallelepiped, that is, the convex hull of the sums of vectors in
all subsets of a nonempty finite linearly independent set of vectors $S\subset\R^d$,
then $F_\pi$ is isomorphic to the poset $I^S$ of all intervals
$[S_-,S_+]=\{T\subset S\mid S_-\subset T\subset S_+\}$, where
$S_+\subset S_-\subset S$, ordered by inclusion.

\subsection{Duality in polytopes}\label{polytope-duality}
(See \cite{Gr}, \cite{Zi} for proofs of all assertions.)
If $P\subset\R^d$ is a convex polytope whose interior contains the origin of $\R^d$,
then its {\it polar} $P^\star=\{x\in\R^d\mid (x,y)\le 1\text{ for all } y\in P\}$
is a convex polytope containing the origin,%
\footnote{A convex polytope containing the origin is determined either 
by its vertices (as their convex hull) or by the hyperplanes through 
its facets (as the intersection of the half-spaces bounded by 
these hyperplanes and containing the origin).
The correspondence between the vertices of $P$ and the hyperplanes through
the facets of $P^\star$ comes from the duality of projective geometry.
In more detail, the canonical isomorphism $\R^d\to(\R^d)^{**}$, 
$v\mapsto(f\mapsto f(v))$, sends each vertex $v$ of $P$ to a nondegenerate 
linear function $\phi_v$ on $(\R^d)^*$.
The hyperplanes $\phi_v^{-1}(1)$ determine the canonically dual 
polytope $P^*$ in $(\R^d)^*$.
The polar $P^\star$ is the preimage of $P^*$ under the identification 
$\R^d\to(\R^d)^*$, $v\mapsto (v,\cdot)$.
The preimages $H_v$ of the hyperplanes $\phi_v^{-1}(1)$ in $\R^d$ 
can be constructed geometrically.
If $v$ lies outside the unit ball $B$, then $H_v\cap\partial B$
is the boundary of the region of $\partial B$ that visible from $v$.
In general, if $\R^d$ is identified with a hyperplane $H$ in $\R^{d+1}$ 
at distance $1$ from the origin, then the line through the origin
of $\R^{d+1}$ and a vertex $v\in H$ is orthogonal to a hyperplane
$\Pi_v$ through the origin of $\R^{d+1}$ such that the centrally
symmetric hyperplane $-\Pi_v$ meets $H$ along $H_v$.}
and $(P^\star)^\star=P$.
The polar of an affine $n$-simplex is an affine $n$-simplex, and that of an
$n$-parallelepiped is an $n$-cross-polytope; the other polar pairs of
regular polytopes are: icosahedron/dodecahedron; $24$-cell/$24$-cell;
and $120$-cell/$600$-cell.

The {\it boundary face poset} $\partial F_P=(\mathcal{F_P}\but\{P\},\incl)$
of a convex polytope $P$ whose interior contains the origin is dual to
the boundary face poset $\partial F_{P^\star}$ of the polar $P^\star$,
in the sense that there is a bijection
$\mathcal{F_P}\but\{P\}\to\mathcal{F_{P^\star}}\but\{P^\star\}$ that reverses
the order by inclusion.
In particular, the poset $\partial\Delta^S$ of all nonempty proper subsets of $S$
ordered by inclusion is self-dual.

A convex polytope is called {\it simplicial} if every its proper face is
an affine simplex.
Clearly, this is equivalent to saying that the face poset of every proper
face is isomorphic to a combinatorial simplex $\Delta^S$.
A convex $d$-polytope $P$ is called {\it simple} if every its vertex is incident
to precisely $d$ edges (which is the minimal possible number).
This is equivalent to saying that the poset of faces properly containing
any given vertex of $P$ is isomorphic to a combinatorial simplex.
Thus a convex polytope $P$ is simple if and only if $P^\star$ is simplicial.

\subsection{Affine simplicial complex}\label{2.3}
An {\it affine simplicial complex} is a finite set $\mathcal{K}$ of non-empty
affine simplices in some $\R^d$ such that
\begin{roster}
\item if $\sigma\in\mathcal{K}$ and $\tau$ is a non-empty face of $\sigma$, then
$\tau\in\mathcal{K}$;
\item if $\sigma,\tau\in\mathcal{K}$, then $\sigma\cap\tau\in\mathcal{K}$.
\end{roster}
The simplices of $\mathcal{K}$ form the {\it face poset} $K=(\mathcal{K},\incl)$
with respect to inclusion.

An affine simplicial complex is called {\it flag} if every subcomplex of $K$
whose face poset is isomorphic to $\partial\Delta^S$, $|S|>2$, is the boundary
of some simplex $\sigma\in\mathcal K$.

\subsection{Affine cone complex}\label{2.4}
An {\it affine cone complex} is a family $\mathcal{P}$ of
subcomplexes of an affine simplicial complex $\mathcal{K}$ such that
\begin{roster}
\item for each $\sigma\in\mathcal{P}$, all maximal simplices of $\sigma$
(i.e.\ those simplices of $\sigma$ that are not faces of other simplices
of $\sigma$) share a common vertex, denoted $\hat\sigma$;
\item for each $\sigma\in\mathcal{P}$, the set $\partial \sigma$ of all
simplices of $\sigma$ disjoint from $\hat\sigma$ is a union of elements
of $\mathcal{P}$;
\item if $\sigma,\tau\in\mathcal{P}$ and $\hat\sigma=\hat\tau$, then
$\sigma=\tau$.
\end{roster}
The cones (i.e.\ the elements) of $\mathcal{P}$ form its {\it face poset}
$P=(\mathcal{P},\incl)$ with respect to inclusion.

Affine cone complexes arose in van Kampen's dissertation \cite{vK} and also
in the work of M. M. Cohen and E. Akin (see \cite[p.\ 456]{Ak}) and were further
studied by McCrory \cite{Mc}.
The following proposition is found in \cite{Mc}:

\begin{proposition} \label{acc=posets}
Every finite poset is isomorphic to the poset of cones
of some affine cone complex.
\end{proposition}

\begin{proof}
Given a finite poset $P=(\mathcal P,\le)$, the combinatorial simplex
$\Delta^{\mathcal P}$ is the face poset of an affine simplex $\sigma$
whose vertices may be identified with the elements of $\mathcal P$.
Every nonempty finite chain $p_1<\dots<p_n$ in $P$ corresponds to the
affine subsimplex $p_1*\dots*p_n$ of $\sigma$ with vertices $p_1,\dots,p_n$.
The set of all such subsimplices forms an affine simplicial subcomplex
$\Delta(P)$ of $\sigma$.
For each $p\in P$ let $\fll p\flr$ be the subposet of $P$ consisting of
all $q\le p$.
Then $\Delta(\fll p\flr)$ is an affine simplicial subcomplex of
$\Delta(P)$.
It is easy to see that $\{\Delta(\fll p\flr)\mid p\in P\}$ is an affine
cone complex whose face poset is isomorphic to $P$.
\end{proof}

We find it instructive to review the proof of a special case in geometric terms.

\begin{corollary} \label{acc=posets2}
The face poset of every convex polytope is isomorphic to
the face poset of some affine cone complex.
\end{corollary}

\begin{proof}
If $P$ is a convex polytope in $\R^d$, then there exists an affine simplicial
complex $P^\flat$ in $\R^d$ such that each face of $P$ is a union of simplices
of $P^\flat$.
Namely, $P^\flat$ has one vertex $\hat\sigma$ in the interior of each nonempty
face $\sigma$ of $P$, and one simplex for every nonempty chain
$\sigma_1\incl\dots\incl\sigma_n$ of faces of $P$, namely the affine simplex
$\hat\sigma_1*\dots*\hat\sigma_n$ with vertices $\hat\sigma_1,\dots,\hat\sigma_n$.
Now the family of all subcomplexes of $P^\flat$ triangulating the nonempty
faces of $P$ is an affine cone complex whose face poset is isomorphic to $F_P$.
\end{proof}

\subsection{Dual cone complex of a simplicial complex}
Given an affine simplicial complex $\mathcal K$, by Proposition \ref{acc=posets}
its face poset $K=(\mathcal K,\le)$ is isomorphic to the face poset of an affine
cone complex $\mathcal P$, consisting of subcomplexes of an affine simplicial
complex $\mathcal K^\flat$.
The latter can be constructed similarly to the proof of Corollary
\ref{acc=posets2}, so that simplices of $K$ are unions of simplices of $K^\flat$.
The dual poset $K^*=(\mathcal K,\ge)$ has the same chains as $K$, and it
follows that it is isomorphic to the face poset of the affine cone complex
$\mathcal Q$, consisting of subcomplexes of $\mathcal K^\flat$ of the form
$\sigma^*=\{\hat\sigma_1*\dots*\hat\sigma_n\mid
\sigma_1>\dots>\sigma_n\ge\sigma\}$, where $\sigma\in K$.

\subsection{Coboundary of a simplicial cochain}\label{cochain coboundary}
Let $\mathcal K$ be an affine simplicial complex, and let $\mathcal Q$ be
its dual cone complex, as above.
The simplicial cochain $c_\sigma\in C^k(\mathcal K;\,\Z/2)$ given by
the indicator function of a $k$-simplex $\sigma\in\mathcal K$ has coboundary
$\delta(c_\sigma)=\sum_i c_{\tau_i}\in C^{k+1}(\mathcal K;\,\Z/2)$, where
$\partial(\sigma^*)=\bigcup_i\tau_i^*$ in $\mathcal Q$; see
\cite[figure on p.\ 26]{Fe}.

\begin{proposition}\label{2.4'} Let $\mathcal{P}$ be an affine cone complex.

(a) If $\sigma\in\mathcal{P}$, then every vertex of $\sigma$ is $\hat\rho$ for some
$\rho\in\mathcal{P}$ with $\rho=\sigma$ or $\rho\incl\partial\sigma$.

(b) For all $\sigma,\tau\in\mathcal{P}$, either $\sigma=\tau$, or
$\sigma\incl\partial \tau$, or $\tau\incl\partial \sigma$, or else
$\sigma\cap\tau\incl\partial \sigma\cap\partial \tau$.

(c) For all $\sigma,\tau\in\mathcal{P}$, the intersection $\sigma\cap\tau$ is
a union of cones of $\mathcal{P}$.
\end{proposition}

This can be easily deduced from Proposition \ref{acc=posets}, but a direct proof
contributes to our understanding of affine cone complexes.

\begin{proof}[Proof. (a)] If $v\ne \hat\sigma$, then $v\in\partial \sigma$.
By (ii), there exists a cone $\rho$ of $\mathcal{P}$ contained in
$\partial \sigma$ and containing $v$.
By finiteness we may assume that $\rho$ contains no other cone with this
property.
Then $v\notin\partial \rho$ by (ii).
Hence $v=\hat\rho$.
\end{proof}

\begin{proof}[(b)] We will show that either $\tau=\sigma$, or
$\tau\incl\partial \sigma$, or else $\sigma\cap\tau\incl\partial \tau$.
If $\hat\tau=\hat\sigma$, then $\tau=\sigma$ by (iii).
If  $\hat\tau\ne \hat\sigma$ and $\hat\tau\in\sigma$, then by (a), $\hat\tau=\hat\rho$
for some $\rho\in\mathcal{P}$, $\rho\incl\partial \sigma$.
Then $\tau=\rho$ by (iii), hence $\tau\incl\partial \sigma$.
Finally, if $\hat\tau\notin\sigma$, then $\sigma\cap\tau\incl\partial \tau$.
\end{proof}

\begin{proof}[(c)] By (b) we may assume that
$\sigma\cap\tau\incl\partial \sigma\cap\partial \tau$.
Let $A$ be a simplex of $\sigma\cap\tau$.
Similarly to the proof of (a), $A$ is contained in some cone $\lambda$ of
$\mathcal{P}$ contained in $\partial \sigma$ such that $A\notin\partial \lambda$,
and in some cone $\mu$ of $\mathcal{P}$ contained in $\partial \tau$ such that
$A\notin\partial \mu$.
Then by (b), $\lambda=\mu$.
\end{proof}

\subsection{Flag affine $\Delta$-set}
An {\it affine $\Delta$-set} is an affine simplicial complex $\mathcal K$ where
every simplex is endowed with a total ordering of vertices so that every
inclusion between simplices of $\mathcal K$ is monotone on the sets of vertices.
(See \cite{Fr}, \cite{RS2} concerning general $\Delta$-sets.)

Since the total orderings of vertex sets of simplices agree, they extend
to a partial order on the set of all vertices of $\mathcal K$.
Thus an affine $\Delta$-set is simply an affine simplicial complex endowed
with a partial ordering of vertices such that every two vertices connected
by an edge are comparable.

We call an affine $\Delta$-set $\mathcal L=(\mathcal K,\le)$ {\it flag} if
$\mathcal K$ is a flag simplicial complex.

If $\mathcal P$ is an affine cone complex consisting of subcomplexes of
an affine simplicial complex $\mathcal K$, then a partial ordering of the
vertices of $\mathcal K$ is given by the relation $\hat\sigma\le\hat\tau$
if and only if $\sigma\subset\tau$.
Clearly, two vertices of $\mathcal K$ are comparable if and {\it only if}
they are connected by an edge.
Moreover, it follows from Proposition \ref{2.4'}(b) that $\mathcal K$
is a flag complex.
Thus affine cone complexes can be viewed as a special case of flag affine
$\Delta$-sets.

\subsection{Strictly acyclic relation}
If $\mathcal L=(\mathcal K,\le)$ is a flag affine $\Delta$-set, then the
$1$-skeleton of $\mathcal K$ carries the structure of an acyclic digraph
(=a directed graph with no directed cycles).
An acyclic digraph $G=(V,E)$ amounts to a binary relation $E\subset V\x V$
that is {\it strictly acyclic} in the sense that it admits no cycles of
the form $v_1Ev_2E\dots Ev_nEv_1$.
Conversely, an acyclic digraph embedded in general position in $\R^d$ spans
a flag affine simplicial complex in $\R^d$, whose vertices are partially
ordered by the relation: $v\le w$ if and only if there exists a directed path
from $v$ to $w$ in the $1$-skeleton.

\begin{proposition} \label{order extension}
Every partial order on a finite set extends to a total order.
\end{proposition}

This is well-known, but we would like to emphasize a particular canonical
extension.

\begin{proof} A finite poset $P=(\mathcal P,\le)$ admits a monotone embedding
in the combinatorial simplex $\Delta^{\mathcal P}$ via $p\mapsto\fll p\flr$,
where $\fll p\flr=\{q\in\mathcal P\mid q\le p\}$.
It remains to compose this embedding with the monotone surjection
$\Delta^{\mathcal P}\to\{1,2,\dots,|\mathcal P|\}$, $S\mapsto |S|$.
\end{proof}

Proposition \ref{order extension} relates flag affine
$\Delta$-sets to PL Morse theory, cf.\ \cite{BB} and \cite[Example on p.\ 275]{Mc}.
See also Kearton--Lickorish \cite{KL} (along with references to Kosi\'nski and
Kuiper therein), Forman \cite{Fo} (along with elaborations in \cite{OW} and
\cite{Ko}) and Bestvina \cite{Be}.

\section{Acknowledgements}

The author would like to thank A. Aizenberg, V. M. Buchstaber, 
A. Gaifullin, D. Gugnin,
I. Izmestiev, L. Montejano, A. Skopenkov, E. Nevo, M. Skopenkov, D. Stepanov and
R. \v Zivaljevi\'c for useful discussions.

\part{FINITE AND INFINITE POSETS}\label{combinatorics}

\section{Basic notions}\label{cone complexes}

\subsection{Posets and anti-reflexive relations}
If $\prec$ is an {\it anti-reflexive} relation on a set $\mathcal P$
(i.e.\ $x\nprec x$ for all $x\in\mathcal{P}$), its {\it inclusive} counterpart
$\preceq$ is defined by $x\preceq y$ iff $x\prec y$ or $x=y$; it is reflexive
(i.e.\ $x\preceq x$ for all $x\in\mathcal{P}$).
For a reflexive relation $\preceq$, its {\it exclusive} counterpart $\prec$ is
defined by $x\prec y$ iff $x\preceq y$ and $x\ne y$; it is anti-reflexive.
The operations of inclusive/exclusive counterpart constitute mutually inverse
bijections between the set of all reflexive relations on a set $\mathcal{P}$ and
the set of all anti-reflexive relations on $\mathcal{P}$.

A binary relation $<$ on a set $\mathcal{P}$ is called a {\it strict partial
order} if it is anti-reflexive and {\it transitive} (i.e.\ $x<y$ and $y<z$ imply
$x<z$ for all $x,y,z\in\mathcal P$).
It is easy to see that $<$ is a strict partial order iff its inclusive
counterpart $\le$ is a {\it partial order}, that is, is reflexive,
{\it anti-symmetric} (i.e.\ $x\le y$ and $y\le x$ imply $x=y$ for all
$x,y\in\mathcal{P}$) and transitive.
A set endowed with a strict or non-strict partial order is called a {\it poset}.

\subsection{Preposets and acyclic relations}
By a {\it preposet} we mean a set $\mathcal{P}$ endowed with a {\it strictly acyclic}
relation, in the sense that there exists no sequence $x_0,\dots,x_n\in\mathcal{P}$,
with $n$ being a nonnegative integer, such that
$x_0\prec x_1\prec\dots\prec x_n\prec x_0$.
In particular, a strictly acyclic relation is anti-reflexive.
We also write $x\prec y$ as $y\succ x$.

The inclusive counterpart $\preceq$ of a strictly acyclic relation $\prec$ is
characterized by being reflexive and {\it acyclic} in the sense that if
$x_0,\dots,x_n\in\mathcal{P}$, with $n$ being a nonnegative integer, satisfy
$x_0\preceq x_1\preceq\dots\preceq x_n\preceq x_0$, then $x_0=\dots=x_n$.
In particular, $\preceq$ is anti-symmetric.
By the above, a preposet may be equivalently viewed as a set $\mathcal{P}$
endowed with a binary relation $\preceq$ that is reflexive and acyclic.

\subsection{Transitive closure and covering relation}
The {\it transitive closure} of a binary relation $\prec$ on a set $\mathcal{P}$
is the relation $\prec\!\prec$ on $\mathcal{P}$ defined by $x\prec\!\prec y$ iff
there exist $z_1,\dots,z_n\in\mathcal{P}$ for some nonnegative integer $n$ such
that $x\prec z_1\prec\dots\prec z_n\prec y$; it is clearly transitive.

It is easy to see that a binary relation is (strictly) acyclic iff its transitive
closure is a (strict) partial order.
Thus every poset is a preposet, and for every preposet $P=(\mathcal{P},\prec)$,
its transitive closure $\left<P\right>\bydef(\mathcal{P},\prec\!\prec)$ is a poset.
So one may view preposets not just as a generalization of posets, capturing the notion of possibly 
non-transitive subordination (``the vassal of my vassal is not necessarily my vassal''), 
but more specifically as posets endowed with an additional structure.

Given a poset $P=(\mathcal P,\le)$, one defines the {\it covering} relation $\prec$
on $P$ by $p\prec q$ if $p<q$ and there exists no $r\in\mathcal P$ with $p<r<q$.
Clearly, $\prec$ is the minimal strictly acyclic relation whose transitive closure
is $<$.

\subsection{Monotone maps}
Let $P=(\mathcal{P},\le)$ and $Q=(\mathcal{Q},\preceq)$ be [pre]posets.
An order preserving or {\it monotone} map between them is a map
$f\:\mathcal{P}\to\mathcal{Q}$ such that $v\le w$ implies $f(v)\preceq f(w)$
for all $v,w\in\mathcal{P}$.
It is called a {\it monotone embedding} if the converse implication holds as well.
Every monotone embedding is obviously injective, but not every injective
monotone map is a monotone embedding.
An {\it isomorphism} of [pre]posets is a monotone bijection whose inverse is
monotone, or equivalently a surjective monotone embedding.

We say that $Q$ is a {\it sub[pre]poset} of $P$, and write $Q\incl P$, if $\mathcal{Q}$
is a subset of $\mathcal{P}$ and the inclusion $\mathcal{Q}\emb\mathcal{P}$ is
an embedding of $Q$ into $P$.

The {\it dual} of a [pre]poset $P=(\mathcal{P},\le)$ is the [pre]poset
$P^*\bydef(\mathcal{P},\ge)$.

From now on we often do not distinguish between a [pre]poset
$P=(\mathcal{P},\preceq)$ and its underlying set $\mathcal{P}$ (by an abuse
of notation).

\subsection{Unary operations: $C$, $C^*$, $\partial$, $\partial^*$}
Let $P$ be a [pre]poset.
The {\it cone} $CP$ over $P$ is obtained by adjoining to $P$ an additional element,
denoted $\hat1$ (or $+\infty$), which is set to be greater than every element of $P$.
The {\it dual cone} $C^*P\bydef(C(P^*))^*$ is obtained by adjoining to $P$ an
additional element, denoted $\hat0$ (or $-\infty$), which is set to be less than
every element of $P$.

If $P$ is a [pre]poset, its {\it boundary} $\partial P$ is the sub[pre]poset
of $P$ consisting of all $p\in P$ such that $p\le q$ for some $q\in P$ such that
there exists precisely one element $r\in P$ satisfying $r>q$.
The {\it coboundary} $\partial^*P=(\partial (P^*))^*$.
We will later see that $\partial$ is related to the boundary of a manifold;
on the other hand, $\partial^*$ can be seen to be related to the coboundary of
a cochain (see \S\ref{cochain coboundary}).
Note that $\partial (CP)=P=\partial^*(C^*P)$.

\begin{example}[$2^S$ and $\Delta^S$]\label{2.1} Let $S$ be a set (possibly
infinite).
The relation of inclusion on the set $\text{\it 2}^S$ of all subsets of $S$ is
a partial order.
The resulting {\it subset poset} $2^S=(\text{\it 2}^S,\incl)$ is isomorphic to its
own dual (by taking the complement).
The poset $\Delta^S$ of all nonempty subsets of $S$ will be called a
{\it (combinatorial) simplex} or the {\it $S$-simplex}, or the
{\it $n$-simplex} (notation: $\Delta^n$) in the case where
$S$ is $[n+1]\bydef\{0,1,\dots,n\}$.
If $T\incl S$ is non-empty, $\Delta^T$ is called a {\it face} of $\Delta^S$.
Faces that are $0$-simplices (i.e.\ singletons) are also called {\it vertices},
and faces that are $1$-simplices are also called {\it edges}.
Note that $\partial\Delta^S=\partial(\partial^*2^S)$ is isomorphic to its own
dual (compare Example \ref{polytope-duality}).
\end{example}

\subsection{Cones}
Let $P$ be a [pre]poset.
The {\it cone} $\fll p\flr$ (resp.\ the {\it dual cone} $\cel p\cer$) of
a $p\in P$ is the sub[pre]poset of $P$ consisting of all $q\in P$ satisfying
$q\le p$ (resp.\ $q\ge p$).
We may also write $\fll p\flr_P$ and $\cel p\cer^P$ to emphasize the
[pre]poset $P$.

The definitions of $\fll p\flr$ and $\cel p\cer$ are in agreement with
the previously defined cone and dual cone over a poset; namely, $\fll p\flr$ is
the cone over $\partial\fll p\flr$, and $\cel p\cer$ is
the dual cone over $\partial^*\cel p\cer$.
The duality is expressed by $\cel p\cer=\fll p^*\flr^*$.
The notation%
\footnote{These are typeset in \TeX\ with the usual {\tt \symbol{92}lfloor},
{\tt \symbol{92}rfloor}, {\tt \symbol{92}lceil} and {\tt \symbol{92}rceil}
delimiters, downscaled using {\tt \symbol{92}scriptstyle} (in regular text)
and {\tt \symbol{92}scriptscriptstyle} (in sub- and superscripts), or their
context-sensitive combination formed with the aid of {\tt \symbol{92}mathchoice}.
(The {\tt \symbol{92}mathsmaller} command of the {\tt relsize} package produces
similar results except in displayed equations.)}
$\cel q\cer$ and $\fll q\flr$ can be thought of as just a concise form of
the interval notation $[p,+\infty)$ and $(-\infty,p]$, where the square brackets
undergo a counterclockwise $90^\circ$ rotation.

\begin{remark}
If $\mathcal{P}$ is an affine simplicial complex and $\preceq$ is
the inclusion relation (see Example \ref{2.3}), the cone of an affine simplex
$\sigma\in\mathcal{P}$
is the face poset $F_\sigma$, viewed as the poset of all simplices of
the subcomplex of $\mathcal{P}$ triangulating $\sigma$ (in fact, this subcomplex
does happen to be a ``cone'' in the terminology of Rourke--Sanderson
\cite[2.8(7)]{RS}); whereas the dual cone of $\sigma$ is isomorphic to what
is known as the ``dual cone'' of $\sigma$ in PL topology (see \cite[2.27(6)]{RS}).
\end{remark}

\subsection{Cone complexes} By a {\it cone complex}
we mean a countable poset where every cone is finite.
A {\it cone precomplex} is a preposet whose transitive closure is a cone
complex.
A cone [pre]complex $P$ such that the dual [pre]poset $P^*$ is also a cone
[pre]complex is called {\it locally finite}.

\begin{remark}
Cone precomplexes other than cone complexes arise in practice as triple deleted
prejoins of cone complexes and as mapping cylinders of non-closed monotone maps
between cone complexes; the non-closed monotone maps
in turn arise in practice as diagonal maps $P\to P\x P$ and more importantly as
bonding maps between nerves of coverings.
\end{remark}

\begin{lemma}\label{2.6} (a) Every preposet admits a monotonous injection
into a simplex.

(b) A preposet is a poset iff it is isomorphic to a subposet of a simplex.
\end{lemma}

\begin{proof}[Proof. (b)] Since every simplex is itself a poset, every preposet
embedded into a simplex is a poset.
Conversely, every poset $P=(\mathcal{P},\ge)$ is isomorphic to the poset of cones
of $P$ ordered by inclusion of their underlying sets, which is a subposet
of $\Delta^\mathcal{P}$.
\end{proof}

\begin{proof}[(a)] This follows from (b) since the inclusion of every preposet
in its transitive closure is monotonous.
\end{proof}

\subsection{Closed and open subposets}
Let $P$ be a [pre]poset.
A sub[pre]poset $Q$ of $P$ is called {\it closed} (resp.\ {\it open}) if
$p\le q$ (resp.\ $p\ge q$), where $q\in Q$, implies $p\in Q$.
Note that if $Q$ is a closed subpreposet of $P$, then $\left<Q\right>$ is a
closed subposet of $\left<P\right>$.

The {\it closure} (resp.\ {\it open hull}) of a sub[pre]poset $Q$ of $P$ is
the sub[pre]poset $\fll Q\flr$ (resp.\ $\cel Q\cer$) of $P$ consisting of
all $p\in P$ such that $p\le q$ (resp.\ $p\ge q$) for some $q\in Q$.
If $P$ is a poset, $\fll Q\flr$ and $\cel Q\cer$ coincide with the smallest
closed subposet and the smallest open subposet containing $Q$.

The collection of all open subposets of a poset $Q=(\mathcal Q,\le)$ determines
a topology on $\mathcal Q$, known as the Alexandroff or the right order topology.
It goes back to P. S. Alexandroff \cite{Al} that the following are equivalent
for a $T_0$ topological space $X$:

(i) Arbitrary intersections of open sets are open in $X$;

(ii) Every point has a smallest neighborhood in $X$;

(iii) $X$ is homeomorphic to a poset endowed with the Alexandroff topology.

\subsection{Closed and open maps} \label{closed-open-maps}
It is easy to see that the following are equivalent for a map $f\:P\to Q$
between posets:
\begin{roster}
\item $f$ is monotone;
\item $f$ continuous with respect to the Alexandroff topologies;
\item $f(\fll p\flr)\subset\fll f(p)\flr$ for each $p\in P$;
\item $f(\cel p\cer)\subset\cel f(p)\cer$ for each $p\in P$.
\end{roster}
For instance, to see that (ii) implies (i), we note that if $p'>p$, then $p'$
lies in every open subposet of $P$ containing $p$; whereas if $f(p)\nless f(p')$,
then $f(p')$ does not lie in the open subposet $\cel f(p)\cer$ of $Q$.

In general topology, a continuous map is called {\it open} [resp.\ {\it closed}] if
it sends every open [closed] set to an open [closed] one.
It is now obvious that a monotone map $f\:P\to Q$ between posets is open [closed]
(with respect to the Alexandroff topologies) if and only if $f(\cel p\cer)=\cel f(p)\cer$
[resp.\ $f(\fll p\flr)=\fll f(p)\flr$] for every $p\in P$.

\begin{example}\label{diagonal}
An example of a non-closed monotone map is given by the diagonal embedding of
the $1$-simplex $\Delta^1$ into $\Delta^1\x\Delta^1$.
\end{example}

\section{Suprema, atoms and simplices}

\subsection{Conditionally complete poset}
A poset $P$ is {\it conditionally complete} if every non-empty
$Q\incl P$ that has an upper bound in $P$ (i.e.\ a $p\in P$ such that
$Q\incl\fll p\flr$) also has a least upper bound in $P$ (i.e.\ an upper bound
$p\in P$ such that $\cel p\cer$ contains all upper bounds of $Q$ in $P$).

\begin{remark}
Conditionally complete posets should not be confused with bounded complete
posets (also known as consistently or coherently complete posets), where
every (not necessarily non-empty!) subset that has an upper bound has a
least upper bound.
It is easy to see that bounded complete posets are precisely those
conditionally complete posets that have a least element; and complete
posets are precisely those conditionally complete posets that have
a least element and a greatest element.
\end{remark}

The following lemma is well-known, cf.\ \cite{Bi}:

\begin{lemma}\label{2.7}
A poset $P$ is conditionally complete if and only if every non-empty subset of $P$
that has a lower bound in $P$ also has a greatest lower bound in $P$.
\end{lemma}

\begin{proof}[Proof]
By symmetry, it suffices to prove the `only if' assertion.
If $P$ is conditionally complete and a subset $Q\incl P$ has a lower bound in $P$, then the set $L$
of all lower bounds of $Q$ in $P$ is nonempty and has an upper bound in $P$
(specifically, any element of $Q$ will do).
Then there exists the greatest lower bound $u$ of $L$ in $P$, that is,
$\fll u\flr$ contains $L$ and $\cel u\cer$ contains all upper bounds of $L$,
in particular, all of $Q$.
By definition, $u$ is the greatest lower bound of $Q$.
\end{proof}

\begin{corollary}\label{2.8} (a) Every simplex is a conditionally complete poset.

(b) Every closed subposet of a conditionally complete poset is conditionally complete.
\end{corollary}

\begin{proof}[Proof. (a)] The greatest lower bound of a set of $A$ of nonempty
subsets $S_\alpha\incl S$ is their intersection $\bigcap_\alpha S_\alpha$, if it is
nonempty; else $A$ has no lower bounds.
\end{proof}

\begin{proof}[(b)] Let $P$ be a conditionally complete poset and $Q$ its closed
subposet.
If the greatest lower bound of a subset $S\incl Q$ exists in $P$, then it belongs
to $Q$, since $Q$ contains all its lower bounds.
\end{proof}

\subsection{Full subposet}
A subposet $Q$ of a poset $P$ is called {\it full} in $P$, if every cone
of $P$ meets $Q$ in a cone of $Q$ or in the empty set.
When $Q$ is a closed subposet, this is McCrory's definition
(see also \cite[Lemma 2.6]{Mc}).
When $P$ is a simplicial complex, this is the usual definition of a full
subcomplex, cf.\ \cite{RS}.

It is easy to see that open subposets are full.

\begin{lemma} \label{full in CCP}
A full subposet of a conditionally complete poset is conditionally complete.
\end{lemma}

\begin{proof}
Let $F$ be a full closed subposet of a conditionally complete poset $P$.
Given a subset $S\subset F$, the set $T$ of all lower bounds of
$S$ in $P$ is a cone since $P$ is conditionally complete.
Since $F$ is a subposet of $P$, the set of all lower bounds of $S$
in $F$ equals $T\cap F$, which is a cone since $F$ is full in $P$.
\end{proof}

\begin{lemma} \label{CCP amalgam}
Let $P$ be a poset that is a union of closed subposets $Q$ and $R$ such that
$Q\cap R$ is full in $Q$ and in $R$.
If $Q$ and $R$ are conditionally complete, then so is $P$.
\end{lemma}

\begin{proof}
Suppose that $S$ is a nonempty subset of $P$, and $p\in P$ is
an upper bound of $S$.
By symmetry we may assume that $p\in Q$.
Since $Q$ is closed, $S\subset Q$.
Since $Q$ is conditionally complete, $S$ has a least upper bound $q$ in $Q$.
If $r\notin Q$ is an upper bound of $S$ in $P$, then $r\in R$.
In this case $S\subset Q\cap R$.
Since $Q\cap R$ is full in $Q$, we get that $S$ has an upper bound $s\le r$
in $Q\cap R$.
Since $q$ is the least upper bound of $S$ in $Q$, we have $q\le s$.
Hence $q\le r$, which shows that $q$ is a least upper bound of $S$ in $P$.
\end{proof}

\begin{lemma} \label{full in open} Let $Q$ be a full subposet of a poset $P$.

(a) A subposet $R$ of $Q$ is full in $Q$ if and only if it is full in $P$.

(b) If $U$ is an open subposet of $P$, then $Q\cap U$ is full in $U$.
\end{lemma}

\begin{proof}[Proof. (a)]
The ``only if'' direction is obvious.
Conversely, let $q\in Q$.
Then $\fll q\flr_Q=\fll q\flr_P\cap Q$, so
$\fll q\flr_Q\cap R=\fll q\flr_P\cap R$.
\end{proof}

\begin{proof}[(b)]
Since $U$ is open, $U\cap Q$ is open in $Q$.
Then $U\cap Q$ is full in $Q$.
Since $Q$ is full in $P$, by the ``only if'' in (a), $U\cap Q$ is full in $P$.
On the other hand, since $U$ is open in $P$, it is full in $P$.
Then by the ``if'' in (a), $U\cap Q$ is full in $U$.
\end{proof}

\begin{lemma} If $Q$ is a full subposet of a poset $P$, then $\cel Q\cer$
strong deformation retracts onto $Q$ via
a monotone map $r\:\cel Q\cer\to Q$, homotopic to the identity by
a monotone homotopy $h\:\cel Q\cer\x[2]\to\cel Q\cer$ that extends the projection
$Q\x[2]\to Q$.
\end{lemma}

Here $[2]=\{1,2\}$ with the usual order.

\begin{proof} Define $r(p)=q$, where $\fll p\flr\cap Q=\fll q\flr$.
Then $r(q)=q$ for each $q\in Q$, and $r(p)\le p$ for each $p\in\cel Q\cer$,
so $r\cup\id_{\cel Q\cer}$ is the desired homotopy $h$.
\end{proof}

\subsection{Atoms}
An element $\sigma$ of a poset $P$ is called an {\it atom} of $P$, if
$\fll\sigma\flr=\{\sigma\}$.
The set of all atoms of $P$ will be denoted $A(P)$.
A poset $P$ is called {\it atomic}, if every its element is the least upper bound
of some subset of $A(P)$.%
\footnote{In the literature on lattice theory such posets are called ``atomistic'',
whereas ``atomic'' has a different meaning.
Our usage of the term can be found e.g.\ in \cite{Ai}.}
It is easy to see that $A(\fll\sigma\flr)=A(P)\cap\fll\sigma\flr$.
Hence every element $\sigma$ of an atomic poset is the least upper bound of
$A(\fll\sigma\flr)$.

In the case of atomic posets Lemma \ref{2.6} admits an amplification:

\begin{lemma}\label{2.9} If $P$ is an atomic poset, then the formula
$\sigma\mapsto A(\fll\sigma\flr)$ defines an embedding of $P$ into
$\Delta^{A(P)}$.
\end{lemma}

\begin{proof}[Proof] It is easy to see that the formula defines a monotone map
$f\:P\to 2^{A(P)}$.
Since each $\sigma\in P$ is the least upper bound of $A(\fll\sigma\flr)$, the
latter is nonempty (whence the image of $f$ is in $\Delta^{A(P)}$).
If $A(\fll\tau\flr)\incl A(\fll\sigma\flr)$, then the least upper bound $\sigma$ of
$A(\fll\sigma\flr)$ is an upper bound of $A(\fll\tau\flr)$.
Hence its least upper bound $\tau$ satisfies $\tau\le\sigma$ (whence $f$ is
an embedding).
\end{proof}

\begin{lemma}\label{atomic CCP}
An atomic poset $P$ is conditionally complete if and only if every non-empty
$R\subset A(P)$ that has an upper bound in $P$ has a least upper bound in $P$.
\end{lemma}

\begin{proof}
Only the ``if'' direction needs a proof.
Suppose we are given an $S\subset P$ that has an upper bound in $P$.
Then so does $R\bydef A(\fll S\flr)$.
Hence by our assumption $R$ has a least upper bound $\rho$.
Since $P$ is atomic, every $p\in S$ is the least upper bound of
$A(\fll p\flr)=A(\fll S\flr)\cap\fll p\flr$.
Since $\rho$ is an upper bound of $R$, it is an upper bound of its subset
$A(\fll p\flr)$ for each $p\in S$.
However $p$ is the least upper bound of the same set, so $\rho$ is an
upper bound $p$.
Thus $\rho$ is an upper bound of $S$.
If $\rho'$ is another upper bound of $S$, then $\rho'$ is an upper bound
of $A(\fll S\flr)$.
But $\rho$ is the least upper bound of the same set, so $\rho'\ge\rho$.
Thus $\rho$ is the least upper bound of $S$.
\end{proof}

\subsection{Simplicial complex}
A {\it simplicial} poset is a conditionally complete poset where every cone is
isomorphic to a simplex.
(Compare \cite{BBC}.)
A simplicial cone complex is abbreviated to a {\it simplicial complex}.
The cones of a simplicial complex $K$ are thus called its {\it simplices}.
Clearly, every simplicial complex is atomic.

\begin{theorem}\label{2.10} A poset is simplicial iff it is isomorphic to
a subcomplex of a simplex.
\end{theorem}

In particular, it follows that a cone complex is a simplicial complex iff
it is isomorphic to a subcomplex of a simplex.

\begin{proof}[Proof] The `if' assertion is straightforward.
Every cone of a simplex is a simplex.
A subcomplex of a simplex is a conditionally complete poset by Corollary \ref{2.8}.

Conversely, let $K$ be a simplicial complex.
Let us consider the embedding $f\:K\to\Delta^{A(K)}$ constructed in Lemma
\ref{2.9}.
If $T\incl A(\fll\sigma\flr)$, then the least upper bound $\sigma$ of
$A(\fll\sigma\flr)$ is an upper bound of $T$, hence its least upper bound $\tau$
exists and satisfies $\tau\le\sigma$.
Therefore $A(\fll\tau\flr)=A(K)\cap\fll\tau\flr$ contains $T$, moreover $\tau$ is
the least upper bound of $A(\fll\tau\flr)$, as well as of $T$.
Since $\fll\tau\flr$ is isomorphic to a simplex, this implies $T=A(\fll\tau\flr)$.
So $f(\tau)=T$, whence the image of $f$ is a subcomplex of $\Delta^{A(K)}$.
\end{proof}

\subsection{Barycentric subdivision}
Let $P=(\mathcal{P},\preceq)$ be a preposet.
A {\it chain} in $P$ is a $\mathcal{Q}\incl\mathcal{P}$ that is a totally ordered
by $\prec$ (that is, for each $p,q\in\mathcal{Q}$ either $p\prec q$ or
$p\succeq q$; note that this already implies that $\preceq$ is transitive on
$\mathcal{Q}$).
The poset $P^\flat$ of all nonempty finite chains of $P$ ordered by inclusion is
a subcomplex of $\Delta^{\mathcal{P}}$ and so a simplicial poset (a simplicial
complex if $P$ is countable); it is called the {\it barycentric subdivision}
of $P$.

Clearly, if $Q$ is a closed subposet of a poset $P$, then $Q^\flat$ is a full subcomplex
of $P^\flat$.

\subsection{Flag complex}
A {\it flag complex} is a simplicial complex $K$ such that every subcomplex of
$K$ that is isomorphic to the boundary of a simplex of dimension $>1$ is
the boundary of some simplex of $K$.
Obviously, every full subcomplex of a flag complex is a flag complex.

It is easy to see that the barycentric subdivision of every cone precomplex
is a flag complex; and that if $Q$ is embedded in $P$, then $Q^\flat$ is a full
subcomplex of $P^\flat$.
Using these facts, it is easy to prove

\begin{proposition}\label{2.13} Let $K$ be a simplicial complex and $L$
a subcomplex of $K^\flat$.

(a) $L$ is a flag complex iff it is the barycentric subdivision of the
image of some cone precomplex under a monotone injective map into $K$.

(b) $L$ is a full subcomplex of $K^\flat$ iff it is the barycentric subdivision
of some cone complex embedded in $K$.
\end{proposition}

\subsection{Simplicial maps}
A closed map between simplicial posets is called {\it simplicial}.
Every map of sets $f\:S\to T$ induces a simplicial map
$\Delta^f\:\Delta^S\to\Delta^T$.
If $K$ is a subcomplex of $\Delta^S$ and $L$ is a subcomplex of $\Delta^T$,
it is easy to see that every simplicial map $K\to L$ is a restriction of
$\Delta^f$ for some $f\:S\to T$.

If $f\:P\to Q$ is a monotone map between preposets, it sends every nonempty
finite chain of $P$ into a nonempty finite chain of $Q$.
The resulting map $f^\flat:P^\flat\to Q^\flat$ is the restriction to
a subcomplex of the simplicial map $\Delta^f\:\Delta^P\to\Delta^Q$, so
it is simplicial.

\subsection{Full map}
Let us call a monotone map of posets $f\:P\to Q$ {\it full}, if
$f^{-1}(q)$ is full in $P$ for each $q\in Q$.

\begin{lemma}\label{point-inverses full}
Simplicial maps are full.
\end{lemma}

\begin{proof}
Let $f\:K\to L$ be a simplicial map between simplicial complexes.
Pick some $\sigma\in L$, and let $F=f^{-1}(\sigma)$.
If $\tau\in K$ is such that $\fll\tau\flr\cap F\ne\emptyset$,
then $f(\tau)\ge\sigma$.
If $\tau\in F$, then $\fll\tau\flr\cap F$ is the cone $\fll\tau\flr_F$ since
$F$ is a subposet of $K$.
So without loss of generality we may assume that $f(\tau)>\sigma$.
Then $\fll\tau\flr=\fll\mu\flr*\fll\nu\flr$, where $f(\mu)=\sigma$ and
$f(\nu)\cap\sigma=\emptyset$.
Hence $\fll\tau\flr_K\cap F=\fll\mu\flr_F$.
\end{proof}

\section{Basic operations}\label{operations}

Let $P=(\mathcal{P},\le)$ and $Q=(\mathcal{Q},\le)$ be preposets.

The {\it prejoin} $P+Q$ is the preposet $(\mathcal{P}\sqcup\mathcal{Q},\preceq)$,
where $p\preceq q$ iff either $p\le q$ and both $p,q\in P$; or $p\le q$ and both
$p,q\in Q$; or $p\in P$ and $q\in Q$.
Clearly, the prejoin of two posets is a poset.
Note that $CP\simeq P+pt$ and $C^*P\simeq pt+P$.

The {\it product} $P\x Q$ is the preposet $(\mathcal{P}\x\mathcal{Q},\preceq)$,
where $(p,q)\preceq (p',q')$ iff $p\le p'$ and $q\le q'$.
It is easy to see that $2^S\x 2^T\simeq 2^{S\sqcup T}$ naturally in $S$ and $T$.

The {\it join} $P*Q\bydef\partial^*(C^*P\x C^*Q)$ is obtained from $(C^*P)\x (C^*Q)$
by removing the bottom element $(\hat0,\hat0)$.
Thus $C^*(P*Q)\simeq C^*P\x C^*Q$, whereas $P*Q$ itself is the union
$C^*P\x Q\cup P\x C^*Q$ along their common part $P\x Q$.

From the above, $\Delta^S*\Delta^T\simeq\Delta^{S\sqcup T}$
naturally in $S$ and $T$.
It follows that the join of simplicial complexes
$K\incl\Delta^S$ and $L\incl\Delta^T$ is isomorphic to the simplicial complex
$\{\sigma\cup\tau\incl S\sqcup T\mid\sigma\in K\cup\{\emptyset\},
\tau\in L\cup\{\emptyset\},\,\sigma\cup\tau\ne\emptyset\}\incl\Delta^{S\sqcup T}$.

The join and the prejoin are related via barycentric subdivision:
$(P+Q)^\flat\simeq P^\flat*Q^\flat$.
Indeed, a nonempty finite chain in $P+Q$ consists of a finite chain in $P$
and a finite chain in $Q$, at least one of which is nonempty.
Note that in contrast to prejoin, join is commutative: $P*Q\simeq Q*P$.
Prejoin is associative; in particular, $C(C^*P)\simeq C^*(CP)$.

\begin{remark} In the case where $P$ and $Q$ are finite simplicial complexes,
the above mentioned isomorphism
$$P*Q\simeq C^*P\x Q\underset{P\x Q}{\cup} P\x C^*Q$$
can be regarded as a combinatorial form of the well-known (cf.\ \cite[4.3.20]{St})
homeomorphism $$X*Y\cong (pt*X)\x Y\underset{X\x Y}{\cup} X\x (pt*Y),$$
where $X=|P|$ and $Y=|Q|$.
However it does not quite fit in the familiar simplicial realm even in this case,
for $C^*P$ and $P\x Q$ are no longer simplicial complexes.
\end{remark}

If $P$ is a conditionally complete poset, then obviously $CP$ and $C^*P$
are conditionally complete.
If $P$ and $Q$ are conditionally complete, then obviously so is $P\x Q$.
In addition, $P*Q$ is conditionally complete since it is an open subposet of
$C^*P\x C^*Q$.

\subsection{Van Kampen duality}
If $P$ and $Q$ are preposets, their {\it cojoin} $P\cojoin Q=(P^**Q^*)^*$.
By dualizing $C^*X\x C^*Y\simeq C^*(X*Y)$, where $X=P^*$ and $Y=Q^*$, one obtains
$$CP\x CQ\simeq C(P\cojoin Q).$$
In the case where $P$ and $Q$ are posets, this formula was known already to
E. R. van Kampen \cite{vK}, cf.\ \cite[Proposition 1.2]{Mc}.
It implies, for instance, that the boundary of the $n$-cube is dual (as a poset)
to the boundary of the $n$-cross-polytope (compare Example \ref{polytope-duality}):
$$\join_{i=1}^n \partial I^1\simeq(\partial I^n)^*.$$

\begin{lemma}\label{cojoin formula}
$CP*CQ\simeq C\left(C^*P\cojoin Q\underset{P\cojoin Q}\cup P\cojoin C^*Q\right)$.
\end{lemma}

\begin{proof} Using the van Kampen duality, we get
$$CP*CQ\simeq C^*CP\x CQ\underset{CP\x CQ}\cup CP\x C^*CQ\simeq
C(C^*P\cojoin Q)\underset{C(P\cojoin Q)}\cup C(P\cojoin C^*Q).$$
\end{proof}

\subsection{Infinite product and join}
Similarly to the case of two factors one defines the {\it product}
$\prod\limits_{\lambda\in\Lambda} P_\lambda$ of an arbitrary family of
preposets $P_\lambda$.
Similarly, their {\it join} is defined by
$\join\limits_{\lambda\in\Lambda} P_\lambda=
\partial^*\prod\limits_{\lambda\in\Lambda} C^*P_\lambda$.
Note the {\it van Kampen duality}
$\join\limits_{\lambda\in\Lambda} P_\lambda\simeq
(\partial\prod\limits_{\lambda\in\Lambda} CP_\lambda^*)^*$.

\subsection{Star and link}
If $P$ is a preposet and $\sigma\in P$, we define the {\it star}
$\st(\sigma,P)=\bigfl\cel\sigma\cer\bigfr$ and the {\it link}
$\lk(\sigma,P)=\partial^*\cel\sigma\cer$.
Thus if $P$ is a poset, $\st(\sigma,P)$ is a closed subposet of $P$ and
$\lk(\sigma,P)$ is an open subposet of $P$.

If $K$ is a simplicial complex, and $\sigma\in K$, then
$\lk(\sigma,K)=\partial^*\cel\sigma\cer$ is isomorphic to the classical link
$\Lk(\sigma,K):=\Fll\cel\sigma\cer\Flr\but\Cel\fll\sigma\flr\Cer$, which is
a subcomplex of $K$; an isomorphism is given by $\sigma\sqcup\tau\mapsto\tau$.
It follows that $\st(\sigma,K)\simeq\fll\sigma\flr*\lk(\sigma,K)$ for every
{\it simplicial} complex $K$.
See \cite[p.\ 12]{Ha0} for a discussion of $\lk$ versus $\Lk$.

For an element $p$ of a poset $P$, it is easy to see that
$\lk(p,P)^\flat=\Lk(p_-,P^\flat)$ and $\lk(p^*,P^*)^\flat=\Lk(p_+,P^\flat)$
as subcomplexes of $P^\flat$ (not just up to isomorphism), where
$p_-$ is a maximal chain in $P$ with greatest element $p$ (i.e.\ a maximal simplex
of $\fll p\flr^\flat$), and $p_+$ is a maximal chain in $P$ with least element $p$
(i.e.\ a maximal simplex of $\cel p\cer^\flat$).

Given $\sigma\in P$ and $\tau\in Q$, clearly $\cel(\sigma,\tau)\cer^{P\x Q}\simeq
\cel\sigma_1\cer\x\cel\sigma_2\cer$; applying the coboundary, we obtain
$$\lk((\sigma,\tau),P\x Q)\simeq\lk(\sigma,P)*\lk(\tau,Q).$$

\begin{remark}
In the case where $P$ and $Q$ are finite simplicial complexes, the latter
isomorphism can be regarded as a combinatorial form of the well-known
(cf.\ \cite[4.3.21]{St}) homeomorphism $$\lk((x,y),X\x Y)\cong\lk(x,X)*\lk(y,Y),$$
where $|P|=X$ and $|Q|=Y$ are compact polyhedra.
However it does not quite fit in the familiar simplicial realm even in this case,
for $P\x Q$ is no longer a simplicial complex.
\end{remark}

\subsection{Cubosimplicial complex}
By a {\it cubosimplicial complex} we mean a conditionally complete cone complex
where every cone is isomorphic to a product of simplices.

\begin{lemma}\label{cubosimplicial}
Let $f\:K\to L$ be a simplicial map of simplicial complexes.
Then $f^{-1}(\sigma)$ is a cubosimplicial complex for each $\sigma\in L$.
\end{lemma}

\begin{proof}
By Lemma \ref{point-inverses full}, $F:=f^{-1}(\sigma)$ is full in $K$.
Hence by Lemma \ref{full in CCP}, $F$ is conditionally complete.
Let $\tau\in F$, and let $g$ be the restriction of
$f$ to the simplex $\fll\tau\flr_K$.
Since $\fll\tau\flr_F=\fll\tau\flr_K\cap F=g^{-1}(\sigma)$, it suffices to show that
$g^{-1}(\sigma)$ is isomorphic to a product of simplices.
Viewing $\sigma$ as the set of vertices of the simplex $\fll\sigma\flr$, let
$\Delta_\lambda=g^{-1}(\lambda)$ for each $\lambda\in\sigma$.
Then $g\:\fll\tau\flr=\underset{\lambda\in\sigma}\join\Delta_\lambda\to
\underset{\lambda\in\sigma}\join\{\lambda\}=\fll\sigma\flr$
is the join of the constant maps
$g|_{\Delta_\lambda}\:\Delta_\lambda\to\{\lambda\}$.
Each $g|_{\Delta_\lambda}$ is the restriction of
$G_\lambda\:C^*\Delta_\lambda\to C^*\{\lambda\}$.
Then $g$ is the restriction of their product
$$G\:C^*\fll\tau\flr=\prod_{\lambda\in\sigma}C^*\Delta_\lambda\to
\prod_{\lambda\in\sigma}C^*\{\lambda\}=C^*\fll\sigma\flr,$$
where the identifications come from the definition of join:
$C^*\underset\lambda\join P_\lambda=\prod\limits_\lambda C^*P_\lambda$.
Hence $g^{-1}(\sigma)=G^{-1}(\sigma)=\prod\limits_{\lambda\in\sigma}
(G_\lambda)^{-1}(\lambda)=\prod\limits_{\lambda\in\sigma}\Delta_\lambda$.
Thus $g^{-1}(\sigma)$ is a product of simplices.
\end{proof}

\subsection{Cubical map}
A monotone map $f\:P\to Q$ of posets is called {\it cubical} if the restriction of
$f$ to every cone of $P$ is isomorphic to the projection $CX\x CY\to CX$ for some
posets $X$ and $Y$.

The following lemma ensures that a monotone map of cubosimplicial complexes is
cubical if and only if its restriction to every cone is isomorphic to
the projection of a product of simplices onto its subproduct.

\begin{lemma}
If a product of simplices is isomorphic to $P\x Q$, then $P$ and $Q$
are products of simplices.
\end{lemma}

\begin{proof} We have $\lk(\hat 0,P^*\x Q^*)\simeq\lk(\hat 0,P^*)*\lk(\hat 0,Q^*)$.
On the other hand this is a join of boundaries of simplices $\partial\Delta^i$
(using that the boundary of a simplex is isomorphic to its dual), in particular
a simplicial complex.
If not all vertices of a $\partial\Delta^i$ are in the same factor of the join
$\lk(\hat 0,P^*)*\lk(\hat 0,Q^*)$, then $\partial\Delta^i$ is itself a nontrivial
join of simplicial subcomplexes.
This is a contradiction.
\end{proof}

\subsection{Hatcher maps} \label{hatcher-maps}
This is a combinatorial version of Hatcher's construction \cite[p.\ 105]{Hat},
\cite{Ste}.
Let $f\:P\to Q$ be a closed full map of posets.
If $q\in Q$, let us write $F_q=f^{-1}(q)$.
Given a $p\in P$ and a $q\in Q$ such that $q<f(p)$, we have
$\fll p\flr_P\cap F_q\ne\emptyset$ since $f$ is closed, and consequently
$\fll p\flr_P\cap F_q=\fll p_q\flr_{F_q}$ for some $p_q\in F_q$ since $f$ is full.
Given a pair of elements $r,q\in Q$ such that $r>q$, define a monotone map
$f_{rq}\:F_r\to F_q$ by $p\mapsto p_q$ if $F_r\ne\emptyset$, and in the only possible way if $F_r=\emptyset$.

Note that if $r>s>t$, then $f_{rt}$ equals the composition
$f^{-1}(r)\xr{f_{rs}}f^{-1}(s)\xr{f_{st}}f^{-1}(t)$.

It follows from the proof of Lemma \ref{cubosimplicial} that

\begin{lemma} \label{Hatcher is cubical}
Hatcher maps of a simplicial map are cubical.
\end{lemma}

We call a poset $P$ {\it nonsingular} if it contains no interval of
cardinality three.

\begin{lemma} \label{nonsingular Hatcher}
Let $f\:P\to Q$ be a closed full map of finite posets.
If $P$ is nonsingular, then the Hatcher maps of $f$ are closed.
\end{lemma}

\begin{proof}
Let us show that each $f_{rq}$ is closed.
Without loss of generality $r$ covers $q$.
Given a $p\in F_r$ and an $s\in\partial\fll p_q\flr_{F_q}$, we have $s<p_q<p$.
We need to show that $s=t_q$ for some $t\in F_q$.
Arguing by induction, we may assume that $p_q$ covers $s$.
If $p$ does not cover $p_q$, there exists a $t<p$, $t\in F_r$ (using that
$r$ covers $q$) such that $t_q=p_q$.
By considering such a minimal $t$ we may assume that $p$ covers $p_q$.
By the hypothesis there exists a $t\in P$, $t\ne p_q$, such that $s<t<p$.
Then $t$ is incomparable with $p_q$, whence $t\notin F_q$.
Since $r$ covers $q$, we conclude that $t\in F_r$.
Since $p>t>s$, we have $p_q\ge t_q\ge s$.
Since $t$ is incomparable with $p_q$, we have $t_q\ne p_q$.
Since $p_q$ covers $s$, this implies that $t_q=s$.
\end{proof}

\section{Intervals and cubes}\label{intervals and cubes}

\subsection{Canonical subdivision} If $P=(\mathcal P,\le)$ is a [pre]poset and
$a,b\in\mathcal P$ are such that $a\le b$, the {\it interval} $[a,b]$ is
the sub[pre]poset $\cel a\cer\cap\fll b\flr=\{c\in P\mid a\le c\le b\}$ of $P$.
If $P$ is a poset, we define the {\it canonical subdivision} $P^\#$ of $P$ to be
the poset of all intervals of $P$ ordered by inclusion.

In the general case some additional care is needed.
We say that an interval $[a,b]$ is {\it pre-included} in an interval $[c,d]$ and
write $[a,b]\Subset[c,d]$ if $a\in [c,d]$ and $b\in [c,d]$.
When $P$ is a poset, this is just the usual inclusion relation.
In general, $\Subset$ is a reflexive acyclic relation on the set of all intervals
of $P$ (the latter set is really just the relation $\le$ which officially is
a subset of $\mathcal P\x\mathcal P$).
We define the {\it canonical subdivision} $P^\#$ of $P$ to be the preposet of
all intervals of $P$ ordered by pre-inclusion.

We note that the canonical subdivision of every poset is an atomic poset.
If $P$ is a conditionally complete poset, then so is $P^\#$.
Clearly, $(P^*)^\#\simeq P^\#$ and $(P\x Q)^\#=P^\#\x Q^\#$ for
all preposets $P$ and $Q$.
It follows that $(C^*(P*Q))^\#=(C^*P)^\#\x (C^*Q)^\#$ and in particular,
$(P*Q)^\#=(C^*P)^\#\x Q^\#\cup P^\#\x (C^*Q)^\#$.

Beware that if $Q$ is a full subposet of $P$, then $Q^\#$ need not be full in $P^\#$.

\begin{remark}
In the case where $P$ is a poset, the operation of canonical
subdivision is known (under different names) in Topological Combinatorics
(see Babson, Billera and Chan \cite{BBC} and references there, and
\v Zivaljevi\'c \cite[Definition 7]{Zhi}), as well as in Order Theory
(see \cite{Lih} and references there).
Geometric versions of this construction, mostly restricted to the case where
$P$ is a simplicial complex, are also known in Algebraic Topology (see
\cite{FRS}), in Combinatorial Geometry (see \cite{SS}, \cite{BuP}) and in Geometric Group
Theory (see \cite{BB}).
\end{remark}

\begin{remark}
It is interesting to compare canonical subdivision with the edgewise
subdivision of simplicial sets \cite{BHM}, which takes posets to posets
(see \cite{EG}).
On the level of order complexes, both subdivisions introduce new
vertices precisely at the midpoints of all original edges.
Moreover, if a simplex of the order complex is identified (keeping
the order of the vertices) with the ``standard skew $n$-simplex''
$\{(x_1,\dots,x_n)\mid 0\le x_1\le\dots\le x_n\le 1\}\subset\R^n$,
then each of the two subdivisions cuts it into $2^n$ mutually congruent
simplices.
But they are not the same subdivision, already in
the $2$-dimensional case.
\end{remark}

\subsection{Cubical complexes} The poset $I^S\bydef(2^S)^\#$ is called
a {\it cube} or the {\it $S$-cube}; or the $n$-cube (notation: $I^n$)
if $S=[n]$.
Note that $$I^S=(C^*\Delta^S)^\#\simeq(C^*(\underset{S}{*}\Delta^0))^\#=
(\prod_S C^*\Delta^0)^\#=\prod_S(C^*\Delta^0)^\#=\prod_S I^1.$$
In particular, $I^{S\sqcup T}\simeq I^S\x I^T$ by an isomorphism natural
in $S$ and $T$.

A {\it cubical} poset is a conditionally complete poset where every cone is
isomorphic to a cube.
(Compare \cite{BBC}.)
A ``cubical cone complex'' is abbreviated to a {\it cubical complex}.
Cones of a cubical complex are thus called its {\it cubes}.
Closed subposets of a cubical complex will be termed its {\it subcomplexes}.
By the above, the product of two cubical complexes is a cubical complex.
Every cubical complex is atomic, since every cube is.

In contrast to Theorem \ref{2.10}, not every cubical complex is isomorphic to
a subcomplex of a cube.
For instance, the simplicial complex $\partial\Delta^2$, which also happens
to be a cubical complex, is not isomorphic to any subcomplex of any cube,
as it contains ``a cycle of odd length''.

\begin{lemma}\label{3.4} (a) If $K$ is a simplicial complex, then $(C^*K)^\#$ is
a cubical complex.

\smallskip
(b) If $Q$ is a cubical complex and $\sigma\in Q$, then
$\lk(\sigma,Q)$ is a simplicial complex and
$\st(\sigma,Q)\simeq \fll\sigma\flr\x (C^*\lk(\sigma,Q))^\#$.
\end{lemma}

\begin{proof}[Proof. (a)] By Theorem \ref{2.10}, $K$ is isomorphic to
a subcomplex of some simplex $\Delta^S$.
Then $(C^*K)^\#$ is isomorphic to a subcomplex of the cube
$(C^*\Delta^S)^\#\simeq I^S$.
\end{proof}

\begin{proof}[(b)] Given a $\tau\in\lk(\sigma,Q)$, we have
$(\fll\tau\flr,\fll\sigma\flr)\simeq(I^T,I^S)$ for
some sets $S\incl T$.
Then $\lk(\sigma,\fll\tau\flr)$ can be identified with
the poset, embedded in $I^T$ and consisting of all intervals strictly
containing $[\emptyset,S]$ --- that is, of all intervals $[\emptyset,S\cup R]$,
where $\emptyset\ne R\incl T\but S$.
Hence it is isomorphic to $\Delta^{T\but S}$.

We have $\st(\sigma,\fll\tau\flr)\simeq\fll\sigma\flr\x I^{T\but S}$,
whereas $I^{T\but S}\simeq(C^*\Delta^{T\but S})^\#\simeq
(C^*\lk(\sigma,\fll\tau\flr))^\#$.
The resulting isomorphism
$\st(\sigma,\fll\tau\flr)\simeq\fll\sigma\flr\x(C^*\lk(\sigma,\fll\tau\flr))^\#$
is natural in $\tau$, which implies the second assertion.

To complete the proof of the first assertion, we note that each cone of
$\lk(\sigma,Q)$ is of the form $\lk(\sigma,\fll\tau\flr)$, which in
turn has been shown to be isomorphic to a simplex.
Now $\lk(\sigma,Q)$ is an open subposet of the conditionally complete poset $Q$,
and hence itself
a conditionally complete poset.
\end{proof}

\begin{theorem}[{\cite{BBC} (see also \cite{D2}, \cite[5.22]{BH}, \cite[Fig.\ 2]{JS};
compare \cite{Ga})}]\label{3.5}
If $K$ is a finite simplicial complex, there exists a finite cubical complex $Q$
such that $\lk(v,Q)$ is isomorphic to $K$ for every vertex $v$ of $Q$.
\end{theorem}

The `mirroring' construction of Theorem \ref{3.5} has a complex analogue,
where the $(\Z/2)^n$ symmetry is replaced by an $(S^1)^n$ symmetry; it is known
as the `moment-angle complex' (see \cite{BuP}).
Both constructions are special cases of the `polyhedral product' \cite{Da}.

\begin{proof}[Proof] $K$ can be identified with a subcomplex of the $(n-1)$-simplex
$\Delta^{[n]}$ for some $n$.
Then $C^*K$ is identified with a closed subposet of $2^{[n]}$.
The `folding' monotone map $f\:I\to 2^{[1]}$ can be multiplied by itself to yield
a monotone map $F\:I^n\to 2^{[n]}$.
Let $Q=F^{-1}(C^*K)$.
Then $Q$ is a subcomplex of $I^n$ and contains the set $F^{-1}(\hat 0)$ of all
vertices of $I^n$.
For each vertex $v=[S,S]$ in this set, define $J_v\:2^{[n]}\to I^n$ by
$T\mapsto [T\but S,T\cup S]$.
Then $J_v$ is a monotone map, $J_v(2^{[n]})=\cel v\cer$, and the
composition $2^{[n]}\xr{J_v} I^n\xr{F} 2^{[n]}$ is the identity.
Hence $F|_{\cel v\cer}$ is an isomorphism.
Therefore $\lk(v,Q)\simeq\lk(\hat 0,C^*K)=K$.
\end{proof}

\begin{comment}
\begin{remark} Replacing the inclusion $K\incl\Delta^{[n]}$ by a simplicial map
$G\:K\to\Delta^{[n]}$ (thought of as a coloring of the vertices of $K$) and
considering the pullback of $F$ and $C^*G$ we obtain a special case of
the Davis complex (M. Davis, 1984) and the Davis--Januszkiewicz construction
\cite[Example 1.15]{DJ}.
A similar construction was devised much earlier by Williams \cite{W1}.
Note that if $K$ is a simplicial combinatorial $(n-2)$-sphere, then the canonical
subdivision of the pullback is a cubical combinatorial $(n-1)$-manifold
(as originally observed by Davis).
\end{remark}
\end{comment}

\subsection{Weak product and join}\label{weak join}
If we are given some basepoints $b_\lambda\in P_\lambda$, there is
the {\it pointed weak product} $\prod\limits_{\lambda\in\Lambda}^w P_\lambda$,
which is embedded in $\prod\limits_{\lambda\in\Lambda} P_\lambda$ and is
the union of $(\prod\limits_{\lambda\in\Phi} P_\lambda)\x
(\prod\limits_{\lambda\in\Lambda\but\Phi}\{b_\lambda\})$
over all finite $\Phi\incl\Lambda$.
It also has the basepoint $(b_\lambda)_{\lambda\in\Lambda}$.

The {\it weak join} $\join\limits^w_{\lambda\in\Lambda}P_\lambda$
is by definition $\partial^*(\prod\limits_{\lambda\in\Lambda}^w C^*P_\lambda)$,
where all the basepoints are taken at $\hat0$.
Thus the weak join is the union of all finite subjoins.
It can be identified with
$\bigcup\limits_{\lambda\in\Lambda} (P_\lambda\x
\prod\limits_{\kappa\in\Lambda\but\{\lambda\}}^w C^*P_\kappa)$.
The {\it weak van Kampen duality} reads
$\join\limits^w_{\lambda\in\Lambda}P_\lambda\simeq
(\partial\prod\limits_{\lambda\in\Lambda}^w CP_\lambda^*)^*$,
where all basepoints are taken at $\hat 1$.

The {\it weak $\Lambda$-simplex} $\Delta^\Lambda_w\bydef\join^w_{\Lambda}\Delta^0$
is identified with the poset $\partial^*2^\Lambda_w$ of all nonempty finite subsets
of $\Lambda$; it has no greatest element whenever $\Lambda$ is infinite.
Neither has the {\it weak $\Lambda$-cube}
$I^\Lambda_w\bydef(2^\Lambda_w)^\#\simeq\prod^w_\Lambda I^1$,
where all basepoints are taken at $[\hat 0,\hat 0]$.
In contrast, the {\it co-weak $\Lambda$-cube}
$I^\Lambda_c\bydef\prod^w_\Lambda I^1$,
where all basepoints are taken at $[\hat 0,\hat 1]$, has a greatest
element.
Moreover, by virtue of the weak van Kampen duality,
$\partial I^\Lambda_c\simeq (S^\Lambda_w)^*$, where the
{\it weak $\Lambda$-sphere} $S^\Lambda_w=\join^w_{\Lambda}\partial I^1$.
Somewhat reminiscent of the mirroring construction (in the proof of
Theorem \ref{3.5}), $S^\Lambda_w$ contains copies of $\Delta^\Lambda_w$, and
therefore $(I^\Lambda_c)^*$ contains copies of $2^\Lambda_w$.
In particular, $(I^\Lambda_c)^\#\simeq ((I^\Lambda_c)^*)^\#$ contains copies of
$(2^\Lambda_w)^\#=I^\Lambda_w$.

\subsection{Simple posets}\label{simple}
A {\it simple poset} is a conditionally complete poset $P$ such that for every
$\sigma\in P$ and every
$\tau\in\lk(\sigma,P)$, the poset $\lk(\sigma,\fll\tau\flr)$ is isomorphic to
a simplex.
This includes simplicial and cubical posets as special cases.

\begin{lemma}\label{3.6}
(a) If $K$ is a simplicial poset, then $C^*K$ is a simple poset.

(b) If $P$ is a simple poset, $\sigma\in P$, then $\lk(\sigma,P)$
is a simplicial poset.

(c) If $P$ is a conditionally complete poset, then it is a simple poset iff $P^\#$
is a cubical poset.

(d) If $P$ is a simple poset, then $P^*$ is a simple poset.

(e) The face poset of a simple polytope is a simple poset.
\end{lemma}

A version of the `only if' implication in (c) is found in \cite{BBC}.

\begin{proof}[Proof. (a)] Since $K$ is conditionally complete, then so is $C^*K$.
If $\tau\in C^*K$, then either $\tau\in K$ or $\tau=\hat0$.
If further $\sigma\in\lk(\tau,C^*K)$, then
$\lk(\tau,\fll\sigma\flr_{C^*K})$ is isomorphic to the simplex
$\lk(\tau,\fll\sigma\flr_K)$ in the first case and to the simplex
$\fll\sigma\flr$ in the second case.
\end{proof}

\begin{proof}[(b)] Similar to the proof of the first assertion of Lemma
\ref{3.4}(b).
\end{proof}

\begin{proof}[(c)] Given a $\sigma\in P$, by (b), $\cel\sigma\cer\simeq C^*K$
for some simplicial poset $K$.
Hence $\cel\sigma\cer^\#$ is a cubical poset by Lemma \ref{3.4}(a).
Thus cones of $P^\#$ are cubes, and each
$\cel\sigma\cer^\#$ is conditionally complete.
Now suppose that $P^\#$ contains a lower bound for a collection of intervals
$[\sigma_\lambda,\tau_\lambda]$.
Since $P$ is conditionally complete,
$\bigcap_\lambda\cel\sigma_\lambda\cer=\cel\sigma\cer$ for some
$\sigma\in P$.
Hence
$\bigcap_\lambda[\sigma_\lambda,\tau_\lambda]=\bigcap[\sigma,\tau_\lambda]$,
which is an interval since $\cel\sigma\cer^\#$ is conditionally complete.
Thus $P^\#$ is conditionally complete.

Conversely, suppose that $P^\#$ is a cubical poset.
Since $P$ is a poset, $\cel\sigma\cer^\#$ is a subcomplex of $P^\#$ for each
$\sigma\in P$, and in particular a cubical poset.
Then by Lemma \ref{3.4}(b), each $\lk([\sigma,\sigma],\cel\sigma\cer^\#)$ is
a simplicial poset.
On the other hand, it is isomorphic to $\lk(\sigma,P)$ since $P$ is a poset.
Thus $P$ is simple.
\end{proof}

\begin{proof}[(d)] Both the hypothesis and the conclusion are equivalent to saying
that $P$ is conditionally complete and every interval $[\sigma,\tau]\in P$ considered as
a subposet is isomorphic to the dual cone over a simplex
($\Leftrightarrow$ the cone over the dual of a simplex).
\end{proof}

\begin{proof}[(e)]
If $P$ is a convex polytope, its face poset
$F_P=C(\partial F_P)=(C^*(\partial F_{P^\star}))^*$, where
$\partial F_{P^\star}$ is a simplicial complex if $P$ is simple
(see Example \ref{polytope-duality}).
So the assertion follows from (a) and (d).
\end{proof}

\subsection{Affine polytopal complexes}
An {\it affine polytopal complex} (compare \cite[last remark to
Definition 2.1]{Be}) is a countable conditionally complete poset $K$ where to every
$\sigma\in K$ there is associated an isomorphism of $\fll\sigma\flr$ with
the poset of non-empty faces of a convex polytope $P_\sigma$ (compare
Example \ref{polytopes}) so that every face inclusion $\fll\tau\flr\incl\fll\sigma\flr$
is realized by an affine isomorphism between $P_\tau$ and the face of
$P_\sigma$ corresponding to $\tau$.
This includes cubical and simplicial complexes as special cases.
Every affine polytopal complex is atomic since the poset of nonempty faces of
every convex polytope is.
Affine polytopal complexes are not equivalent to ``cell complexes'' in the sense
of Rourke--Sanderson \cite{RS}, who additionally require linear embeddability
into some Euclidean space.
For instance, the cubulation of the M\"obius band into $3$ squares is an
affine polytopal complex (and a cubical complex) but not a cell complex
in the sense of \cite{RS} (see \cite{BuP}).

We note the following consequence of Lemma \ref{3.6}(c,e), which can be seen as
a combinatorial abstraction of a well-known geometric construction (see \cite{BuP}).

\begin{corollary} If $K$ is an affine polytopal complex whose polytopes
are simple, then $K$ is a simple poset.
In particular, $K^\#$ is a cubical complex.
\end{corollary}

\section{Mapping cylinder}

\subsection{Adjunction preposet}
Let $P=(\mathcal P,\le)$ be a poset, $A=(\mathcal A,\le)$ a subposet of $P$
and $f\:A\to Q$ a monotone map, where $Q=(\mathcal Q,\le)$ is a preposet.
We recall that the adjunction set $\mathcal P\cup_f\mathcal Q$ is
the set-theoretic pushout of the diagram
$\mathcal P\supset\mathcal A\xr{f}\mathcal Q$, that is, the quotient of
$\mathcal P\sqcup\mathcal Q$ by the equivalence relation generated by
$a\sim f(a)$ for all $a\in\mathcal A$.

We define the {\it adjunction preposet} $P\cup_f Q$ to be the preposet
$P\cup_f Q=(\mathcal P\cup_f\mathcal Q,\le)$, where $q\le q'$ in $P\cup_f Q$
if there exist $p\in [q]$ and $p'\in [q']$ such that $p\le p'$.
It is easy to see that this relation is reflexive and acyclic
(given a nontrivial cycle in $P\cup_f Q$, it can be shortened to a cycle
in $Q$ using that $P$ is a poset).
It is also clear that $P\cup_f Q$ is the pushout of the diagram
$P\supset A\xr{f} Q$ in the category of preposets.

\begin{lemma}\label{adjunction poset} Let $P\supset A\xr{f} Q$ be a partial
monotone map of posets.

(a) If $f$ is closed and $A$ is closed, then $P\cup_f Q$ is a poset.

(b) If $f$ is open and $A$ is open, then $P\cup_f Q$ is a poset.
\end{lemma}

\begin{proof}
Since $P$ and $Q$ are posets, it suffices to check transitivity for triples
of the form $[p]>[a]=[f(a)]>[q]$ and $[p]<[a]=[f(a)]<[q]$, where $p\in P$,
$a\in A$ and $q\in Q$.

Suppose that $p\ge a$ and $f(a)\ge q$.
If $f$ is closed, then $q=f(b)$ for some $b\le a$.
Then $p\ge b$, and consequently $[p]\ge[b]=[q]$ in $P\cup_f Q$.
On the other hand, if $A$ is open, then $p\in A$, and consequently
$f(p)\ge f(a)\ge q$.
Since $Q$ is a poset, $f(p)\ge q$, and thus $[p]=[f(p)]\ge[q]$.

Now suppose that $p\le a$ and $f(a)\le q$.
If $f$ is open, then $q=f(b)$ for some $b\ge a$.
Then $p\ge b$, and consequently $[p]\le[b]=[q]$ in $P\cup_f Q$.
On the other hand, if $A$ is closed, then $p\in A$, and consequently
$f(p)\le f(a)\le q$.
Since $Q$ is a poset, $f(p)\le q$, and thus $[p]=[f(p)]\le[q]$.
\end{proof}

\subsection{Quotient poset}\label{quotient poset}
Given a subposet $Q$ of a poset $P$, it is easy to see that the adjunction
preposet $P\cup_c pt$, where $f\:Q\to pt$ is the constant map, is a poset.
We call it the {\it quotient poset} and denote $P/Q$.

\subsection{Amalgam}
If $A$ and $B$ are subposets of posets $P$ and $Q$, respectively, and
$h\:A\to B$ an isomorphism, then the {\it amalgam} $P\cup_h Q$, also written
$P\cup_{A=B}Q$ when $h$ is clear from context or irrelevant, is the adjunction
preposet $P\cup_{jh} Q$, where $j\:B\to Q$ is the inclusion.

\begin{corollary} \label{amalgam}
Let $A$ and $B$ be subposets of posets $P$ and $Q$, respectively.
If $A$ and $B$ are both closed or both open, then $P\cup_{A=B}Q$ is a poset.
\end{corollary}

\subsection{Mapping cylinder}
Given a monotone map of posets $f\:P\to Q$, the {\it mapping cylinder}
$MC(f)=P\x[2]\cup_{f_1}Q$ and the {\it dual mapping cylinder}
$MC^*(f)=P\x[2]\cup_{f_2}Q$, where $f_i$ is the composition
$P\x\{i\}\simeq P\xr{f}Q$.
Note that $MC^*(f)\simeq (MC(f^*))^*$ and that there are natural monotone
bijections (which are not embeddings) $MC(f)\to Q+P$ and $MC^*(f)\to P+Q$.

Note that $$P*Q\simeq MC(P\x Q\to P)\underset{P\x Q}\cup MC(P\x Q\to Q)$$
and
$$P\cojoin Q\simeq MC^*(P\x Q\to P)\underset{P\x Q}\cup MC^*(P\x Q\to Q).$$

\begin{corollary}\label{mapping cylinder}
Let $f\:P\to Q$ be a monotone map of posets.
Then $MC(f)$ is a poset if and only if $f$ is closed; dually, $MC^*(f)$
is a poset if and only if $f$ is open.
\end{corollary}

\begin{proof} If $f$ is closed, then $MC(f)$ is a poset by Lemma \ref{adjunction
poset}.
If $f$ is not closed, there exist a $p\in P$ and a $q<f(p)$ such that
$q\ne f(p')$ for any $p'<p$.
Then $(p,2)>(p,1)=f(p)>q$ but $(p,2)\ngtr q$ in $MC(f)$.
The dual assertion follows from $MC^*(f)\simeq (MC(f^*))^*$.
\end{proof}

\begin{lemma} \label{pullback of MC} Let $\psi\:Q'\to Q$ be a closed map of posets.
Consider pullback diagrams
$$\begin{CD}
P'@>\phi>>P\\
@Vf'VV@VfVV\\
Q'@>\psi>>Q
\end{CD}
\qquad\text{and}\qquad
\begin{CD}
X@>\Phi>>P\\
@VFVV@VfVV\\
MC(\psi)@>\psi\cup\id_Q>>Q.
\end{CD}$$
Then

(a) $\phi$ is closed;

(b) $X\simeq MC(\phi)$, $\Phi\simeq\phi\cup\id_P$ and $F\simeq f'\cup f$.
\end{lemma}

\begin{proof}[Proof. (a)]
Let $(p,q')\in P'\subset P\x Q'$.
Then $f(p)=\psi(q')$, and $\phi((p,q'))=p$.
Let $p_0\le p$.
Then $f(p_0)\le\psi(q')$.
Since $\psi$ is closed there exists a $q_0'\le q'$ such that $\psi(q_0')=f(p_0)$.
Thus $(p_0,q_0')\in P'$ and $(p_0,q_0')\le (p,q')$.
\end{proof}

\begin{proof}[(b)]
Clearly, $X$ is a union of an open copy of $P'$ and a closed copy of $P$.
It suffices to check that the resulting bijection between $X$ and $MC(\phi)$
is monotone.
Let $(p',q')\in P'\subset P\x Q'$, so that $f(p')=\psi(q')$, and let
$(p,f(p))\in\Gamma_f\subset P\x Q$, where $\Gamma_f$ is the graph of $f$.
Then $(p',q')>(p,f(p))$ in $P\x MC(\psi)$ if and only if
$p'\ge p$ in $P$ and $q'>f(p)$ in $MC(\psi)$.
We recall that $q'>f(p)$ in $MC(\psi)$ if and only if $\psi(q')\ge f(p)$.
But the latter follows from $p'\ge p$ due to our assumption $\psi(q')=f(p')$.
On the other hand, $(p',q')>p$ in $MC(\phi)$ if and only if
$\phi((p',q'))\ge p$.
But $\phi((p',q'))=p'$, and the assertion follows.
\end{proof}

\subsection{Long mapping cylinder}
Given a monotone map $f\:P\to Q$ between preposets, let
$LMC(f)=MC(f)\cup_{P=P\x\{1\}}P\x[2]$.

\begin{lemma}\label{MC CCP}
Let $f\:P\to Q$ be a monotone map between conditionally complete posets
that preserves infima.
Then

(a) $\left<MC(f)\right>$ is conditionally complete.

(b) $\left<LMC(f)\right>$ is conditionally complete.
\end{lemma}

Note that $MC(pt\sqcup pt\to pt)\cup_{pt\sqcup pt}MC(pt\sqcup pt\to pt)$ is not
conditionally complete.

Dually to (a), if $f$ preserves suprema, then $\left<MC^*(f)\right>$ is
conditionally complete (using that $MC^*(f)=(MC(f^*))^*$).

\begin{proof}[Proof. (a)] Let $S$ be a nonempty subset of $\left<MC(f)\right>$.
Suppose that $q\in Q$ and $s\in S\cap P$.
If $q\le f(s)$, then $q\le s$.
Conversely, if $q\le s$, then $q\le f(p)$ for some $p\le s$; hence $q\le f(s)$.
Thus lower bounds of $S$ that lie in $Q$ coincide with lower bounds of
$(S\cap Q)\cup f(S\cap P)$.

If every lower bound of $S$ belongs to $Q$, then the set of lower bounds of $S$
coincides with that of $(S\cap Q)\cup f(S\cap P)$.
If this set is nonempty, then it has a greatest element (since $Q$ is conditionally
complete).

It remains to consider the case where $S$ has a lower bound in $P$.
Then $S$ has a greatest lower bound $p$ in $P$ (since $P$ is conditionally complete).
Since $f$ preserves infima, $f(p)$ is a greatest lower bound of $f(S)$.
If $q\in Q$ is a lower bound of $S$, then it is a lower bound of $f(S)$,
and hence $q\le f(p)\le p$.
Thus $p$ is globally a greatest lower bound of $S$.
\end{proof}

\begin{proof}[(b)]
Let $S$ be a nonempty subset of $\left<LMC(f)\right>$.
If $S$ has a lower bound $(p,0)$ in $P\x\{0\}$ and a lower bound in $Q$, then
$S\subset P\x\{1\}$, and consequently $(p,1)$ is also a lower bound of $S$.
Hence either every lower bound of $S$ belongs to $Q$, or every lower bound of $S$
belongs to $R$, or $S$ has a lower bound in $P$.
Each of the three cases is considered similarly to the proof of (a).
\end{proof}

\begin{example} \label{non-CCP MC}
Let $f\:\Delta^2\to\Delta^1$ be a simplicial surjection.
It is easy to see that $f$ preserves suprema.
Being simplicial, $f$ is closed, so $MC(f)$ is a poset;
but not a conditionally complete one.

Indeed, writing $f$ as $\{a\}*\{b,c\}\to\{a\}*\{d\}$,
we see that the edges $\{a,b\}$ and $\{a,c\}$ in the domain have a greatest
lower bound in the domain (namely, their common vertex $\{a\}$) and a greatest
element among lower bounds in the range (namely, their common image $\{a,d\}$),
but no greatest lower bound in the entire mapping cylinder.

The problem persists for the mapping cylinder $MC(f^{\#n})$ of the iterated
canonical subdivision $f^{\#n}\:(\Delta^2)^{\#n}\to(\Delta^1)^{\#n}$.
\end{example}

\subsection{Thick mapping cylinder}
Let $f\:P\to Q$ be a monotone map of posets.
Let us consider $R:=\fll\Gamma(f)\flr_{P\x Q}$, the closure of the graph
$\Gamma(f)=\{(p,f(p))\mid p\in P\}$ in $P\x Q$.
Let $\pi_P$ and $\pi_Q$ be the projections onto the factors of $P\x Q$.
Being cubical, they are closed maps, and since $R$ is closed in $P\x Q$, so
are their restrictions to $R$.
Hence $TMC(f)=MC(\pi_P|_R)\cup_R MC(\pi_Q|_R)$ is a poset.
It follows then that $TMC(f)$ is a closed subposet of $P*Q$.
From this we get that if $P$ and $Q$ are conditionally complete,
then so is $TMC(f)$.

\begin{theorem}\label{tmc}
If $f\:P\to Q$ is a monotone map between posets, where $Q$ is conditionally
complete, then $TMC(f)$ strong deformation retracts onto $LMC(f)$ via
a monotone map $r\:TMC(f)\to LMC(f)$, homotopic to the identity by
a monotone homotopy $H\:TMC(f)\x I\to TMC(f)$ that extends the projection
$LMC(f)\x I\to LMC(f)$.
\end{theorem}

\begin{proof}
Let $R=\fll\Gamma(f)\flr_{P\x Q}$.
Given a $(p,q)\in R$, we have $(p,q)\le (x_{pq},f(x_{pq}))$ for some $x_{pq}\in P$.
Then $q\le f(x_{pq})$ and $f(p)\le f(x_{pq})$, so $f(x_{pq})$ is an upper bound
for $q$ and $f(p)$.
Let $y_{pq}$ be the least upper bound for $q$ and $f(p)$.
Let us define $g\:R\to R$ by $(p,q)\mapsto(p,y_{pq})$.
If $(p,q)\le (p',q')$, then $q\le q'$ and $f(p)\le f(p')$, so $y_{p'q'}$ is
an upper bound for $q$ and $f(p)$.
Hence $y_{pq}\le y_{p'q'}$, and thus $g$ is monotone.

Let $r$ be the composition $P\x Q\xr{\pi_P|}P\xr{\id_P\x f}\Gamma(f)$.
Since $g((p,q))\ge(p,q)$ and $g((p,q))\ge (p,f(p))=r((p,q))$, the maps
$\id_R$, $g$ and $r|_R$ combine into a monotone map $h\:R\x I\to R$.
Since $\pi_P(g(p,q))\ge p$ and $\pi_Q(g(p,q))\ge q$, the map $h$ and
the projection $(P\sqcup Q)\x I\to P\sqcup Q$ combine into a monotone
map $H\:TMC(f)\x I\to TMC(f)$.
\end{proof}

\subsection{Iterated mapping cylinder}
For a monotone map $f\:P\to Q$, let $\pi\:MC(f)\to Q$ and $\pi\:MC^*(f)\to Q$
be defined by $\pi|_P=f$ and $\pi|_Q=\id$.
If $f$ is closed [resp.\ open], clearly $\pi\:MC(f)\to Q$ is closed [resp.\
$\pi\:MC^*(f)\to Q$ is open].
Let us define the preposet $MC(P_n\xr{f_n}\dots\xr{f_1}P_0)=MC(F_n)$, where $F_n$
is the composition $MC(P_n\xr{f_n}\dots\xr{f_2}P_1)\xr{\pi}P_1\xr{f_1}P_0$.
Similarly $MC^*(P_n\xr{f_n}\dots\xr{f_1}P_0)=MC^*(F_n)$, where $F_n$
is the composition $MC^*(P_n\xr{f_n}\dots\xr{f_2}P_1)\xr{\pi}P_1\xr{f_1}P_0$.
If $f_1,\dots,f_n$ are closed [resp.\ open], clearly
$MC(P_n\xr{f_n}\dots\xr{f_1}P_0)$ [resp.\ $MC^*(P_n\xr{f_n}\dots\xr{f_1}P_0)$]
is a poset.

\subsection{Homotopy colimit} \label{homotopy-colimit}
Suppose we have a covariant [resp.\ contravariant] commutative diagram,
indexed by a poset $\Lambda$, consisting of posets $P_\lambda=(\mathcal P_\lambda,\le)$,
$\lambda\in\Lambda$, and monotone maps $f_{\lambda\mu}\:P_\lambda\to P_\mu$ whenever
$\lambda>\mu$ [resp.\ $\lambda<\mu$] in $\Lambda$.
Then the {\it homotopy colimit} $\hocolim_\Lambda (P_\lambda,f_{\lambda\mu})$
is the preposet $Q=(\mathcal Q,\le)$, where
$\mathcal Q=\bigsqcup_{\lambda\in\Lambda}\mathcal P_\lambda$,
and given a $p\in\mathcal P_\lambda$ and a $q\in\mathcal P_\mu$, we have $p\ge q$
in $Q$ if and only if either $\lambda=\mu$ and $p\ge q$ in $P_\lambda$, or
$\lambda>\mu$ and there exists a $p'\in\mathcal P_\lambda$ such that $p'\le p$ in
$P_\lambda$ and $f_{\lambda\mu}(p')=q$ [resp.\ there exists a $q'\in\mathcal P_\mu$
such that $q'\ge q$ in $P_\mu$ and $f_{\mu\lambda}(q')=p$].
If the maps $f_{\lambda\mu}$ are closed [resp.\ open], it is clear that
$\hocolim_\Lambda(P_\lambda,f_{\lambda\mu})$ is a poset.
Clearly, $MC(P_n\xr{f_n}\dots\xr{f_2}P_1)=\hocolim_{[n]}(P_i,f_{ij})$,
where $f_{ij}\:P_i\to P_j$ is the composition $P_i\xr{f_i}\dots\xr{f_{j+1}}P_j$.
Similarly $MC^*(P_n\xr{f_n}\dots\xr{f_2}P_1)=\hocolim_{[n]^*}(P_i,f_{ij})$.
In general, it is easy to see that every homotopy colimit indexed by a poset
is in fact an iterated amalgam of the iterated mapping cylinders.
Since every $f_{\lambda_mu}$ is a single-valued (not multi-valued) map,
every subposet $P_\lambda$ of $\hocolim_\Lambda (P_\lambda,f_{\lambda\mu})$
is full in the homotopy colimit.

The following is straightforward from Lemma \ref{nonsingular Hatcher}.

\begin{theorem} \label{hocolim}
If $f\:P\to Q$ is a closed full map of finite posets and $P$ is
nonsingular (e.g.\ $f$ is simplicial), then
$P=\hocolim_Q(f^{-1}(q),f_{rq})$, where $f_{rq}\:f^{-1}(r)\to f^{-1}(q)$
are the Hatcher maps.
\end{theorem}

The following is a combinatorial form of Homma's factorization lemma
(see \cite[proof of Lemma 5.4.1]{Ru} and \cite[Lemma 2.1]{Co3}).

\begin{corollary}\label{factorization}
Let $f\:P\to Q$ be a closed full map of finite posets, where $P$
is nonsingular.
Then $f$ factors as a composition of closed full maps $f_i$ such that
each $f_i$ has only one non-degenerate point inverse.
\end{corollary}

\begin{proof}
Let us extend the partial order on $Q$ to a total order.
Let $F_q^{(i)}=f^{-1}(q)$ if $q$ is not among the first $i$ elements of $Q$,
and otherwise let $F_q^{(i)}$ be the singleton poset.
Let $f_{rq}^{(i)}\:F_r^{(i)}\to F_q^{(i)}$ be the Hatcher map if
$F_q^{(i)}=f^{-1}(q)$ and the unique map if $F_q^{(i)}$ is a singleton poset.
Let $P_i=\hocolim_Q(F_q^{(i)},f_{rq}^{(i)})$.
Then the obvious maps $P=P_0\xr{f_1}\dots\xr{f_n}P_n$ are as desired.
\end{proof}

The following lemma is straightforward.

\begin{lemma} \label{conical hocolim}
(a) Given a covariant commutative diagram of posets $P_\lambda$ and
monotone maps $f_{\lambda\mu}\:P_\lambda\to P_\mu$ indexed by a cone $C\Lambda$,
we have
$$\hocolim_{C\Lambda}(P_\lambda,f_{\lambda\mu})\simeq
MC\left(P_{\hat 1}\xl{\pi} P_{\hat 1}\x\Lambda\xr{F}
\hocolim_\Lambda(P_\lambda,f_{\lambda\mu})\right),$$ where
$F(p,\lambda)=f_{\hat 1\lambda}(p)$.

(b) Given a contravariant commutative diagram of posets $P_\lambda$ and monotone
maps $f_{\lambda\mu}\:P_\lambda\to P_\mu$ indexed by a dual cone $C^*\Lambda$,
we have $\hocolim_{C^*\Lambda}(P_\lambda,f_{\lambda\mu})\simeq MC(F)$,
where $F\:\hocolim_\Lambda(P_\lambda,f_{\lambda\mu})\to P_{\hat 0}$
is given by $F|_{P_\lambda}=f_{\lambda\hat 0}$.
\end{lemma}

\section{Canonical and barycentric handles} \label{handles}

The {\it barycentric subdivision map} $\flat\:P^\flat\to P$ is defined by
sending every nonempty chain $\{p_1<\dots<p_n\}$ to its maximal element $p_n$.
The {\it canonical subdivision map} $\#\: P^\#\to P$ is defined by sending
every interval $[p,q]$ to its maximal element $q$.
(These are instances of a general notion of subdivision map to be studied in
\S\S\ref{subdivision map-section} below.)

\begin{lemma} Let $P$ be a conditionally complete poset.
Then $\#\:P^\#\to P$ preserves infima and suprema.
\end{lemma}

\begin{proof}
Suppose that $S\subset P^\#$ has a supremum $[p_0,q_0]$.
Thus each $[p,q]\in S$ is included in $[p_0,q_0]$, and each upper bound $[p,q]$
of $S$ includes $[p_0,q_0]$.
Given an upper bound $q_1$ of $\#(S)$ and a $[p,q]\in S$, we have
$[p,q]\subset [p_0,q_0]$ and $q\le q_1$, hence $[p,q]\subset [p_0,q_1]$.
Then $[p_0,q_1]$ is an upper bound of $S$, so $[p_0,q_0]\subset[p_0,q_1]$.
Hence $q_0\le q_1$, which shows that $q_0$ is the supremum of $\#(S)$.

Suppose that $S\subset P^\#$ has an infimum $[p_0,q_0]$.
Thus each $[p,q]\in S$ includes $[p_0,q_0]$, and each lower bound $[p,q]$
of $S$ is included in $[p_0,q_0]$.
Given a lower bound $q_1$ of $\#(S)$ and a $[p,q]\in S$, we have
$[p_0,q_0]\subset [p,q]$ and $q_1\le q$, hence $[p_0,q_1]\subset [p,q]$.
Then $[p_0,q_1]$ is a lower bound of $S$, so $[p_0,q_1]\subset[p_0,q_0]$.
Hence $q_1\le q_0$, which shows that $q_0$ is the infimum of $\#(S)$.
\end{proof}

\begin{corollary}
If $P$ is a conditionally complete poset, then $MC^*(\#\:P^\#\to P)$ is
a conditionally complete poset.
\end{corollary}

\subsection{Canonical handles}
Let $X$ be a poset.
Then $[\sigma,\tau]<[\rho,\upsilon]$ in $X^\#$ iff $\rho\le\sigma$ and
$\tau\le\upsilon$.
The dual cone $\cel[\sigma,\tau]\cer$ is the poset of all such intervals
$[\rho,\upsilon]$; clearly, it is isomorphic to
$(\fll\sigma\flr)^*\x\cel\tau\cer$.
Then every cone $\fll[\sigma,\tau]^*\flr=(\cel[\sigma,\tau]\cer)^*$
of the dual poset $h(X)\bydef(X^\#)^*$ is isomorphic to the product
$\fll\sigma\flr\x(\cel\tau\cer)^*=\fll\sigma\flr\x\fll\tau^*\flr$.

The maximal cones of $h(X)$ (i.e.\ those cones that are not properly contained
in other cones) are of the form $\fll[\sigma,\sigma]^*\flr$, where
$\sigma\in X$, and are called the {\it (canonical) handles} $h_\sigma$ of
the poset $X$.
Each $h_\sigma$ is isomorphic to the product of the cone $\fll\sigma\flr$ of
$X$ and the cone $\fll\sigma^*\flr$ of $X^*$.
These cones are called the {\it core} and {\it cocore} of the handle
$h_\sigma$.

When $\sigma<\tau$, the intersection $h_\sigma\cap h_\tau$ is clearly
the cone $\fll[\sigma,\tau]^*\flr$ of $h(X)$.
However, when $\sigma$ and $\tau$ are incomparable in $X$, it may be
that $h_\sigma\cap h_\tau$ is nonempty (namely, it is nonempty iff both
$\fll\sigma\flr\cap\fll\tau\flr$ and $\cel\sigma\cer\cap\cel\tau\cer$ are
nonempty).
This distinguishes the canonical handles from the barycentric ones, which will
be discussed in a moment.

\subsection{Collapsing handles onto cores}\label{handles onto cores}
Let $X$ be a poset.
Let us consider the composition $X^\#\xr{j_X}(X^*)^\#\xr{\#}X^*$, where
the isomorphism $j_X\:X^\#\to (X^*)^\#$ is given by
$[\sigma,\tau]\mapsto [\tau^*,\sigma^*]$.
Then the dual map $r_X\:h(X)\to X$ to this composition is given by
$[\sigma,\tau]^*\mapsto\sigma$.
On the other hand, the composition
$\bar r_X\:h(X)\xr{(j_X)^*}h(X^*)\xr{r_{X^*}}X^*$ is dual to $\#\:X^\#\to X$
and is given by $[\sigma,\tau]^*\mapsto\tau^*$.

The restrictions of $r_X$ and $\bar r_X$ to the cone $\fll[\sigma,\tau]^*\flr\simeq
\fll\sigma\flr\x\fll\tau^*\flr$ are the projection onto the first and second factor,
respectively.
We also note that $h_\sigma=r_X^{-1}(\fll\sigma\flr)$ and
$h_\tau=\bar r_X^{-1}(\fll\tau^*\flr)$.

\begin{lemma} \label{canonical of prejoin}
Let $P$ and $Q$ be posets.
Then

(a) $(P^*+Q)^\#\simeq MC(P\x Q\xl{\#\x\id}P^\#\x Q\xr{\pi}P^\#)\underset{P\x Q}\cup
MC(P\x Q\xl{\id\x\#}P\x Q^\#\xr{\pi}Q^\#),$\!

(b) $h(CP)\simeq MC(r_P)\cup_P CP$.
\end{lemma}

\begin{proof}[Proof. (a)]
An interval of $P^*+Q$ is either an interval of $P^*$ (or equivalently of $P$)
or an interval of $Q$, or a dual cone of $P^*$ plus a cone of $Q$ (or equivalently,
a cone $P\x Q$).
\end{proof}

\begin{proof}[(b)] Part (a) specializes to
$(CP)^\#\simeq MC^*((P^*)^\#\xr{\#}P^*)\cup_{P^*} MC(P^*\to pt)$.
Note that $MC(P^*\to pt)\simeq C^*(P^*)$.
Hence $h(CP)\simeq MC(h(P)\xr{r_P}P)\cup_P CP$.
\end{proof}

\subsection{Barycentric handles}
Associated to every poset $P$ is the barycentric handle decomposition
$H(P):=(P^\flat)^*$.

Since $(P+Q)^\flat\simeq P^\flat*Q^\flat$, we have $H(P+Q)\simeq H(P)\cojoin H(Q)$.
In particular, $H(CP)\simeq H(P)\x[2]\cup_{H(P)}CH(P)$.

A barycentric handle corresponding to a $\sigma\in P$ is the cone
$H_\sigma:=\fll(\hat\sigma)^*\flr$ of the maximal element $(\hat\sigma)^*$ of $H(P)$.
We write $H_\sigma=H_\sigma^P$ when $P$ is not clear from the context.

\begin{lemma} $H_\sigma$ is isomorphic to the product of cones
$CH(\partial\fll\sigma\flr)\x CH(\partial\fll\sigma^*\flr)$.
\end{lemma}

It follows, in particular, that
$H_\sigma\simeq H_\sigma^{\fll\sigma\flr}\x H_\sigma^{\cel\sigma\cer}$.

\begin{proof}
$\partial H_\sigma=\partial\fll (\hat\sigma)^*\flr$ can be identified with
$H(\partial\fll\sigma\flr+\partial^*\cel\sigma\cer)$.
By the above the latter is isomorphic to $H(\partial\fll\sigma\flr)\cojoin
H(\partial^*\cel\sigma\cer)$.
Since $H(Q)\simeq H(Q^*)$, we have $H(\partial^*\cel\sigma\cer)\simeq
H(\partial\fll\sigma^*\flr)$.
Thus $H_\sigma$ is isomorphic to the cone over
$H(\partial\fll\sigma\flr)\cojoin H(\partial\fll\sigma^*\flr)$
By the van Kampen duality, this cone is isomorphic to
$CH(\partial\fll\sigma\flr)\x CH(\partial\fll\sigma^*\flr)$.
\end{proof}

\part{SUBDIVISION AND COLLARS}\label{subdivision and collars}

Starting with this chapter, all posets and preposets are assumed to be finite.

If $P$ is a poset or preposet, $|P|$ will denote
the polyhedron triangulated by the simplicial complex $P^\flat$.
If $f\:P\to Q$ is a monotone map, $|f|$ will denote the PL map triangulated
by the simplicial map $f^\flat\:P^\flat\to Q^\flat$.

To be more specific, we can canonically embed $P^\flat$ into the simplex
$\Delta^{\mathcal P}$, where $\mathcal P$ is the underlying set of $P$
(Lemma \ref{2.9}) and consider the corresponding affine simplicial subcomplex
of an affine simplex whose face poset is isomorphic to $\Delta^{\mathcal P}$.

We will also assume all polyhedra to be compact (and endowed with PL structure
apart from topology), and all maps between polyhedra to be PL.

\section{Subdivision maps} \label{subdivision map-section}

\begin{lemma} (a) If $P$ is a closed suposet of a poset $Q$, then $|P|$
is homeomorphic to $|CQ|$ keeping $|Q|$ fixed if and only if $(|P|,|Q|)$
is homeomorphic to $(|CQ|,|Q|)$.

(b) If $P$ is a poset, and $|P|$ is homeomorphic to $|C(\partial P)|$,
then they are homeomorphic keeping $|\partial P|$ fixed.
\end{lemma}

\begin{proof}[Proof. (a)] This follows from the fact that every self-homeomorphism
of $|Q|$ extends to a self-homeomorphism of the cone $C|Q|$ (which in turn is
homeomorphic to $|CQ|$ keeping $|Q|$ fixed).
\end{proof}

\begin{proof}[(b)] We recall that $\partial P$ is the closure of the subposet of
all $p\in P$ such that $\lk(p,P)$ is a singleton.
By considering barycentric subdivisions we may assume that $P$ is
a simplicial complex.
Then by the combinatorial invariance of link any homeomorphism
$|P|\to|\partial P*pt|$ sends $|\partial P|$ into itself.
Thus $(|P|,|\partial P|)$ is homeomorphic to $(|C(\partial P)|,|\partial P|)$,
and the assertion follows from (a).
\end{proof}

\subsection{Subdivision map} By a {\it subdivision map} we mean a monotone map
$f\:P'\to P$ between posets such that $|f^{-1}(\fll p\flr)|$ is homeomorphic to
$|Cf^{-1}(\partial\fll p\flr)|$ by a homeomorphism fixed on
$|f^{-1}(\partial\fll p\flr)|$ for each $p\in P$.

Clearly, every subdivision map is an open map.
If $f\:P'\to P$ is a subdivision map, then by the combinatorial invariance of
link $\partial (P')=f^{-1}(\partial P)$.

\subsection{Realizable subdivision map}
If $P$ and $P'$ are affine polytopal [resp.\ simplicial] complexes, we say that
a subdivision map $\alpha\:P'\to P$ is {\it affinely [resp.\ simplicially]
realizable} if there exists an affine subdivision $K'\vartriangleright K$ of
affine polytopal [simplicial] complexes $K$ and $K'$ such that $P$ and $P'$ are
the posets of nonempty faces of $K$ and $K'$, and the smallest polytope [simplex]
of $K$ containing the polytope [simplex] of $K'$ corresponding to a $p\in P'$
corresponds to $\alpha(p)\in P$.

\begin{theorem}\label{subdivision map} A monotone map of posets $f\:P'\to P$ is
a subdivision map if and only if there exists a homeomorphism $h\:|P'|\to |P|$
such that $h^{-1}(|\fll p\flr|)=|f^{-1}(\fll p\flr)|$ for each $p\in P$.
\end{theorem}

Such a homeomorphism $h$ will be called an {\it underlying homeomorphism} of
the subdivision map $f$.

The proof of Theorem \ref{subdivision map} is rather straightforward.

\begin{proof}
Assume inductively that for some closed subposet $Q$ of $P$ there exists
an underlying homeomorphism $h_Q$ of $f|_{Q'}$, where $Q'=f^{-1}(Q)$.
Pick some $p\in P\but Q$ such that $R:=\partial\fll p\flr\subset Q$.
Let $Q_+=Q\cup CR\subset P$ and $(Q_+)'=f^{-1}(Q_+)$.
Let $R'=f^{-1}(R)$ and let $(Q')_+$ be the amalgam $Q'\cup_{R'} CR'$.
Then $f|_{(Q_+)'}$ factors uniquely as $(Q_+)'\xr{g}(Q')_+\xr{f_+} Q_+$,
where $g$ is a subdivision map extending $\id_{Q'}$ and $f_+$ is a subdivision map
extending $f|_{Q'}$.
Then $h_{Q'}$ extends conically to an underlying homeomorphism $h_+$ of $f_+$,
and $\id_{|Q'|}$ extends, using the homeomorphism from the definition of
a subdivision map, to an underlying homeomorphism $h_g$ of $g$.
The composition $|(Q_+)'|\xr{h_g}|(Q')_+|\xr{h_+}|Q_+|$ is clearly an underlying
homeomorphism of $f|_{(Q_+)'}$, which completes the induction step.

Conversely, suppose that $h$ is an underlying homeomorphism of $f$.
Given a $p\in P$, we have $|\partial\fll p\flr|=\bigcup_{q<p}|\fll q\flr|$,
and consequently
$$h^{-1}(|\partial\fll p\flr|)=\bigcup_{q<p} h^{-1}(|\fll q\flr|)=
\bigcup_{q<p}|f^{-1}(\fll q\flr|)|=|f^{-1}(\partial\fll p\flr)|.$$
Hence we get the composite homeomorphism
$$|f^{-1}(\fll p\flr)|\xr{h|_{h^{-1}(|\fll p\flr|)}}|\fll p\flr|=
C|\partial\fll p\flr|\xr{C(h^{-1}|_{|\partial\fll p\flr|})}
C(f^{-1}(|\partial\fll p\flr|)),$$
which is the identity on $f^{-1}(|\partial\fll p\flr|$, as desired.
\end{proof}

\begin{corollary} \label{subdivision-composition}
Let $P''\xr{f}P'\xr{g}P$ be monotone maps of posets, where $f$ is a subdivision map.
Then $g$ is a subdivision map if and only if $gf$ is.
\end{corollary}

\begin{remark}
Subdivision maps of cell complexes were studied by Stanley \cite{Stan} (see also \cite{Ath}) and 
by Mn\"ev, who calls them ``abstract assemblies'' \cite{Mn}.
They seem to originate in \cite[Conjecture on p.\ 17]{Ste} (in the case
of affine cell complexes).

Subdivisions that are {\it dual} to simplicial maps of simplicial complexes
were studied by Akin \cite{Ak2}, who called them transversely cellular maps,
and later rediscovered by Dragotti and Magro (see \cite{DM}) under the
name of strong cone-dual maps; both approaches extend Cohen's work on
the case where the domain is a manifold \cite{Co1}.

Akin gave a nontrivial proof of Corollary \ref{subdivision-composition}
in the case of maps dual to simplicial maps \cite[Corollary 3 on p.\ 423]{Ak2}.
He also proved, by providing some invariant characterizations, that the property
of being a subdivision map for the dual to a simplicial map does not depend on
the choice of triangulations \cite{Ak2}.

Subdivisions of posets as defined by Theorem \ref{subdivision map} are easily seen
to be equivalent to McCrory's notion of subdivision of a cone complex \cite{Mc}.
\end{remark}

The following example is found in \cite{Ak2}:

\begin{example}
If $g$ and $gf$ are subdivision maps, $f$ need not be one: let $M$ be a manifold
and $f=gf\:M\to M$ shrink a codimension zero ball $B$ to a point; then $g$ must
be the identity outside $B$ but can behave arbitrarily inside $B$.
\end{example}

\subsection{Fiberwise subdivision map} A subdivision map $f\:P'\to P$ is
called {\it fiberwise} with respect to monotone maps $\pi\:P\to B$ and $\pi'\:P'\to B$
if $f$ sends $(\pi')^{-1}(\cel b\cer)$ onto $\pi^{-1}(\cel b\cer)$ for each $b\in B$,
and for each $p\in P$ there exists a homeomorphism
$h_p\:|f^{-1}(\fll p\flr)|\to C|f^{-1}(\partial\fll p\flr)|$ keeping
$|f^{-1}(\partial\fll p\flr)|$ fixed and sending
$|(\pi')^{-1}(\cel b\cer)\cap f^{-1}(\fll p\flr)|$ onto
$C|(\pi')^{-1}(\cel b\cer)\cap f^{-1}(\partial\fll p\flr)|$ for every $b\in B$.

In the case that $\pi'$ equals the composition $P'\xr{f}P\xr{\pi} B$, we also say that
the subdivision map $f$ is {\it fiberwise} with respect to $\pi$, or just lies {\it over $B$}
when $f$ is clear from context.

It is easy to see that every subdivision map $f\:P'\to P$ lies over the constant map $P\to pt$.
On the other hand, if $f$ lies over $\id\:P\to P$, then $f$ itself is the identity map.

In dealing with transversality we will also need a peculiar variety of a fiberwise
subdivision map.
To define when the subdivision map $f$ is {\it fiberwise} with respect to monotone maps
$\pi\:P\to B$ and $\pi'\:P'\to B^*$, we repeat the same conditions as above, but
with $(\pi')^{-1}(\cel b\cer)$ replaced by $(\pi')^{-1}(\fll b^*\flr)$ throughout.

\begin{lemma} \label{fiberwise}
Let $f\:P'\to P$ be a fiberwise subdivision map with respect to
$\pi\:P\to B$ and $\pi'\:P'\to B$ [resp.\ $\pi'\:P'\to B^*$], and let $B_0\subset B$ be
an open subposet.
Let $P_0=\pi^{-1}(B_0)$ and $P'_0=(\pi')^{-1}(B_0)$ [resp.\ $P'_0=(\pi')^{-1}(B_0^*)$].
Then the restriction $f_0\:P'_0\to P_0$ of $f$ is a subdivision map with respect to
the restrictions $\pi_0\:P_0\to B_0$ and $\pi'_0\:P'_0\to B_0$
[resp.\ $\pi'_0\:P'_0\to B_0^*$] of $\pi$ and $\pi'$.
In particular, $f_0$ is a subdivision map.
\end{lemma}

\begin{proof} We only treat one case; the case in square brackets is similar.
Given a $p\in P_0$, the hypothesis provides a homeomorphism
$h_p\:|f^{-1}(\fll p\flr)|\to C|f^{-1}(\partial\fll p\flr)|$ keeping
$|f^{-1}(\partial\fll p\flr)|$ fixed and sending $|P_0'\cap f^{-1}(\fll p\flr)|$
onto $C|P_0'\cap f^{-1}(\partial\fll p\flr)|$.
We have $P_0'\cap f^{-1}(\fll p\flr)=f_0^{-1}(P_0\cap\fll p\flr)=
f_0^{-1}(\fll p\flr_{P_0})$, and similarly
$P_0'\cap f^{-1}(\partial\fll p\flr)=f_0^{-1}(\partial\fll p\flr_{P_0})$.
So $h_p$ sends
$|f_0^{-1}(\fll p\flr_{P_0})|$ onto $C|f_0^{-1}(\partial\fll p\flr_{P_0})|$
and keeps $|f_0^{-1}(\partial\fll p\flr_{P_0})|$ fixed.
\end{proof}

Similarly to Theorem \ref{subdivision map} one has

\begin{lemma}\label{fiberwise subdivision map} Suppose that we are given monotone maps
$f\:P'\to P$, $\pi\:P\to B$ and $\pi'\:P'\to B$ [resp.\ $\pi'\:P'\to B^*$] such that
$f$ sends $(\pi')^{-1}(\cel b\cer)$ [resp.\ $(\pi')^{-1}(\fll b^*\flr)$] onto
$\pi^{-1}(\cel b\cer)$ for each $b\in B$.
Then $f$ is a fiberwise subdivision map with respect to $\pi$ and $\pi'$
if and only if there exists a homeomorphism $h\:|P'|\to |P|$
sending $|f^{-1}(\fll p\flr)|$ onto $|\fll p\flr|$ for each $p\in P$ and
$|(\pi')^{-1}(\cel b\cer)|$ [resp.\ $|(\pi')^{-1}(\fll b^*\flr)|$] onto
$|\pi^{-1}(\cel b\cer)|$] for each $b\in B$.
\end{lemma}

\begin{corollary} \label{fiberwise subdivision-composition}
Suppose that we are given monotone maps $P''\xr{f}P'\xr{g}P\xr{\pi}B$,
$\pi'\:P'\to B$ and $\pi''\:P''\to B$ [resp.\ $\pi''\:P''\to B^*$] such that
$f$ sends $(\pi'')^{-1}(\cel b\cer)$ [resp.\ $(\pi'')^{-1}(\fll b^*\flr)$] onto
$(\pi')^{-1}(\cel b\cer)$ and $g$ sends $(\pi')^{-1}(\cel b\cer)$ onto
$\pi^{-1}(\cel b\cer)$ for each $b\in B$.

Then $g$ is a fiberwise subdivision map with respect to $\pi$ and $\pi'$ if and only if $gf$ is
a fiberwise subdivision map with respect to $\pi$ and $\pi''$.
\end{corollary}

To compare combinatorial and topological mapping cylinders it is
convenient to work in concrete categories (see \cite[8.10]{AHS}).

\begin{lemma} \label{quotient map}
Let $f\:P\to Q$ be a quotient map in the concrete category of
preposets and monotone maps over the category of sets.
Then $|f|\:|P|\to |Q|$ is a quotient map in the concrete category of
compact polyhedra and PL maps over the category of sets.
\end{lemma}

\begin{proof}
Clearly, $|f|$ is a quotient map in the concrete category of affine
simplicial complexes and affine simplicial maps over the category of sets.
This reduces the assertion to a standard fact in PL topology.
\end{proof}

\begin{lemma} \label{MC vs MC*}
For any monotone map $f\:P\to Q$ between posets there exist homeomorphisms
$|MC(f)|\cong |MC^*(f)|\cong MC(|f|)$ keeping $|P|$ and $|Q|$ fixed.
\end{lemma}

\begin{proof} The subdivision map $I\to [2]$ yields a subdivision map
$\alpha\:P\x I\to P\x [2]$.
The monotone involution of $I$ yields a monotone involution $h$ of $P\x I$.
Using $h_\alpha$ (see Theorem \ref{subdivision map}) and $|h|$ we obtain
a homeomorphism $|P\x I|\cong|P\x[2]|$ and an involution $H$ of $|P\x[2]|$
interchanging $|P\x\{1\}|$ and $|P\x\{2\}|$.
These descend to homeomorphisms between $|P\x I|\cup_{|f_0|}|Q|$,
$|P\x[2]|\cup_{|f_1|}|Q|$ and $|P\x[2]|\cup_{|f_2|}|Q|$, where $f_0$ is
the composition $P\x\{\{0\}\}\simeq P\xr{f}Q$.
The assertion now follows from Lemma \ref{quotient map}.
\end{proof}

\begin{lemma} \label{subdivision of MC*} Consider a commutative diagram of
monotone maps
$$\begin{CD}
P'@.@>f'>>Q'\\
@VV\alpha V@.\hskip -30pt\searrow^{f^\circ}\hskip 10pt@VV\beta V\\
P@.@>f>>Q,
\end{CD}$$
where $\alpha$ and $\beta$ are subdivision maps and $f$ and $f'$ are open maps.

(a) The induced monotone map $\gamma\:MC^*(f^\circ)\to MC^*(f)$ is a subdivision map
over $[2]$.

(b) Suppose additionally that there exists a commutative diagram
$$\begin{CD}
|P'|@>|f'|>>|Q'|\\
@VVh_\alpha V@VVh_\beta V\\
|P|@>|f|>>|Q|,
\end{CD}$$
where $h_\alpha$ and $h_\beta$ are underlying homeomorphisms of the subdivision maps
$\alpha$ and $\beta$.
Then the induced monotone map $\delta\:MC^*(f')\to MC^*(f^\circ)$ is a subdivision map
over $[2]$.
\end{lemma}

Note that the upper-right triangle is irrelevant for the purposes of (a), since
one can take $\beta=\alpha$ and $f'=f^\circ$.

A diagram as above except that $f$ and $f'$ are not assumed to be open can
be converted into a diagram as above by using the handle functor, see
Theorem \ref{handle subdivision} below.

\begin{proof} Since $\alpha$ is open, so is $f^\circ$.
Hence $M:=MC^*(f)$, $M':=MC^*(f')$ and $M^\circ:=MC^*(f^\circ)$ are posets.
Let us identify $P$, $P'$, $Q$ and $Q'$ with their copies in $M$, $M'$ and
$M^\circ$.
Then $\delta=\id_{P'}\cup\beta$ and $\gamma=\alpha\cup\id_Q$.

(a) Given a $p\in P$, $\fll p\flr_M=\fll p\flr_P$.
Then $\gamma^{-1}(\fll p\flr_M)=\alpha^{-1}(\fll p\flr_P)$ and
$\gamma^{-1}(\partial\fll p\flr_M)=\alpha^{-1}(\partial\fll p\flr_P)$.
Hence $\gamma$ is a subdivision map over $P$.

Given a $q\in Q$, we have $\gamma^{-1}(\fll q\flr_M)=\fll q\flr_{M^\circ}$ and
$\gamma^{-1}(\partial\fll q\flr_M)=\partial\fll q\flr_{M^\circ}$.
Thus $\gamma$ is a subdivision map.
Moreover, it is a subdivision map over $[2]$ since $\gamma|_Q=\id_Q$.

(b) Clearly $\delta$ is the identity (and hence a subdivision map) over $P'$.

Now let us fix a $q\in Q$.
Let us write $B=\fll q\flr_Q$, $A=f^{-1}(B)$ and $dA=f^{-1}(\partial B)$.
Let $B'=\beta^{-1}(B)$, so that $\partial B'=\beta^{-1}(\partial B)$.
Further let $A'=\alpha^{-1}(A)=(f^\circ)^{-1}(B)$ and
$dA'=\alpha^{-1}(dA)=(f^\circ)^{-1}(\partial B)$.

We have $\fll q\flr_M=MC^*(A\xr{f|} B)$ and
$\partial\fll q\flr_M=A\cup MC^*(dA\xr{f|}\partial B)$.
Similarly, $\fll q\flr_{M^\circ}=MC^*(A'\xr{f^\circ|} B)$ and
$\partial\fll q\flr_{M^\circ}=A'\cup MC^*(dA'\xr{f^\circ|}\partial B)$.
Hence $\delta^{-1}(\fll q\flr_{M^\circ})=MC^*(A'\xr{f'|}B')$ and
$\delta^{-1}(\partial\fll q\flr_{M^\circ})=
A'\cup MC^*(dA'\xr{f'|}\partial B')$.

By the hypothesis $|f'|$ is conjugate to $|f|$, namely
$|f'|=h_\beta^{-1}|f|h_\alpha$.
Since $h_\alpha$ and $h_\beta$ are underlying homeomorphisms of $\alpha$ and
$\beta$, they also carry $|A'|\xr{\left.|f'|\vphantom{^J}\right|} |B'|$ onto
$|A|\xr{\left.|f|\vphantom{^J}\right|}|B|$, and
$|dA'|\xr{\left.|f'|\vphantom{^J}\right|}|\partial B'|$ onto
$|dA|\xr{\left.|f|\vphantom{^J}\right|}|\partial B|$.
Since $|MC^*(\phi)|\cong MC(|\phi|)$ by Lemma \ref{MC vs MC*}, it follows that
$(|MC^*(A'\xr{f'|}B')|,\,|A'\cup MC^*(dA'\xr{f'|}\partial B')|)$
is homeomorphic to $(|MC^*(A\xr{f|}B)|,\,|A\cup MC^*(dA\xr{f|}\partial B)|)$.
But the latter pair is the same as
$(|\fll q\flr_M|,|\partial\fll q\flr_M|)$.
\end{proof}

In the case $\beta=\alpha$, $f=\id$, $f'=\id$ we get

\begin{corollary} \label{MC* of subdivision} Let $\alpha\:K'\to K$ be a subdivision map.
Then the induced monotone maps $MC^*(\id_{K'})\xr{\delta} MC^*(\alpha)\xr{\gamma}
MC^*(\id_K)$ are subdivision maps over $[2]$.
\end{corollary}

\section{Collaring}

In this section we give a new short proof of Whitehead's collaring theorem
for simplicial complexes (see \cite{RS} and \cite[4.2]{Co2} for two other proofs
and references to older proofs) using mapping cylinders of non-simplicial maps.
We also note the (rather trivial) generalization to posets, which will be
used later.

We say that a closed subposet $Q$ of $P$ is {\it collared} in $P$ if for every
$q\in Q$, $|\lk(q,P)|$ is homeomorphic to the cone over $|\lk(q,Q)|$ by
a homeomorphism keeping $|\lk(q,Q)|$ fixed.

\begin{lemma}\label{collaring invariant}
If $Q$ is a closed subposet of a poset $P$, the property of $Q$ to be collared
in $P$ is an invariant of the homeomorphism type of the pair $(|P|,|Q|)$
\end{lemma}

\begin{proof}
Given a $q\in Q$, let $q_-$ be a maximal simplex of $\fll q\flr^\flat$.
Then $\lk(q,P)^\flat=\Lk(q_-,P^\flat)$ as subcomplexes of $P^\flat$,
and similarly $\lk(q,Q)^\flat=\Lk(q_-,Q^\flat)$.
Also $(\Lk(q_-,P^\flat),\Lk(q_-,Q^\flat))$ is isomorphic to
$(\lk(q_-,P^\flat),\lk(q_-,Q^\flat))$.
Since every simplex of $Q^\flat$ is a maximal simplex of $\fll q\flr^\flat$
for some $q\in Q$, we get that $Q$ is collared in $P$ if and only if
$Q^\flat$ is collared in $P^\flat$.

This reduces the lemma to the case where $P$ is a simplicial complex.
This case follows by the combinatorial invariance of link.
\end{proof}

If $Q$ is a closed subposet of a poset $P$, let $P^+_Q$ denote the amalgam
$P\cup_{Q=Q\x\{0\}}Q\x I$.

\begin{lemma} \label{collaring lemma}
Let $P$ be a simplicial complex and $Q$ a full subcomplex of $P$.
Then $Q$ is collared in $P$ if and only if there exists a homeomorphism
$h\:|P|\to |P^+|$ keeping $|P\but\cel Q\cer|$ fixed and extending the obvious
homeomorphism $|Q|\to |Q\x\{1\}|$.
Furthermore, $h$ sends $|\cel Q\cer|$ onto $|Q\x I|$ and $|P\but Q|$ onto
$|P|$.
\end{lemma}

\begin{proof} We first note that $Q\x\{1\}$ is collared in $Q\x I$.
Indeed, given a $q\in Q$, we have
$\lk((q,1),Q\x I)\simeq\lk(q,Q)*\lk(\{1\},I)\simeq\lk(q,Q)*pt$, and this isomorphism
sends $\lk((q,1),Q\x\{1\})$ onto $\lk(q,Q)*\emptyset$.
Hence $Q\x\{1\}$ is collared in $P^+_Q$.
The ``if'' implication now follows from Lemma \ref{collaring invariant}.

Conversely, suppose that $Q$ is collared in $P$.
Let $Q'=\cel Q\cer\but Q$ and let $R=\Fll\cel Q\cer\Flr\but\cel Q\cer$.
Since $R$ is closed in $P$, it is a simplicial complex.
(In fact, $Q'$ is a cubosimplicial complex, see Lemma \ref{cubosimplicial}.)
Since $Q$ is full in $P$, for each $p\in Q'$ the simplex $\fll p\flr$ meets
$Q$ in a simplex $\fll q_p\flr$.
Then $\fll p\flr$ meets $R$ in the opposite simplex
$\fll r_p\flr=\fll p\flr\but\Cel\fll q_p\flr\Cer$.
The resulting maps $f\:Q'\to Q$, $p\mapsto q_p$, and $g\:Q'\to R$, $p\mapsto r_p$,
are easily seen to be monotone and (using that $P$ is a simplicial complex) closed.
(In fact, they are cubical, see Lemma \ref{Hatcher is cubical}.)
It follows that $\fll Q'\flr$ is isomorphic to the double mapping cylinder
$MC(f)\cup_{Q'}MC(g)$.
Note that $P=\fll Q'\flr\cup (P\but\cel Q\cer)$ and
$\fll Q'\flr\cap (P\but\cel Q\cer)=R$.

Pick some $q\in Q$.
Let $K=\lk(q,P)\but\lk(q,Q)$ and let $L=\Cel\lk(q,Q)\Cer_{\lk(q,P)}\but\lk(q,Q)$.
Clearly, $f^{-1}(\cel q\cer)=K$ and $f^{-1}(\partial^*\cel q\cer)=L$.

Since $Q$ is collared in $P$, clearly $\lk(q,Q)$ is collared in $\lk(q,P)$.
On the other hand, $\Lk(q,Q)$ is full in $\Lk(q,P)$ since $Q$ is full in $P$,
and $\Lk(q,P)$ is a simplicial complex since $P$ is.
Since $\lk(q,Q)$ is naturally isomorphic to $\Lk(q,Q)$, we get that
$\lk(q,Q)$ is full in $\lk(q,P)$, which is a simplicial complex.
Arguing by induction, we may assume that there is a homeomorphism
$|\lk(q,P)|\to|\lk(q,P)^+_{\lk(q,Q)}|$ which sends $|K|$ onto $|\lk(q,P)|$
and $|L|$ onto $|\lk(q,Q)|$.
Hence $(|K|,|L|)\cong(|\lk(q,P)|,|\lk(q,Q)|)$.
Since $Q$ is collared in $P$, we get that $|K|\cong C|L|$ keeping $|L|$ fixed.
Thus $f^*$ is a subdivision map.

Since $MC^*(f^*)\simeq MC(f)^*$, we get using Corollary \ref{MC* of subdivision} that
the duals of the natural monotone maps
$MC(\id_{Q'})\xr{\phi} MC(f)\xr{\psi}MC(\id_Q)$
are subdivision maps.
On the other hand, the projection $Q'\x (I^*)\to Q'$ extends to a subdivision map
$MC(f)\cup_{Q'} Q'\x (I^*)\cup_{Q'}MC(g)\to MC(f)\cup_{Q'}MC(g)$.
By Theorem \ref{subdivision map}, this yields a homeomorphism
$|MC(f)\cup_{Q'} Q'\x (I^*)\cup_{Q'}MC(g)|\cong|MC(f)\cup_{Q'}MC(g)|$.
We note that if $X$ is a poset, $X\x I$ is the amalgam of two copies of $MC(\id_X)$
along the two open copies of $X$, and $X\x (I^*)$ is that along the two
closed copies of $X$ (see Lemma \ref{amalgam}).
Hence $\phi$, $f\x\id_{[2]}$ and $\psi$ combine into a monotone map
$Q'\x (I^*)\cup_{Q'} MC(f)\to MC(f)\cup_Q Q\x I$ whose dual is a
subdivision map.
By Theorem \ref{subdivision map} this yields a homeomorphism
$|Q'\x (I^*)\cup_{Q'} MC(f)|\cong |MC(f)\cup_Q Q\x I|$.
Combining it with the previous homeomorphism, we obtain the desired homeomorphism
$|P|\cong|P^+_Q|$.
\end{proof}

\begin{theorem} \label{collaring theorem}
Let $P$ be a poset and $Q$ a closed subposet of $P$.
Then $Q$ is collared in $P$ if and only if there exists a homeomorphism
$h\:|P|\to |P^+_Q|$ keeping $|P\but\cel Q\cer|$ fixed and extending the
homeomorphism $|Q|\to |Q\x\{1\}|$.
\end{theorem}

\begin{proof} By Lemma \ref{collaring invariant} we may replace $(P,Q)$
by $(P^\flat,Q^\flat)$.
Since $Q^\flat$ is full in $P^\flat$, the assertion follows from
Lemma \ref{collaring lemma}.
\end{proof}

\begin{addendum} \label{collaring addendum}
Let $P$ be a poset and $Q$ a collared closed subposet of $P$.
Further let $R$ be an open subposet of $P$, and write $S=R\cap Q$.
Then the homeomorphism $h$ of Theorem \ref{collaring theorem} sends
$|R|$ onto $|R^+_S|$ and extends the homeomorphism $|S|\to|S\x\{1\}|$.
\end{addendum}

\begin{proof}
Since $R$ is open in $P$, $\lk(r,R)=\lk(r,P)$ and $\lk(r,S)=\lk(r,Q)$
for each $r\in S$; in particular, $S$ is collared in $R$.
If follows that for each $\sigma\in S^\flat$,
$(|\lk(\sigma,P^\flat)|,|\lk(\sigma,R^\flat)|)$ is homeomorphic to the cone over
$(|\lk(\sigma,Q^\flat)|,|\lk(\sigma,S^\flat)|)$
keeping $(|\lk(\sigma,Q^\flat)|,|\lk(\sigma,S^\flat)|)$ fixed.
The proof of Lemma \ref{collaring lemma} then works to show that the subdivision map
$f^*$ defined in there is fiberwise with respect to
$\chi\:(Q^\flat)^*\to [2]$, where $\chi^{-1}(2)=(S^\flat)^*$.
The remainder of the proof closely follows that of Theorem \ref{collaring theorem}.
\end{proof}

\begin{lemma} \label{Morton1}
If $P$ and $Q$ are posets, and $P_0$ is a closed subposet of $P$
such that $|Q+P|$ is homeomorphic to $|C(Q+P_0)|$ keeping $|Q+P_0|$ fixed,
then $|P|$ is homeomorphic to $|CP_0|$ keeping $P_0$ fixed.
\end{lemma}

\begin{proof} Since $Q+P_0$ is collared in $C(Q+P_0)$, by Lemma
\ref{collaring invariant} it is also collared in $Q+P$.
Thus $|\lk(q,\,Q+P)|\cong|C\lk(q,\,Q+P_0)|$ keeping $|\lk(q,\,Q+P_0)|$ fixed
for each $q\in Q$.
On the other hand, if $q$ is a maximal element of $Q$, then
$(\lk(q,\,Q+P),\ \lk(q,\,Q+P_0))\simeq(P,P_0)$.
\end{proof}

\begin{remark}
The Armstrong--Morton argument (see \cite[p.\ 180]{Arm}) works to show
that if $(P,Q)$ and $(X,Y)$ are pairs of polyhedra such that $(S^n*P,S^n*Q)$
is homeomorphic to $(S^n*X,S^n*Y)$, then $(P,Q)$ is homeomorphic to $(X,Y)$.
(Beware that $(P,Q)$ need not be a suspension if $P$ and $Q$ are suspensions,
as shown by classical knots.)
Similarly, Morton's argument \cite[Theorem 2]{Mo} works to show that any pair
$(P,Q)$ of polyhedra factors uniquely as a ball or sphere joined to a pair
of polyhedra that is reduced in the sense that it is neither a cone pair
nor a suspension pair.
A relative version of Morton's uniqueness of factorization of reduced polyhedra
into joins (which would imply Lemma \ref{Morton1}) could be more tricky, however.
\end{remark}

\begin{lemma} \label{cojoin collars}
Let $P$ be a poset and $Q$, $R$ its closed subposets such that
$P=Q\cup R$, where $Q\cap R$ is collared in $Q$ and in $R$.
Then

(a) $Q$ is collared in $CP$;

(b) $(Q\cap R)\x CZ$ and $(Q\cap R)\x CZ\cup Q\x Z$ are collared in
$R\x CZ\cup CP\x Z$ for every poset $Z$.
\end{lemma}

\begin{proof}[Proof. (a)]
Given a $p\in Q$, we need to show that $|\lk(p,CP)|$
is homeomorphic to $C|\lk(p,Q)|$ keeping $|\lk(p,Q)|$ fixed.
If $p\notin Q\cap R$, then $\lk(p,Q)=\lk(p,P)$, and the assertion follows since $P$
is collared in $CP$.
Suppose that $p\in Q\cap R$.
Then $|\lk(p,P)|$ is homeomorphic to $\{q,r\}*|\lk(p,Q\cap R)|$ keeping
$|\lk(p,Q\cap R)|$ fixed, and this extends to a homeomorphism between
$|\lk(p,CP)|$ and $\{x\}*\{q,r\}*|\lk(p,Q\cap R)|$.
Now $\{x\}*\{q,r\}$ is homeomorphic to $C\{r\}$ keeping $\{r\}$ fixed, which
implies the assertion.
\end{proof}

\begin{proof}[(b)]
Similarly to the proof of (a), the assertion follows from the fact that
$\{x\}*\{q,r\}\cup_{\{r\}}\{r\}*\{y\}$ is homeomorphic to $C\{q\}$ keeping
$\{q\}$ fixed and to $C\{q,y\}$ keeping $\{q,y\}$ fixed.
\end{proof}

\section{Stratification maps}

We say that $f\:P\to Q$ is a {\it stratification map} if $f^{-1}(\partial\fll q\flr)$
is collared in $f^{-1}(\fll q\flr)$ for each $q\in Q$.
When $Q$ is a cell complex, stratification map is a combinatorial form of the notion
of a ``polyhedral mock bundle'' of \cite[p.\ 35]{BRS}.

Clearly, every subdivision map is a stratification map, and every stratification map
is an open map.
If $f\:P\to Q$ is a stratification map, then by the combinatorial invariance of
link $\partial P=f^{-1}(\partial Q)$.

\begin{theorem} \label{stratification map} A monotone map of posets $f\:P\to Q$ is
a stratification map if and only if for each $p\in P$ the map
$$f_p\:\lk(p,P)\to F+\lk(f(p),Q),$$
defined by the identity on $F:=\lk(p,f^{-1}(f(p)))$ and as the restriction
of $f$ elsewhere, is a subdivision map.
\end{theorem}

\begin{proof} Let us write $\lk p=\lk(p,P)$ and $\lk f(p)=\lk(f(p),Q)$.
We observe that the composition of $f_p$ and the ``projection''
$h\:F+\lk f(p)\to C^*\lk f(p)\simeq\cel f(p)\cer$ coincides with $f|_{\lk p}$.
Pick some $q\in\lk f(p)$.
Let us consider $M_{pq}:=f_p^{-1}(\fll q\flr_{F+\lk f(p)})$ and
$L_{pq}:=f_p^{-1}(\partial\fll q\flr_{F+\lk f(p)})$.
The cone $\fll q\flr_{F+\lk f(p)}=F+\fll q\flr_{\lk f(p)}$, whence
$M_{pq}=g^{-1}(h^{-1}(\fll q\flr_{C^*\lk f(p)}))=(f|_{\lk p})^{-1}(\fll q\flr_Q)$;
similarly, $L_{pq}=(f|_{\lk p})^{-1}(\partial\fll q\flr_Q)$.
Let us consider $M_q:=f^{-1}(\fll q\flr_Q)$ and
$L_q:=f^{-1}(\partial\fll q\flr_Q)$.
Then $M_{pq}=\lk(p,M_q)$ and $L_{pq}=\lk(p,L_q)$.

If $L_q$ is collared in $M_q$, then $|M_{pq}|\cong C|L_{pq}|$ keeping
$|L_{pq}|$ fixed.
This works for every $q\in\lk f(p)$; so if $f$ is a stratification map, then $f_p$ is
a subdivision map.

Conversely, if $f_p$ is a subdivision map, then $|M_{pq}|\cong C|L_{pq}|$ keeping
$|L_{pq}|$ fixed.
Hence if $f_p$ is a subdivision map for every $p\in f^{-1}(\partial\fll q\flr)$, then
$L_q$ is collared in $M_q$.
This works for every $q\in Q$; so if $f_p$ is a subdivision map for all $p\in P$,
then $f$ is a stratification map.
\end{proof}

\begin{corollary}\label{cubical dual to stratification map}
A cubical map of posets is dual to a stratification map.
\end{corollary}

\begin{proof}
If $f^*\:P^*\to Q^*$ is cubical, the restriction $\cel p\cer\to\cel f(p)\cer$
of $f$ is isomorphic to a projection of the form $CX\x CY\to CY$ for each $p\in P$.
Then $f_p$ is isomorphic to the monotone map $X*Y\to X+Y$ that restricts to the
identity map $X*\emptyset\to X$ and projects $C^*X\x Y$ onto $Y$.
Given a $y\in Y$, we have $\fll y\flr_{X+Y}=X+\fll y\flr_Y$, and consequently
$f^{-1}(\fll y\flr)=X*\fll y\flr_Y$ and
$f^{-1}(\partial\fll y\flr)=X*\partial\fll y\flr_Y$.
Let $Z=\partial\fll y\flr_Y$.
By the associativity of join of polyhedra, $|X*CZ|\cong C|X*Z|$ keeping
$|X*Z|$ fixed.
\end{proof}

\begin{theorem} \label{amalgamation}
Let $f\:P\to Q$ be a stratification map and $Z$ a closed subposet of $Q$.

(a) If $Z$ is collared in $Q$, then $f^{-1}(Z)$ is collared in $P$.

(b) When $Z\incl f(P)$, the converse holds.
\end{theorem}

\begin{proof}[Proof. (a)] Let $p\in Y:=f^{-1}(Z)$, and write
$F=\lk(p,f^{-1}(f(p)))$.
By Theorem \ref{stratification map}, $f_p\:\lk(p,P)\to F+\lk(f(p),Q)$ is
a subdivision map.
Then by Theorem \ref{subdivision map} there is a homeomorphism
$|\lk(p,P)|\to |F+\lk(f(p),Q)|$, which takes $|\lk(p,Y)|$ onto
$|F+\lk(f(p),Z)|$.
Hence $|\lk(p,P)|\cong C|\lk(p,Y)|$ keeping $|\lk(p,Y)|$ fixed if and only
if $|F+\lk(f(p),Q)|\cong C|F+\lk(f(p),Z)|$ keeping $|F+\lk(f(p),Z)|$ fixed.
The latter follows from the existence of a homeomorphism
$|\lk(f(p),Q)|\cong C|\lk(f(p),Z)|$ keeping $|\lk(f(p),Z)|$ fixed.
\end{proof}

\begin{proof}[(b)]
By Lemma \ref{Morton1}, the last inference in the proof of (a) can be reversed.
\end{proof}

\begin{corollary} \label{composition}
Let $P\xr{f}Q\xr{g}R$ be monotone maps of posets.
If $f$ and $g$ are stratification maps, then so is $gf$.
If $f$ and $gf$ are stratification maps, and $f$ is surjective,
then $g$ is a stratification map.
\end{corollary}

\begin{lemma} Let $P$ and $Q$ be posets. Then

(a) $\partial (P\x Q)=(\partial P)\x Q\cup P\x(\partial Q)$;

(b) if $\partial P$ is collared in $P$ and $\partial Q$ is collared in $Q$,
then $\partial (P\x Q)$ is collared in $P\x Q$.
\end{lemma}

Part (a), which is easy, implies in particular that
$CA\x CB\simeq C(A\x CB\cup CA\x B)$.

\begin{proof}[Proof of (b)]
We recall the formula $\lk((p,q),P\x Q)\simeq\lk(p,P)*\lk(q,Q)$.
Let $(p,q)$ be an element of $\partial (P\x Q)$.
We want to show that $|\lk(p,P)*\lk(q,Q)|\cong C|\lk((p,q),\partial(P\x Q))|$
keeping $|\lk((p,q),\partial(P\x Q))|$ fixed.

By symmetry we may assume that $p\in\partial P$.
Then $|\lk(p,P)|\cong C|\lk(p,\partial P)|$ keeping $|\lk(p,\partial P)|$ fixed.
Hence $|\lk(p,P)*\lk(q,Q)|\cong C|\lk(p,\partial P)*\lk(q,Q)|$
keeping $|\lk(p,\partial P)*\lk(q,Q)|$ fixed.
Indeed, $(CX)*Y\cong C(X*Y)$ keeping $X*Y$ fixed, due to
$(c*A)*B\simeq c*(A*B)$.

If $q\notin\partial Q$, then
$\lk((p,q),\partial(P\x Q))=\lk((p,q),(\partial P)\x Q)$ and
the assertion follows.

If $q\in\partial Q$, then $|\lk(q,Q)|\cong C|\lk(q,\partial Q)|$ keeping
$|\lk(q,\partial Q)|$ fixed.
Also $$\lk((p,q),\partial(P\x Q))=\lk((p,q),(\partial P)\x Q)
\underset{\lk((p,q),(\partial P)\x (\partial Q))}\cup \lk((p,q),P\x\partial Q).$$
The assertion now follows using that $(CX)*(CY)\cong C(CX*Y\cup_{X*Y}X*CY)$ keeping
$CX*Y\cup_{X*Y}X*CY$ fixed, which in turn is a consequence of Lemma
\ref{cojoin formula}.
\end{proof}

\begin{corollary} \label{map product}
Let $f\:P\to Q$ and $g\:R\to S$ be stratification maps or subdivision maps.
Then so is $f\x g\:P\x R\to Q\x S$.
\end{corollary}

\begin{lemma} \label{inclusion lemma}
Let $Q$ be an open subposet of a poset $P$.

(a) The inclusion $Q\emb P$ is a stratification map.

(b) A monotone map $f\:R\to Q$ is a stratification map if and only if the composition
$R\xr{f}Q\emb P$ is.

(c) If $g\:S\to P$ is a stratification map, then so is $g|\:g^{-1}(Q)\to Q$.
\end{lemma}

\begin{proof}[Proof. (a)]
Given a $p\in P$, $\partial\fll p\flr_Q$ is the boundary of the cone
$\fll p\flr_Q$, and therefore is collared in the latter.
\end{proof}

\begin{proof}[(b)]
The ``if' implication is straightforward from the definition, and the ``only if''
implication follows from Corollary \ref{composition}.
\end{proof}

\begin{proof}[(c)]
Since $Q$ is open in $P$, $g^{-1}(Q)$ is open in $S$.
Then by (a) and Corollary \ref{composition}, the composition
$g^{-1}(Q)\emb S\xr{g}P$ is a stratification map.
Then by (b) so is the restriction $g|\:g^{-1}(Q)\to Q$.
\end{proof}

\section{Filtration maps}

\subsection{Codimension one}
We say that a closed subposet $Q$ of $P$ is {\it of codimension one} in $P$
if every maximal element of $Q$ is covered by at least one maximal element
of $P$.
If additionally no maximal element of $Q$ is covered by a non-maximal element of
$P$, then we say that $Q$ is {\it of pure codimension one} in $P$.

We note that if $Q$ is collared in $P$, then the link in $P$ of every maximal
element of $Q$ is a singleton; in other words, every maximal element of $Q$ is
covered by {\it precisely} one maximal element of $P$, and by no other element
of $P$.
On the other hand, if $Q$ is of codimension one in $P$, then it is nowhere
dense in $P$, in the sense of the Alexandroff topology (i.e.\ contains no maximal
element of $P$).

\subsection{Filtration map}
We call a monotone map of posets $f\:P\to Q$ a {\it [pure] filtration map} if
$f^{-1}(\partial\fll q\flr)$ is of [pure] codimension one in $f^{-1}(\fll q\flr)$
for each $q\in Q$.
Filtration maps (an in particular stratification maps) are open, since
if $q>f(p)$, then $f^{-1}(\partial\fll q\flr)$ is contains a maximal element $p'$
such that $p'\ge p$, and this $p'$ must be covered by an element of
$f^{-1}(\fll q\flr)$, which can only be in $f^{-1}(q)$.

\begin{lemma} \label{pseudo-amalgamation}
Let $f\:P\to Q$ be a filtration map and $Z$ a closed subposet of $Q$.
If $Z$ is of codimension one in $Q$, then $f^{-1}(Z)$ is of
codimension one in $P$.
\end{lemma}

\begin{proof}
Let $p$ be a maximal element of $f^{-1}(Z)$.
The restriction of $f$ to $f^{-1}(Z)$ is a filtration map,
and in particular an open map.
Hence $f(\{p\})$ is open in $Z$, that is, $f(p)$ is maximal in $Z$.
Since $Z$ is of codimension one in $Q$, $f(p)$ is covered by
a maximal element $q$ of $Q$.
Since $f(p)$ is maximal in $Z$, $q\notin Z$.

Since $q>f(p)$, we have $p\in f^{-1}(\partial\fll q\flr)$.
Let $p'$ be a maximal element of $f^{-1}(\partial\fll q\flr)$ such that $p'\ge p$.
Since $f$ is a filtration map, $p'$ is covered by a maximal element
$p''$ of $f^{-1}(\fll q\flr)$.
Since $p'$ is maximal in $f^{-1}(\partial\fll q\flr)$,
$p''\notin f^{-1}(\partial\fll q\flr)$, that is, $f(p'')=q$.

If $p'>p$, then since $p$ is maximal in $f^{-1}(Z)$,
$p'\notin f^{-1}(Z)$, that is, $f(p')\notin Z$.
Since $f(p)\in Z$, $f(p')\ne f(p)$.
On the other hand, $p'\in f^{-1}(\partial\fll q\flr)$, so $f(p')\ne q$.
Hence $f(p)<f(p')<q$, contradicting our choice of $q$.
Thus $p'=p$, and so $p''$ covers $p$.

If $p'''\ge p''$, then $f(p''')=f(p'')$ since $f(p'')=q$ is maximal in $Q$.
Hence $p'''\in f^{-1}(\fll q\flr)$, and so $p'''=p''$ since
$p''$ is maximal in $f^{-1}(\fll q\flr)$.
Thus $p''$ is maximal in $P$.
\end{proof}

\begin{corollary} \label{pseudo-composition}
Filtration maps are closed under composition.
\end{corollary}

\begin{example} Let $P=(\{a,b,c,d\},\le)$, where the only relations are $a<b$
and $a<c<d$ and their implications.
Let $Q$ be the subposet of $P$ formed by $a$, $c$ and $d$, and let $f\:P\to Q$
be the retraction sending $b$ onto $d$.
Then $f$ is a filtration map, and $\{a\}$ is of codimension one in $P$ but
not in $Q$.
\end{example}

\begin{theorem} \label{pure pseudo-amalgamation}
Let $f\:P\to Q$ be a pure filtration map and $Z$ a closed subposet of $Q$.

(a) If $Z$ is of pure codimension one in $Q$, then $f^{-1}(Z)$ is of pure
codimension one in $P$.

(b) When $Z\incl f(P)$, the converse holds.
\end{theorem}

\begin{proof}[Proof. (a)]
Let $p$ be a maximal element of $f^{-1}(Z)$.
By Lemma \ref{pseudo-amalgamation}, $p$ is covered by a maximal
element of $P$.
Suppose that $p$ is also covered by a $p'\in P$, which in turn is covered
by a $p''\in P$.
The restriction of $f$ to $f^{-1}(Z)$ is a pure filtration map,
and in particular an open map.
Hence $f(\{p\})$ is open in $Z$, that is, $f(p)$ is maximal in $Z$.

Since $p$ is maximal in $f^{-1}(Z)$, $p'\notin f^{-1}(Z)$, that is, $f(p')\notin Z$.
Since $f(p)\in Z$, $f(p')\ne f(p)$.
Thus $f(p')>f(p)$; since $Z$ is of pure codimension one in $Q$,
$f(p')$ covers $f(p)$.

Suppose that $f(p'')=f(p')$.
Then $p'$ and $p''$ lie in $f^{-1}(R)$, where $R=\fll f(p'')\flr$.
If $p'''\in f^{-1}(\partial R)$ and $p'''\ge p$, then $f(p''')=f(p)$ since $f(p)$
is maximal in $\partial R$.
Then $p'''\in f^{-1}(Z)$, so $p'''=p$ since $p$ is maximal in $f^{-1}(Z)$.
Thus $p$ is maximal in $f^{-1}(\partial R)$.
Then the inequalities $p<p'<p''$ contradict the hypothesis that $f$ is
a pure filtration map.
Thus $f(p'')\ne f(p')$.
Then the inequalities $f(p'')>f(p')>f(p)$ contradict the hypothesis that $Z$
is of pure codimension one in $Q$.
\end{proof}

\begin{proof}[(b)] Let $q$ be a maximal element of $Z$.
Since $Z\incl f(P)$, there exists a maximal element $p$ in $f^{-1}(q)$.
Since $f^{-1}(q)$ is open in $f^{-1}(Z)$, $p$ is maximal in $f^{-1}(Z)$.
Since $f^{-1}(Z)$ is of codimension one in $P$, $p$ is covered by a maximal
element $p'$ of $P$.
Since $p$ is maximal in $f^{-1}(Z)$, $p'\notin f^{-1}(Z)$, that is,
$f(p')\notin Z$.
Since $p'\ge p$, $f(p')\ge f(p)=q$.
Since $q\in Z$, $f(p')>q$.
Since $Z$ is of pure codimension one in $Q$, $f(p')$ covers $q$.
Since $f$ is open and $p'$ is maximal in $P$, $f(p')$ is maximal in $Q$.

Suppose that $q$ is covered by a $q'\in Q$, which is in turn covered by a
$q''\in Q$.
Since $f$ is open, $q'=f(p'')$ for some $p''>p$.
By the same token $f(p')=f(p''')$ for some $p'''>p''$.
Since $p$ is maximal in $f^{-1}(Z)$, the inequalities $p<p''<p'''$ contradict
the hypothesis that $f^{-1}(Z)$ is of pure codimension one in $P$.
\end{proof}

\begin{corollary} \label{pure pseudo-composition}
Let $P\xr{f}Q\xr{g}R$ be monotone maps of posets.
If $f$ and $g$ are pure filtration maps, then so is $gf$.
If $f$ and $gf$ are pure filtration maps, and $f$ is surjective,
then $g$ is a pure filtration map.
\end{corollary}

\section{Manifolds}

\subsection{Sphere and ball} We call a poset $P$ an {\it $n$-sphere} if
$|P|$ is homeomorphic to $|\partial\Delta^{n+1}|$, and an {\it $n$-ball with
boundary} $\partial P$ if
$\partial P$ is a closed subposet of $P$ and $(|P|,|\partial P|)$ is homeomorphic to
$(|\Delta^n|,|\partial\Delta^n|)$.

We recall that a {\it PL manifold} is polyhedron $M$ where each point lies in the interior
of a closed subpolyhedron PL homeomorphic to a simplex; the {\it boundary} of $M$
consists of those points that are taken into the boundary of the simplex.
It is well-known (see \cite{Hu}) that, given a simplicial complex $K$ and its
subcomplex $L$, $|K|$ is a PL $n$-manifold with boundary $|L|$ if and only if $K$
is a {\it combinatorial $n$-manifold} with boundary $L$.
The latter means that for every $i$-simplex $\sigma\in K$, the link $\Lk(\sigma,K)$
is an $(n-i-1)$-sphere if $\sigma\notin\partial K$ and an $(n-i-1)$-ball with
boundary $\Lk(\sigma,\partial K)$ if $\sigma\in\partial K$.
It is easy to see that the case $i=0$ of this condition implies the other cases.

\subsection{Manifold}
We call a poset $P$ a {\it manifold with [co]boundary} $\partial P$ if $\partial P$
is a (possibly empty) closed [resp.\ open] subposet of $P$, and $|P|$ is
a PL manifold with boundary $|\partial P|$.
Clearly, $P$ is a manifold with boundary $Q$ if and only if $P^*$ is a manifold with
coboundary $Q^*$.

\subsection{Cell complex}
A poset $P$ is called a {\it cell complex}, if $\partial\fll p\flr$ is a sphere
of some dimension for each $p\in P$.
The cones of a cell complex are also referred to as its {\it cells}; clearly, they
are balls.
A closed subposet of a cell complex will be referred to as its {\it subcomplex}.

Affine polytopal complexes are clearly cell complexes.

\begin{lemma} \label{semi-cell complex}
Let $P$ be a poset.
Then $P^\#$ is a cell complex if and only if the open interval
$\partial\partial^* [p,q]$ is a sphere for every $p,q\in P$ with $p<q$.
\end{lemma}

\begin{proof}
If $S:=\partial\partial^*[p,q]$ is an $n$-sphere, then
$B_1:=(\partial[p,q])^\#$ and $B_2:=(\partial^*[p,q])^\#$ are $(n+1)$-balls with
common boundary $S^\#$, so $B_1\cup B_2=\partial([p,q]^\#)$ is an $(n+1)$-sphere.

Conversely, suppose that $B_1\cup B_2$ is an $m$-sphere.
Since each $|B_i|\cong C|\partial B_i|$ keeping $|\partial B_i|$ fixed,
we have $|B_1\cup B_2|\cong S^0*|S|$.
Hence by \cite{Mo} $|S|$ is an $(m-1)$-sphere.
\end{proof}

\subsection{Cell complex with coboundary}
A poset $P$ will be called a {\it cell complex with coboundary} $\partial^*P$, if
$\partial^*P$ is an open subposet of $P$, and $(\partial\fll p\flr)^*$ is a sphere
for each $p\in P\but\partial^*P$ and a ball with boundary
$(\partial\fll p\flr_{\partial^*P})^*$ for each $p\in\partial^*P$.

If $P$ is a cell complex, then $C^*P$ is a cell complex with coboundary $P$, so
the new definition of coboundary $\partial^*P$ extends the former one.
A closed subposet $K$ of a cell complex $P$ with coboundary is clearly a cell complex
with coboundary $K\cap\partial^*P$.

\begin{remark} \label{collared coboundary}
Obviously, a poset $P$ is a cell complex with coboundary $Q$ if and only if
$P$ and $P\but Q$ are cell complexes, and $Q^*$ is collared in $P^*$.
\end{remark}

\subsection{Pseudo-manifold}
A poset $P$ is called {\it purely $n$-dimensional} if every simplex of $P^\flat$
lies in an $n$-simplex; and {\it pure} if it is purely $n$-dimensional for some $n$.
(Some authors call such posets ranked or graded, but those terms may also mean
something else.)
$P$ will be called an {\it $n$-pseudo-manifold with coboundary $Q$} if $P$ is
purely $n$-dimensional, $P$ is a cell complex with coboundary $Q$, and
the $1$-skeleton $(P^*)^{(1)}$ (which consists of atoms and elements that
cover atoms) is a cell complex (in other words, a graph).
For brevity, $P^*$ will be called an {\it $n$-pseudo$^*$-manifold with boundary
$Q^*$}.

It is easy to see that if $P$ is a closed $n$-pseudo-manifold, then $C^*P$ is an
$(n+1)$-pseudo-manifold with coboundary $P$.
Dually, $C(P^*)$ is an $(n+1)$-pseudo$^*$-manifold with boundary $P^*$.

The following basic result is essentially known (see \cite[4.1]{Mc}, where the case
$\partial P=\emptyset$ is proved by a nontrivial topological argument) but for
completeness we spell out the details of a trivial combinatorial argument.

\begin{theorem}\label{manifolds} Let $P$ be a poset and $Q$ a closed subposet of $P$.
Then $P$ is a manifold with boundary $Q$ if and only if $P$ is pure, $P$ is
a cell complex and $P^*$ is a cell complex with coboundary $Q^*$.
\end{theorem}

\begin{proof}
Suppose that $P$ is a manifold with boundary $\partial P$.
By definition this implies that the barycentric subdivision $P^\flat$ is
a combinatorial manifold with boundary $Q^\flat$.
Given a $p\in P$, let $p_-$ be a maximal simplex of $\fll p\flr^\flat$, and $p_+$
a maximal simplex of $\cel p\cer^\flat$.
Then $\lk(p,P)^\flat=\Lk(p_-,P^\flat)$ and $\lk(p^*,P^*)^\flat=\Lk(p_+,P^\flat)$
as subcomplexes of $P^\flat$.
When $p\in Q$, the same equations also hold with $Q$ in place of $P$.

Since $Q$ is closed in $P$, we have $p_-\in Q^\flat$ if and only if
$p\in Q$; whereas $p_+\notin Q^\flat$ for all $p\in P$.
Hence $\lk(p^*,P^*)$ is a sphere for all $p\in P$; whereas $\lk(p,P)$ is a sphere if
$p\notin Q$, and a ball with boundary $\lk(p,Q)$ if $p\in Q$.
Since $\lk(p^*,P^*)=(\partial\fll p\flr_P)^*$ and
$\lk(p,P)=(\partial\fll p^*\flr_{P^*})^*$, we conclude that $P$ is a cell complex
and $P^*$ is a cell complex with coboundary $Q^*$.

Conversely, suppose that $P$ is a cell complex and $P^*$ is a cell complex with
coboundary $Q^*$.
For each $p\in P$, we have $\Lk(\{p\},P^\flat)=\lk(p,P)^\flat*\lk(p^*,P^*)^\flat$
as subcomplexes of $P^\flat$.
When $p\in Q$, the same equation also holds with $Q$ in place of $P$.
Since the join of two PL spheres is a PL sphere, and the join of a PL ball with
a PL sphere is a PL ball with appropriate boundary, it follows that for each
vertex $v$ of $P^\flat$, the link $\Lk(v,P^\flat)$ is a sphere if $v\notin Q^\flat$,
and a ball with boundary $\Lk(v, Q^\flat)$ if $v\in Q^\flat$.

The hypothesis that $P$ is pure ensures that different connected components
of $|P|$ are PL manifolds of the same dimension.
\end{proof}

\begin{corollary} Spheres, balls and other manifolds are cell complexes.
\end{corollary}

Since balls are cell complexes, we also get

\begin{corollary} \label{cell complex subdivision}
A subdivision of a cell complex is a cell complex.
\end{corollary}

\begin{lemma}\label{coboundary recognition}
A poset $P$ is a cell complex with coboundary $Q$ if and only if
$(\partial\fll p\flr)^*$ is a sphere for all $p\in P\but Q$ and a ball for all
$p\in Q$.
\end{lemma}

\begin{proof}
The ``only if'' direction is trivial.
Conversely, by the hypothesis $P\but Q$ is a cell complex.
So it is closed in $P$, whence $Q$ is open in $P$.

Pick some $q\in Q$, and let $B=\partial\fll q\flr_P$.
Then $B$ is a subcomplex of $P$, and since $B^*$ is a ball,
by Theorem \ref{manifolds}, $B$ is a cell complex with coboundary.
If $p\in B\but\partial^*B$, then $\partial\fll p\flr_P=\partial\fll p\flr_B$ is
a sphere, so $p\notin Q$.
If $p\in\partial^*B$, then $(\partial\fll p\flr_P)^*=(\partial\fll p\flr_B)^*$
is a ball, so $p\in Q$.
Thus $B\cap Q=\partial^*B$.
In other words, $B^*$ is a ball with boundary $(\partial\fll q\flr_Q)^*$, as desired.
\end{proof}

\begin{lemma}\label{coboundary capped off}
A poset $P$ is a cell complex with coboundary $Q$ if and only if the amalgam
$P\cup_Q C^*Q$ is a cell complex.
\end{lemma}

\begin{proof} Let $q\in Q$ and write $P\div Q=P\cup_Q C^*Q$.
Then $\partial\fll q\flr_{P^+}\simeq\partial\fll q\flr_P\div\partial\fll q\flr_Q$.
So it suffices to show that a poset $B$ is a ball if and only if
$B^+:=(B^*\div(\partial B)^*)^*$ is a sphere.
We have $B^+\simeq B\cup_{\partial B}C(\partial B)$.
If $B$ is a ball, then so is $C(\partial B)$, hence $B^+$ is a sphere.
If $B^+$ is a sphere, then by Theorem \ref{manifolds} it is a cell complex.
Then $C(\partial B)$ is a ball, hence $B$ is a ball.
\end{proof}

\begin{corollary} \label{cell complex with coboundary subdivision}
If $\alpha\:K'\to K$ is a subdivision map and $K$ is a cell complex with coboundary,
then $K'$ is a cell complex with coboundary $\alpha^{-1}(\partial^*K)$.
\end{corollary}

\begin{remark} A {\it regular CW complex}, defined by using attaching 
maps that are continuous (and not necessarily PL) embeddings, need not
come from any cell complex.
Indeed, let $H$ be a simplicial complex triangulating 
a non-simply-connected homology $3$-sphere. 
Let $X$ be the regular CW-complex obtained by attaching $B^6$ to 
the double suspension $\Sigma^2(H\sqcup pt)$ along a homeomorphism
$\partial B^6\to \Sigma^2 H$ given by the Edwards--Cannon theorem.
Then $X$ is obtained by attaching $B^2$ to $B^6$ along a wild embedding
$\partial B^2\to\partial B^6$, and therefore cannot be obtained 
from any cell complex by forgetting its PL structure.
Note, however, that the barycentric subdivision of $X$ 
is a genuine simplicial complex. 
\end{remark}

\section{Comanifolds}

\subsection{Cellular map}
A monotone map of posets $f\:P\to Q$ will be called {\it cellular}, if
$f^{-1}(\cel q\cer)$ is a cell complex with coboundary
$f^{-1}(\partial^*\cel q\cer)$ for each $q\in Q$.

\begin{theorem}\label{partition} Let $f\:K\to L$ be a simplicial map of simplicial
complexes.
Then $f$ is cellular.
\end{theorem}

\begin{proof}
Let us fix a $\sigma\in L$, and let $D=f^{-1}(\cel\sigma\cer)$ and
$E=f^{-1}(\partial^*\cel\sigma\cer)$.
Pick some $\tau\in D$.
Then $\fll\tau\flr=\fll\mu\flr*\fll\nu\flr$, where
$f(\mu)=\sigma$ and $f(\nu)\cap\sigma=\emptyset$.
We note that $\nu\ne\emptyset$ if and only if $\tau\in E$.
The isomorphism $\fll\tau\flr_K\simeq\fll\mu\flr_K*\fll\nu\flr_K$ can be equivalently
written as $C^*\fll\tau\flr_K\simeq C^*\fll\mu\flr_K\x C^*\fll\nu\flr_K$.
The latter restricts to an isomorphism
$\fll\tau\flr_D\simeq\fll\mu\flr_D\x C^*\fll\nu\flr_K$,
given by $\tau'\mapsto(\mu',\nu')$, where $\fll\tau'\flr=\fll\mu'\flr*\fll\nu'\flr$.
By Lemma \ref{cubosimplicial} $\fll\mu\flr_D=CR^*$, where $R$ is a sphere
(in fact, a join of boundaries of simplices).

If $\nu=\emptyset$, then $\fll\tau\flr_D=\fll\mu\flr_D=CR^*$,
hence $(\partial\fll\tau\flr_D)^*=R$, which is a sphere.
If $\nu\ne\emptyset$, then $\fll\tau\flr_D\simeq CR^*\x C^*\fll\nu\flr_K$.
We may write $\fll\nu\flr_K=CQ^*$, where $Q$ is a sphere (in fact, the boundary of
a simplex).
By the associativity of prejoin $C^*(CQ^*)\simeq C(C^*Q^*)\simeq C(CQ)^*$.
Hence $\fll\tau\flr_D\simeq CR^*\x C(CQ)^*$.
Then $\fll\tau\flr_D^*\simeq C^*R\x C^*(CQ)\simeq C^*(R*CQ)$.
Therefore $(\partial\fll\tau\flr_D)^*\simeq R*CQ$, which is a ball.
Thus by Lemma \ref{coboundary recognition}, $D$ is a cell complex with
coboundary $E$.
\end{proof}

By Remark \ref{collared coboundary} if $f$ is cellular, then $f^*\:P^*\to Q^*$ is
a stratification map.
This yields the following corollary of Theorem \ref{partition}, which was known
in different terms \cite[proof of Lemma 2.2]{Co3} (see also \cite[p.\ 35]{BRS}):

\begin{corollary}\label{simplicial dual to stratification map}
If $f\:K\to L$ is a simplicial map between simplicial complexes,
then $f^*\:K^*\to L^*$ is a stratification map.
\end{corollary}

\subsection{Comanifold}
A monotone map of posets $f\:P\to Q$, will be called a {\it [pseudo-]comanifold}
if $f^{-1}(\fll\sigma\flr)$ is a [pseudo$^*$-]manifold with boundary
$f^{-1}(\partial\fll\sigma\flr)$ for each $\sigma\in Q$.

Depending on the structure of $Q$, a [pseudo-]comanifold $f\:P\to Q$ may have
dimension and/or codimension.
If $Q^*$ has pure cones, then $f$ is of {\it dimension} $d$ if
$f^{-1}(\fll\sigma\flr)$ is of dimension $d-i$ for each $\sigma\in Q$ with
$\fll\sigma^*\flr$ of dimension $i$.

More important is codimension.
If $Q$ itself has pure cones, then $f$ of {\it codimension} $k$, or simply is
a {\it $k$-[pseudo-]comanifold}, if $f^{-1}(\fll\sigma\flr)$ is of dimension
$i-k$ for each $\sigma\in Q$ with $\fll\sigma\flr$ of dimension $i$.

When $Q$ is a cell complex, $k$-comanifolds are a combinatorial form of
``$k$-mock bundles'', introduced by Buoncristiano, Rourke and Sanderson;
$k$-[pseudo-]comanifolds work as geometric representatives of $k$-dimensional
unoriented PL cobordism [resp.\ $\bmod 2$ cohomology] classes of $|Q|$ \cite{BRS}.

When $Q$ is purely $n$-dimensional, both dimension $d$ and codimension $k$ are
defined for [pseudo-]manifolds into $Q$; obviously, $k=n-d$.

If $f$ is a pseudo-comanifold, then $f^*$ is a cellular map, and in
particular $f$ itself is a stratification map.
When $P$ is a cell complex, any subdivision of $P$ is a comanifold.

\begin{corollary} \label{comanifolds} Let $f\:M\to N$ be a simplicial map of simplicial
complexes.
If $M$ is an $m$-[pseudo-]manifold, then $f^*\:M^*\to N^*$ is a [pseudo-]comanifold
of dimension $m$.
\end{corollary}

The manifold part is essentially Cohen's theorem \cite[5.6]{Co1};
concerning the pseudo-manifold part see \cite{Gra} and \cite{Do3}.

When $N$ is purely $n$-dimensional, we can add to Corollary \ref{comanifolds} that
$f^*$ is an $(n-m)$-[pseudo-]comanifold; and if additionally $N^*$ is a
cell complex (which is equivalent to $N$ being a manifold), then $f^*$
represents an $(n-m)$-dimensional cobordism [resp.\ cohomology] class
$[f^*]$ of $|N|$.
One more application of Corollary \ref{comanifolds} ensures that $[f^*]$ depends
only on the $m$-dimensional bordism [resp.\ homology] class $f_*([M])$.
Thus Corollary \ref{comanifolds} is really a combinatorial description of
the Poincar\'e duality homomorphism.
(A geometric proof that this homomorphism is an isomorphism involves
a transversality argument, see \cite{BRS}.)

\begin{proof}
Let $\sigma\in N$.
By Theorem \ref{partition}, $D:=f^{-1}(\cel\sigma\cer)$ is a cell complex with
coboundary $E:=f^{-1}(\partial^*\cel\sigma\cer)$.
On the other hand, since $M$ is a [pseudo-]manifold, $M^*$ [resp.\ $(M^*)^{(1)}$]
is a cell complex.
Since $\cel\sigma^*\cer$ is closed in $N^*$, and $f^*\:Y^*\to X^*$ is monotone,
$D^*=(f^*)^{-1}(\cel\sigma^*\cer)$ is closed in $M^*$.
Hence $D^*$ [resp.\ $(D^*)^{(1)}$] is a cell complex.
Thus $D$ is a [pseudo-]manifold.
The assertion on its dimension follows by the proofs of Lemma \ref{cubosimplicial} and
Theorem \ref{partition}.
\end{proof}

\begin{remark} \label{comanifold vs cellular}
It is implicit in the above proof that a monotone map $f\:P\to Q$ is
a [pseudo-]comanifold if and only if $P$ [resp.\ $P^{(1)}$] is a cell complex
and $f^*\:P^*\to Q^*$ is cellular.
\end{remark}

\begin{example} \label{projection example}
Let $\pi\:P\x Q\to P$ be the projection.
Then

\begin{itemize}
\item $\pi$ and $\pi^*$ are stratification maps;
\item if $P$ and $Q$ are cell complexes, then $\pi$ is cellular;
\item if $P$ and $Q$ are manifolds, then $\pi$ and $\pi^*$ are comanifolds.
\end{itemize}

The assertions on $\pi^*$ follow from the corresponding assertions
on $\pi$ since $\pi^*\:P^*\x Q^*\to P^*$ is of the same form as $\pi$.
We now prove the three assertions on $\pi$.
The first assertion follows from Corollary \ref{map product}, since
identity and constant maps are obviously stratification maps.
To prove the second [resp.\ third] assertion, it suffices to note that if $K$ is
a cell complex [resp.\ manifold] with coboundary $\partial K$, then clearly
$K\x Q$ is a cell complex [resp.\ manifold] with coboundary $\partial K\x Q$.
\end{example}

\begin{example} \label{inclusion example}
Let $\iota\:Q\emb P$ be an embedding onto a closed subposet.
Then

\begin{itemize}
\item $\iota^*$ is a stratification map;
\item if $Q^\#$ is a cell complex, then $\iota$ is cellular;
\item if $Q$ is a manifold, then $\iota^*$ is a comanifold.
\end{itemize}

We will identify $Q$ with its image in $P$.
The first assertion was proved in Lemma \ref{inclusion lemma}(a).
If $Q$ is additionally a cell complex, and $q\in\partial^*\cel p\cer_Q$, then
$\partial\fll q\flr_{\cel p\cer_Q}=\partial [p,q]$.
By Lemma \ref{semi-cell complex} the latter is the dual cone over a
sphere.
Hence its dual is a ball, so $\cel p\cer_Q$ is a cell complex with
coboundary $\partial^*\cel p\cer_Q$.
Finally, if $Q$ is a manifold, then $\partial^*\cel p\cer_Q$ is a sphere,
so $\fll p^*\flr_{Q^*}$ is a ball (in particular, a manifold) with boundary
$\partial\fll p^*\flr_{Q^*}$.
\end{example}

Corollary \ref{comanifolds} shows how to get a comanifold from a manifold.
Theorem \ref{amalgamation2}(b) below is a converse result, producing a manifold
from a comanifold (even in a slightly more general setup).

\begin{theorem} \label{amalgamation2}
(a) Let $f\:P\to Q$ be a cellular map, and $Z$ an open subposet of $Q$.
If $Q$ is a cell complex with coboundary $Z$, then $P$ is a cell complex with
coboundary $f^{-1}(Z)$; when $f$ is surjective, the converse holds.

(b) Let $f\:M\to N$ be a [pseudo-]comanifold of dimension $m$, where $N^*$ is
a cell complex with coboundary $Z^*$.
Then $M$ is an $m$-[pseudo$^*$-]manifold with boundary $f^{-1}(Z)$.
\end{theorem}

In the case where $N$ is a manifold with boundary, the manifold part of (c)
is essentially the amalgamation lemma of Buoncristiano--Rourke--Sanderson
\cite[Lemma II.1.2]{BRS}.

\begin{proof}[Proof. (a)]
By Remark \ref{collared coboundary}, $f^*\:P^*\to Q^*$ is a stratification map,
and $f^{-1}(q)$ is a cell complex for each $q\in Q$.
Let $p\in Y:=f^{-1}(Z)$, and write $F=\partial\fll p\flr_{f^{-1}(f(p))}$.
By Theorem \ref{stratification map} (applied to $f^*$),
$(f^*)_p\:(\partial\fll p\flr_P)^*\to F^*+(\partial\fll f(p)\flr_Q)^*$ is
a subdivision map.
Then by Theorem \ref{subdivision map} there is a homeomorphism
$|\partial\fll p\flr_P|\to |\partial\fll f(p)\flr_Q+F|$.
By the hypothesis, $F$ is a sphere.
Hence by \cite{Mo}, $|\partial\fll p\flr_P|$ is a sphere [resp.\ a ball]
if and only if $|\partial\fll f(p)\flr_Q|$ is a sphere [resp.\ a ball].
The assertion now follows from Lemma \ref{coboundary recognition}.
\end{proof}

\begin{proof}[(b)] This follows immediately from (b) and Remark
\ref{comanifold vs cellular}.
\end{proof}

\begin{corollary} \label{composition2}
Let $P\xr{f}Q\xr{g}R$ be monotone maps of posets.

(a) If $f$ and $g$ are cellular, then so is $gf$.
If $f$ and $gf$ are cellular, and $f$ is surjective,
then $g$ is cellular.

(b) If $f$ and $g$ are comanifolds, then so is $gf$.
\end{corollary}

\subsection{Quasi-cellular maps}
Let us call a monotone map $f\:P\to Q$ {\it quasi-cellular} if $f^*$ is
a stratification map and $f^{-1}(q)$ is a cell complex for each $q\in Q$.
By Remark \ref{collared coboundary}, cellular maps are quasi-cellular.

\begin{remark} \label{quasi-cellular}
We note that the proof of Theorem \ref{amalgamation2}(a) works under the
weaker hypothesis that $f$ is quasi-cellular.
\end{remark}

From Remark \ref{quasi-cellular} (in the case $Z=\emptyset$) and
Corollary \ref{composition} we obtain

\begin{corollary}
Let $P\xr{f}Q\xr{g}R$ be monotone maps of posets.
If $f$ and $g$ are quasi-cellular, then so is $gf$.
If $f$ and $gf$ are quasi-cellular, and $f$ is surjective,
then $g$ is quasi-cellular.
\end{corollary}

The following lemma should be compared with Remark \ref{comanifold vs cellular}.

\begin{lemma} \label{cellular vs quasi-cellular}
(a) A monotone map $f\:P\to Q$ is cellular if and only if $f$ is
quasi-cellular and $Q^\#$ is a cell complex.

(b) If $f\:P\to Q$ is a cellular map, then $P^\#$ is a cell complex.
\end{lemma}

\begin{proof}[Proof. (a)] It follows from Lemma \ref{semi-cell complex} that
$P^\#$ is a cell complex  if and only if $\lk(q,Q)$ is a cell complex for each
$q\in Q$.
On the other hand, by Remark \ref{collared coboundary}, $f$ is cellular if and
only if $f$ is quasi-cellular and $f^{-1}(\lk(q,Q))$ is a cell complex for each
$q\in Q$.
Thus it suffices to show the following: assume that $f$ is quasi-cellular and
$q\in Q$; then $\lk(q,Q)$ is a cell complex if and only if $f^{-1}(\lk(q,Q))$ is.

But if $f$ is quasi-cellular, then so is its restriction to $f^{-1}(\lk(q,Q))$.
The assertion now follows from Remark \ref{quasi-cellular}.
\end{proof}

\begin{proof}[(b)]
By Lemma \ref{semi-cell complex} it suffices to show that every open
interval $\partial^*\partial[p,q]$ is a sphere.
Let $p,q\in P$ with $p<q$.
Since $f$ is cellular, $R:=f^{-1}(\cel f(p)\cer)$ is a cell complex with
coboundary.
Then $S:=(\partial\fll q\flr_R)^*$ is either a sphere or a ball; in any case,
it is a cell complex.
So $\partial\fll p^*\flr_S$ is a sphere.
Its dual $(\partial\fll p^*\flr_S)^*$, which is also a sphere, is isomorphic to
$\partial^*\cel p\cer_{S^*}=\partial^*\partial[p,q]$.
\end{proof}

\section{Multi-collaring}

Given a monotone map $f\:P\to Q$, the simplicial map $f^\flat\:P^\flat\to Q^\flat$
is dual to a monotone map of the barycentric handle decompositions $H_f\:H(P)\to H(Q)$.
On the other hand, the subdivision map $\flat\:P^\flat\to P$ is dual to a monotone
map $\bar R_P\:H(P)\to P^*$.
The composition $P^\flat\simeq (P^*)^\flat\xr{\flat}P^*$ is dual to a
monotone map $R_P\:H(P)\to P$.

\begin{theorem}\label{grothendieck}
Let $f\:P\to Q$ be a stratification map.
Then for each $q\in Q$ there exists a commutative diagram
$$\begin{CD}
H_f^{-1}(H_q^Q)@>\alpha_q>>H_f^{-1}(H_q^{\fll q\flr})\x H_q^{\cel q\cer}\\
@VV{H_f|}V@VV{(H_f|)\x\id}V\\
H_q^Q@>\simeq>>H_q^{\fll q\flr}\x H_q^{\cel q\cer},
\end{CD}$$
where $\alpha_q$ is a fiberwise subdivision map with respect to the composition
$H_f^{-1}(H_q^{\fll q\flr})\x H_q^{\cel q\cer}\xr{\pi}H_q^{\cel q\cer}
\xr{\bar R_Q|}\cel q\cer^*$.
\end{theorem}

\begin{remark} Theorem \ref{grothendieck} implies that a stratification map gives rise
to a PL variety filtration in the sense of D. Stone \cite{Sto}, and in particular
to a locally conelike TOP stratified set in the sense of Siebenmann \cite{Sie}.
As discussed in \cite{Sto} and \cite{Sie}, these generalize Whitney stratified
spaces and Thom's topologically stratified spaces; see also Grothendieck's
reflections on ``tame topology'' and ``d\'evissage''
\cite[\S\S5,6 and endnote (6)]{Gro}.
\end{remark}

\begin{proof}
If $Q=\fll q\flr$ or more generally $\cel q\cer_Q=\{q\}$, we let $\alpha=\id$.
Arguing by induction, we may assume that the assertion is known for
$Q':=Q\but\{r\}$ and $P':=f^{-1}(Q')$ and $f':=f|_{P'}$, for some maximal
element $r$ of $Q$.
We may assume that $r>q$ (otherwise the inductive step is trivial).

Let $S=f^{-1}(\fll r\flr)$.
Since $f|_S$ is a stratification map, $\partial S=f^{-1}(\partial\fll r\flr)$.
By the combinatorial invariance of link,
$(\partial S)^\flat=\partial (S^\flat)$ so we may omit the brackets.
Since $\partial S^\flat$ is a collared full subcomplex of $S^\flat$, by
the proof of Lemma \ref{collaring lemma},
$\cel\partial S^\flat\cer\simeq MC(\beta^*)$, where
$\beta^*\:\cel\partial S^\flat\cer\but\partial S^\flat\to\partial S^\flat$
is dual to a subdivision map.

From the construction of $\beta$ it is easy to see that
$\beta^{-1}(H_f^{-1}(H_q^{\partial\fll r\flr}))=
H_f^{-1}(H_q^{\fll r\flr}\but H_q^{\partial\fll r\flr})$.
It follows that $H_f^{-1}(H_q^{\fll r\flr})\simeq MC^*(\beta_q)$, where
$\beta_q\:H_f^{-1}(H_q^{\fll r\flr}\but H_q^{\partial\fll r\flr})\to
H_f^{-1}(H_q^{\partial\fll r\flr})$ is the restriction of $\beta$.
Then Corollary \ref{MC* of subdivision} yields a subdivision map
$\alpha_1^r\:H_f^{-1}(H_q^{\fll r\flr})\to H_f^{-1}(H_q^{\partial\fll r\flr})\x[2]$
over $[2]$.
In the case that $f=\id$, this subdivision map in the domain specializes
to an isomorphism in the range,
$H_q^{\fll r\flr}\simeq H_q^{\partial\fll r\flr}\x[2]$,
which is the identity on $H_q^{\partial\fll r\flr}=H_q^{\partial\fll r\flr}\x\{2\}$.
Since $[q,r]=\fll r\flr_{\cel q\cer}$, the latter isomorphism in turn specializes to
$H_q^{[q,r]}\simeq H_q^{\partial[q,r]}\x[2]$.

Since $\alpha_1^r$ is the identity on $H_f^{-1}(H_q^{\partial\fll r\flr})$, it
extends by the identity on $H_f^{-1}(H_q^{Q'})$ to a subdivision map
$$\alpha_1\:H_f^{-1}(H_q^Q)\to H_f^{-1}(H_q^{Q'})
\underset{H_f^{-1}(H_q^{\partial\fll r\flr})}\cup
H_f^{-1}(H_q^{\partial\fll r\flr})\x[2]$$
over $\left(\cel q\cer_{Q'}\cup_{\partial[q,r]}\partial[q,r]\x[2]\right)^*$.
In more detail, the subdivision map is fiberwise with respect to the obvious map extending
the composition
$$H_f^{-1}(H_q^{Q'})\xr{H_f|}H_q^{Q'}\simeq H_q^{\fll q\flr}\x H_q^{\cel q\cer_{Q'}}
\xr{\pi}H_q^{\cel q\cer_{Q'}}\xr{\bar R_{Q'}|}\cel q\cer_{Q'}^*.$$

On the other hand, the inductive hypothesis yields a subdivision map
$\alpha_2'\:H_f^{-1}(H_q^{Q'})\to
H_f^{-1}(H_q^{\fll q\flr})\x H_q^{\cel q\cer_{Q'}}$ over $\cel q\cer_{Q'}^*$.
By Lemma \ref{fiberwise} the latter restricts to a subdivision map
$\alpha_2^r\:H_f^{-1}(H_q^{\partial\fll r\flr})\to
H_f^{-1}(H_q^{\fll q\flr})\x H_q^{\partial[q,r]}$ over $(\partial[q,r])^*$.
Then $\alpha_2'$ and $\alpha_2^r\x\id_{[2]}$ combine into a subdivision map
$$\alpha_2\:H_f^{-1}(H_q^{Q'})\underset{H_f^{-1}(H_q^{\partial\fll r\flr})}\cup
H_f^{-1}(H_q^{\partial\fll r\flr})\x[2]
\to H_f^{-1}(H_q^{\fll q\flr})\x\left(H_q^{\cel q\cer_{Q'}}
\underset{H_q^{\partial [q,r]}}\cup
H_q^{\partial [q,r]}\x[2]\right)$$
over $\left(\cel q\cer_{Q'}\cup_{\partial[q,r]}\partial[q,r]\x[2]\right)^*$.
The range of $\alpha_2'$ is isomorphic to
$H_f^{-1}(H_q^{\fll q\flr})\x H_q^{\cel q\cer_Q}$
over the composition
$\left(\cel q\cer_{Q'}\cup_{\partial[q,r]}\partial[q,r]\x[2]\right)^*\to
\left(\cel q\cer_{Q'}\cup_{\partial[q,r]}[q,r]\right)^*\simeq\cel q\cer_{Q}^*$.
Hence $\alpha_1$ followed by $\alpha_2$ is a subdivision map
$H_f^{-1}(H_q^Q)\to H_f^{-1}(H_q^{\fll q\flr})\x H_q^{\cel q\cer_Q}$
over $\cel q\cer_{Q}^*$, as desired.
The commutativity of the diagram is clear from the construction.
\end{proof}

\begin{corollary} \label{handle subdivision}
If $f\:P\to Q$ is a subdivision map, then so is $H_f\:H(P)\to H(Q)$.
\end{corollary}

\begin{proof}
Let $q\in Q$, and write $R=\fll q\flr$.
We have homeomorphisms
$$(|H_f^{-1}(H_q^R)|,|H_f^{-1}(\partial H_q^R)|)\cong
(|f^{-1}(R)|,|f^{-1}(\partial R)|)\cong (|R|,|\partial R|)\cong
(|H_q^R|,|\partial H_q^R|),$$
the first and the third one being given by Lemma \ref{collaring lemma}, and
the second one by the hypothesis.
The composition of these three homeomorphisms along with Theorem
\ref{grothendieck} yield a homeomorphism
$(|H_f^{-1}(H_q^Q)|,|H_f^{-1}(\partial H_q^Q)|)\cong
(|H_q^Q|,|\partial H_q^Q|)$, as desired.
\end{proof}

\section{Pullback} \label{transversality}

McCrory has shown that PL transversality of subpolyhedra is a symmetric relation
\cite{Mc}: subpolyhedra $X$ and $Y$ of a closed PL manifold $M$ are
PL transverse if and only if $(M,X,Y)$ is PL homeomorphic to $(P,Q,R)$, where
$P$ is a poset, $Q$ a closed subposet of $P$ and $R$ an open subposet of $P$.

On the other hand, Buoncristiano, Rourke and Sanderson implicitly extended
the notion of PL transversality to pairs of PL maps \cite[``Induced mock bundles''
(p.\ 23) and ``Extension to polyhedra'' (p.\ 35)]{BRS}: PL maps $f\:X\to Z$
and $g\:Y\to Z$ between polyhedra may be called PL transverse if $f$ and $g$
admit triangulation by simplicial maps $F\:P\to K$ and $G\:Q\to K'$ such that
$K'$ is an affine simplicial subdivision of $K$, and the composition
$Q\xr{G}K'\xr{\alpha}K$, where $\alpha$ is the subdivision map, is a
stratification map.

The following example shows that McCrory's symmetric definition does not extend
to maps, i.e.\ one cannot generalize the simplicial map $F$ to an arbitrary
monotone map whose dual is a stratification map.

\begin{example} \label{pullback fails}
Let $\alpha\:(\partial\Delta^1)*(\partial\Delta^2)
\to\partial\Delta^3$ be the stellar subdivision map of a facet $F$ of the boundary
of the tetrahedron.
Up to isomorphism, $\alpha^*$ is the same as the quotient map
$f\:\partial(\Delta^1\x\Delta^2)\to\partial\Delta^3$, shrinking $\{a\}\x\Delta^2$
to a vertex $p$ of $\Delta^3$, for some vertex $a$ of $\Delta^1$.

Now let us identify the range of $f$ with the range of $\alpha$ (rather than
the dual of the range of $\alpha$ as we did before) so that $p$ is identified
with a vertex of $F$.
Then we may consider the pullback diagram
$$\begin{CD}
X@>\beta>>\partial(\Delta^1\x\Delta^2)\\
@VgVV@VfVV\\
(\partial\Delta^1)*(\partial\Delta^2)@>\alpha>>\Delta^3.
\end{CD}$$
Then $\beta$ is not a subdivision map, and not even a stratification map; by symmetry,
it follows then that $g^*$ is also not a stratification map.

Indeed, let us consider the restriction of the above diagram over
the $2$-simplex $F$.
Up to isomorphism, it is the same as the following pullback diagram:
$$\begin{CD}
X_0@>\beta_0>>\Delta^1\x\Delta^1\\
@Vg_0VV@Vf_0VV\\
\Delta^0*\partial\Delta^2@>\alpha_0>>\Delta^2.
\end{CD}$$
Here $\alpha_0$ is the stellar subdivision map of the $2$-simplex, and $f_0$ is
the quotient map, shrinking one side of the quadrilateral $\Delta^1\x\Delta^1$
onto the vertex $p$ of the triangle $\Delta^2$.

Let $q\in\Delta^2$ be the maximal element.
Then $f_0^{-1}(p)\simeq\Delta^1$ and $\alpha_0^{-1}(q)\simeq (\Delta^2)^*$.
On the other hand, $f_0^{-1}(q)$ is a singleton $\{q'\}$ and
$\alpha_0^{-1}(p)$ is a singleton $\{p'\}$.
Let $Q=\cel p'\cer\cap\alpha_0^{-1}(q)$; thus $Q\simeq(\Delta^1)^*$.

For each pair $(\bar p,\bar q)\in f_0^{-1}(p)\x\alpha_0^{-1}(q)$ we have
$f_0(\bar p)=p=\alpha_0(p')$ and $f_0(q')=q=\alpha_0(\bar q)$.
Hence $(\bar p,p')\in X_0$ and $(q',\bar q)\in X_0$.
Moreover, $\bar p<q'$ in $\Delta^1\x\Delta^1$, and if $\bar q\in Q$, then
$p'<\bar q$ in $\Delta^3$.
Hence $(\bar p,p')<(q',\bar q)$ in $X$ as long as $\bar q\in Q$.
Therefore $X_0$ contains a copy of the prejoin $\Delta^1+(\Delta^1)^*$.
Then $|X_0|$ is $3$-dimensional, and so $\beta_0$ cannot be a subdivision map
by Theorem \ref{subdivision map}.
Also $\beta_0$ is not a stratification map, since $\Delta^1$ is not
collared in $\Delta^1+(\Delta^1)^*$.
Then also $\beta$ cannot be a stratification map since $\beta_0$ is its restriction
over a closed subposet.
\end{example}

The remainder of this chapter is concerned with finding ways to avoid
the above example.

\begin{lemma} \label{cubical pullback}
Consider a pullback diagram
$$\begin{CD}
P'@>\beta>>P\\
@Vf'VV@VfVV\\
Q'@>\alpha>>Q.
\end{CD}$$
If $f$ is cubical, then so is $f'$.
If additionally $\alpha$ is a subdivision map or a stratification map, then so is $\beta$.
\end{lemma}

\begin{proof}
Given a $p'=(q',p)\in P'$, we have $q'\in\alpha^{-1}(f(p))$, and consequently
$\fll p'\flr\subset\alpha^{-1}(\fll f(p)\flr)\x\fll p\flr$.
Hence in order to prove that $f'$ is cubical, we may assume without loss of
generality that $P=\fll p\flr$ and $Q=\fll f(p)\flr$.
Clearly, we may also assume this without loss of generality in proving
the assertion on $\beta$.

Then $f$ is isomorphic to the projection $Q\x CR\to Q$ for some poset $R$.
Hence $f'$ is isomorphic to the projection $Q'\x CR\to Q'$, and therefore
is cubical.
On the other hand, $\beta$ is isomorphic to $\alpha\x\id_{CR}\:Q'\x CR\to Q\x CR$.
If $\alpha$ is a stratification map, then $\partial Q'=\alpha^{-1}(\partial Q)$.
The assertion on $\beta$ now follows from Corollary \ref{map product}.
\end{proof}

\begin{corollary}\label{transversality1}
Consider the following diagram consisting of three pullback squares:
$$\begin{CD}
Z@>h>>Y@>\bar\rho_Y>>R\\
@V\chi VV@V\psi VV@V\phi VV\\
X@>g>>h(Q)@>\bar r_Q>>Q^*\\
@V\rho_XVV@Vr_QVV\\
P@>f>>Q.
\end{CD}$$
Then $r_Q$, $\bar r_Q$, $\rho_X$ and $\rho_Y$ as well as both vertical compositions
and both horizontal compositions in the diagram are cubical.
If $f$ is a stratification map or a subdivision map, then so are $g$ and $h$, and if
$\phi$ is a stratification map or a subdivision map, then so are $\psi$ and $\chi$.
\end{corollary}

\begin{proof}
Let us pick a cone $A:=\fll[\sigma,\tau]^*\flr$ of $h(Q)$.
The restrictions $s\:A\to\fll\sigma\flr$ and $t\:A\to\fll\tau^*\flr$
of $r_Q$ and $\bar r_Q$ are isomorphic to the projections of
$\fll\sigma\flr\x\fll\tau^*\flr$ onto the two factors.
Hence $r_Q$ and $\bar r_Q$ are cubical.

Let $P_0=f^{-1}(\fll\sigma\flr)$ and $X_0=g^{-1}(A)$,
and consider the restriction $f_0\:P_0\to\fll\sigma\flr$ of $f$.
Then the restriction $g_0\:X_0\to A$ of $g$ is isomorphic to
$f_0\x\id_{\fll\tau^*\flr}\:P_0\x\fll\tau^*\flr\to\fll\sigma\flr\x\fll\tau^*\flr$.
Hence the composition $X_0\xr{g_0} A\xr{t}\fll\tau^*\flr$ is isomorphic
to the projection $P_0\x\fll\tau^*\flr\to\fll\tau^*\flr$.
Thus the composition $X\xr{g}h(Q)\xr{\bar r_Q}Q^*$ is cubical.
By symmetry, the composition $Y\xr{\psi}h(Q)\xr{r_Q}Q$ is also cubical.

The remaining assertions now follow using Lemma \ref{cubical pullback}.
\end{proof}

\begin{comment}
Using Lemma \ref{cubical pullback} we also get

\begin{corollary}
Consider a pullback square
$$\begin{CD} P'@>f'>>Q'\\
@V\beta VV@V\alpha VV\\
P@>f>>Q
\end{CD}$$
where $\alpha$ is a subdivision map and $f$ is cubical.
Then the induced monotone map $MC(f')\to MC(f)$ is a subdivision map.
\end{corollary}
\end{comment}

\begin{theorem} \label{stretching}
Consider a pullback diagram
$$\begin{CD}
X@>\beta>>h(Q)\\
@VfVV@Vr_QVV\\
P@>\alpha>>Q.
\end{CD}$$
If $\alpha$ is a stratification map, then $f^*\:X^*\to P^*$ is a fiberwise subdivision map
with respect to $\alpha^*\:P^*\to Q^*$ and the composition
$X^*\xr{\beta^*}Q^\#\xr{\#}Q$.
\end{theorem}

Beware that the range of the composition $\#\beta^*$ is $Q$ (and not $Q^*$),
so $f^*$ is the peculiar kind of a fiberwise subdivision map.

\begin{proof}
Let $p\in P$.
It suffices to show that $f^*|\:(f^*)^{-1}(\cel p\cer^*)\to\cel p\cer^*$
is a fiberwise subdivision map with respect to the restrictions of $\alpha^*$ and
$\#\beta^*$.
By Lemma \ref{inclusion lemma}(a) the inclusion $\cel p\cer\emb P$ is a
stratification map, hence by Corollary \ref{composition}, $\alpha|_{\cel p\cer}$
is a stratification map.
Thus we may assume that $P=\cel p\cer$.
By Lemma \ref{inclusion lemma}(b) without loss of generality $\alpha$ is
surjective, in particular, $Q=\cel\alpha(p)\cer$.

Let $r$ be a maximal element of $Q$.
Let $R=\alpha^{-1}(\fll r\flr)$.
Since $\alpha$ is a stratification map, $\partial R=\alpha^{-1}(\partial\fll r\flr)$.
Let $Q'=Q\but\{r\}$ and $P'=\alpha^{-1}(Q')$.
Since $\alpha$ is surjective, $P'\ne P$.
Let $X'=f^{-1}(P')$ and $X'_0=f^{-1}(\partial^*P')$.
Similarly, let $X=f^{-1}(P)$ and $X_0=f^{-1}(\partial^*P)$.
Arguing by induction, we may assume that $f^*|_{(X')^*}$ is a fiberwise subdivision map
with respect to $\alpha^*|_{(P')^*}$ and $\#\beta^*|_{(X')^*}$.
Then, in particular, there exists a homeomorphism
$\phi\:|X'|\to C|X'_0|$ keeping $|X'_0|$ fixed and sending $|Y_q|$ onto
$C|X'_0\cap Y_q|$ for each $q\in Q'$, where $Y_q=\beta^{-1}(h(\fll q\flr))$.
In particular, $\phi$ restricts to a homeomorphism $\phi_0\:|Z_r|\to C|X'_0\cap Z_r|$
keeping $|X'_0\cap Z_r|$ fixed, where $Z_r=\beta^{-1}(h(\partial\fll r\flr))$.
In order to show that $\phi$ extends to a homeomorphism $\phi^+\:|X|\to C|X_0|$ keeping
$|X_0|$ fixed and sending $|Y_q|$ onto $C|X_0\cap Y_q|$ for each $q\in Q$, it suffices
to show that $\phi_0$ extends to a homeomorphism $\phi_0^+\:|Y_r|\to C|X_0\cap Y_r|$
keeping $|X_0\cap Y_r|$ fixed.

By Lemma \ref{canonical of prejoin}(b),
$h(\fll r\flr)\simeq MC(r_{\partial\fll r\flr})\cup_{\partial\fll r\flr}\fll r\flr$.
Therefore by Lemma \ref{pullback of MC}, $Y_r\simeq MC(f|_{Z_r})\cup_{\partial R}R$
keeping $Z_r$ fixed; by the same token, this isomorphism sends $X_0\cap Y_r$ onto
$MC(f|_{X_0\cap Z_r})\cup_{\partial^*\partial R}\partial^*R$.
By Lemma \ref{fiberwise}, $f^*|_{Z_r^*}\:Z_r^*\to(\partial R)^*$ is a subdivision map.
Then by Corollary \ref{MC* of subdivision} there exists a monotone map
$\gamma\:MC(f|_{Z_r})\to(\partial R)\x[2]$ such that $\gamma^*$ is a subdivision map
over $[2]$.
Hence by Theorem \ref{subdivision map}, $|MC(f|_{Z_r})|$ is homeomorphic to
$|(\partial R)\x [2]|$ ``over $[2]$''.
Since $(\partial^*\partial R)^*$ is closed in $(\partial R)^*$, it follows
that this homeomorphism sends $|MC(f|_{X_0\cap Z_r})|$ onto
$|(\partial^*\partial R)\x[2]|$.
Thus
$(|Y_r|,|Z_r|)\cong (|R\cup_{\partial R}(\partial R)\x I|,|\partial R\x\{1\}|)$
by a homeomorphism sending $(|X_0\cap Y_r|,|X_0\cap Z_r|)$ onto
$(|\partial^*R\cup_{\partial^*\partial R}(\partial^*\partial R)\x I|,
|\partial^*\partial R\x\{1\}|)$.

On the other hand, since $\alpha$ is a stratification map, $\partial R$
is collared in $R$.
Since $\partial^*R$ is open in $R$, by Theorem \ref{collaring theorem} and
Addendum \ref{collaring addendum},
$(|R|,|\partial R|)\cong (|R\cup_{\partial R}(\partial R)\x I|,|\partial R\x\{1\}|)$
by a homeomorphism sending $(|\partial^*R|,|\partial^*\partial R|)$ onto
$(|\partial^*R\cup_{\partial^*\partial R}(\partial^*\partial R)\x I|,
|\partial^*\partial R\x\{1\}|)$.
Thus $(|Y_r|,|Z_r|)\cong (|R|,|\partial R|)$ by a homeomorphism sending
$(|X_0\cap Y_r|,|X_0\cap Z_r|)$ onto $(|\partial^*R|,|\partial^*\partial R|)$.
Since $(|R|,|\partial R|)$ is homeomorphic to the cone over
$(|\partial^*R|,|\partial^*\partial R|)$, this yields the desired
extension $\phi_0^+$ of $\phi_0$.
\end{proof}

\begin{theorem}\label{transversality2}
Consider the following diagram consisting of three pullback squares:
$$\begin{CD}
Z@>h>>Y@>\bar\rho_Y>>R\\
@V\chi VV@V\psi VV@V\phi VV\\
X@>g>>h(Q)@>\bar r_Q>>Q^*\\
@V\rho_XVV@Vr_QVV\\
P@>f>>Q.
\end{CD}$$
If $f$ is a subdivision map, then the composition
$g_+\:X\xr{g}h(Q)\xr{\bar r_Q}Q^*$ is dual to a subdivision map, and
if additionally $\phi$ is a stratification map, then
the composition $h_+\:Z\xr{h}Y\xr{\bar\rho_Y}R$
is dual to a subdivision map.

Dually, if $\phi$ is a subdivision map, then the composition
$\psi_+\:Y\xr{\psi}h(Q)\xr{r_Q}Q$ is dual to a subdivision map, and
if additionally $f$ is a stratification map, then the composition
$\chi_+\:Z\xr{\chi}X\xr{\rho_X}P$ is dual to a subdivision map.
\end{theorem}

\begin{proof} By symmetry it suffices to prove the assertions on
$g_+$ and $\chi_+$.

By Theorem \ref{stretching} and Lemma \ref{fiberwise subdivision map}
(with $B=Q^*$) there exists a homeomorphism $|X^*|\to |P^*|$ sending
$|(g^*_+)^{-1}(\fll q\flr)|$ onto $|(f^{-1}(\fll q\flr))^*|$
for each $q\in Q$.
Since $f$ is a subdivision map, we get that so is $g^*_+$.

It remains to show that $\chi_+^*$ is a subdivision map, assuming that
$\phi$ is a subdivision map and $f$ is a stratification map.
Let $p\in P$.
By Lemma \ref{inclusion lemma}(a) the inclusion $\cel p\cer\emb P$ is a
stratification map, hence by Corollary \ref{composition}, $f|_{\cel p\cer}$
is a stratification map.
Thus we may assume that $P=\cel p\cer$.
Then by Lemma \ref{inclusion lemma}(b) without loss of generality
$Q=\cel f(p)\cer$.

By Theorem \ref{transversality1} $\rho_X$ is cubical.
In particular, by Lemma \ref{cubical dual to stratification map} $\rho_X^*$
is a stratification map.
Hence $\partial^*X=\rho_X^{-1}(\partial^*P)$.
By Corollary \ref{transversality1}, $\chi$ is a subdivision map.
We will now examine the proof of this assertion in order to show that
$\chi$ is a fiberwise subdivision map with respect to the map
$c\:X\to[2]$ defined by $c^{-1}(2)=\partial^*X$.

Given an $x\in X$, let $q^*=g_+(x)$ and $p'=\rho_X(x)$.
Then $g(x)=[f(p'),q]$.
For convenience of reference we will mark the assertions.
The restriction of $r_Q$ [resp.\ $\bar r_Q$] to $\fll[f(p'),q]\flr$ is
isomorphic to the projection of $\fll f(p')\flr\x\fll q^*\flr$ onto
the first factor (1) [resp.\ onto the second factor (2)].
By (1), the restriction $\rho_x\:\fll x\flr\to\fll p'\flr$ of
$\rho_X$ is isomorphic to the projection
$\fll p'\flr\x\fll q^*\flr\to\fll p'\flr$ (3), and the restriction
$g_x\:\fll x\flr\to\fll g(x)\flr$ of $g$ is isomorphic to the map
$\fll p'\flr\x\fll q^*\flr\xr{f|\x\id}\fll f(p')\flr\x\fll q^*\flr$ (4).
By (2) and (4), the restriction
$g^+_x\:\fll x\flr\to\fll g_+(x)\flr$ of $g_+$ is isomorphic
to the projection $\fll p'\flr\x\fll q^*\flr\to\fll q^*\flr$.
Hence the restriction
$\chi_x\:\chi^{-1}(\fll x\flr)\to\fll x\flr$ of $\chi$ is isomorphic to
the map
$\fll p'\flr\x\phi^{-1}(\fll q^*\flr)\xr{\id\x\phi|}\fll p'\flr\x\fll q^*\flr$
(5).

If $p'\ne p$, then $x\in\partial^*X$, and it is clear that the isomorphism
in (3) sends the restriction
$\partial^*\rho_x\:\fll x\flr_{\partial^*X}\to\fll p'\flr_{\partial^*P}$
of $\rho_x$ onto the projection
$(\partial^*\fll p'\flr)\x\fll q^*\flr\to\partial^*\fll p'\flr$.
It follows that the isomorphism in (5) sends the restriction
$\partial^*\chi_x\:\chi^{-1}(\fll x\flr_{\partial^*X})\to
\fll x\flr_{\partial^*X}$ of $\chi$ onto the map
$(\partial^*\fll p'\flr)\x\phi^{-1}(\fll q^*\flr)\xr{\id\x\phi|}
(\partial^*\fll p'\flr)\x\fll q^*\flr$ (7).
Since $\phi$ is a subdivision map, we have a homeomorphism
$h\:|\phi^{-1}(\fll q^*\flr)|\to |\fll q^*\flr|$ sending
$|\phi^{-1}(\partial\fll q^*\flr)|$ onto $|\partial\fll q^*\flr|$.
Then, in view of (5), $\id_{|\fll p'\flr|}\x h$ is a homeomorphism between
$|\chi^{-1}(\fll x\flr)|$ and $|\fll x\flr|$, which sends
$|\chi^{-1}(\partial\fll x\flr)|$ onto $|\partial\fll x\flr|$
and, in view of (7), $|\chi^{-1}(\fll x\flr_{\partial^*X})|$ onto
$|\fll x\flr_{\partial^*X}|$.

Thus $\chi$ is a fiberwise subdivision map with respect to the map $c\:X\to[2]$,
where $c^{-1}(2)=\partial^*X$.
Hence by Lemma \ref{fiberwise subdivision map} there exists a homeomorphism
$|Z|\cong |X|$ sending $|\chi^{-1}(\partial^*X)|$ onto $|\partial^*X|$.
Since $\rho_X^*$ is a subdivision map by Theorem \ref{stretching}, $|X|$
is homeomorphic to $C|\partial^*X|$ keeping $|\partial^*X|$ fixed.
Hence $|Z|$ is homeomorphic to $C|\chi_+^{-1}(\partial^*\cel p\cer)|$
keeping $|\chi_+^{-1}(\partial^*\cel p\cer)|$ fixed.
\end{proof}

\part{CONSTRUCTIBLE POSETS AND MAPS} \label{chapter-collapsing}

\section{Constructibility}

\subsection{Constructible poset}\label{constructible-definition}
We call a poset $P$ {\it [transversely] constructible} if either $P$ is a cone
(i.e.\ has a greatest element) or $P=Q\cup R$, where $Q$ and $R$ are closed
subposets
of $P$ such that $Q\cap R$ is of codimension one [resp.\ collared]
both in $Q$ and in $R$, and each of the three posets $Q$, $R$ and $Q\cap R$ is
[transversely] constructible.
Such a decomposition $P=Q\cup R$ will be called a {\it [transverse] construction step}
for brevity.
Note that $\emptyset$ is not constructible.

\begin{remark}
In general topology, a subset of a space $X$ is called regular closed in $X$ if
it is equal to the closure of its interior.
A subposet $Q$ of a poset $P$ is regular closed in $P$ (with respect to
the Alexandroff topology) iff $Q=\fll S\flr$, where $S$ is a set of maximal
elements of $P$;
and nowhere dense in $P$ if it contains no maximal element of $P$.

If $P=Q\cup R$ is a construction step, it is easy to see that $Q$ and $R$
are regular closed in $P$, and $P\cap Q$ is nowhere dense in $P$ and equals
the frontier of $Q$ as well as the frontier of $R$.
\end{remark}

\begin{lemma}\label{pure lemma}
If $P=Q\cup R$ is a construction step, then $Q\cap R$ is of pure
codimension one in $Q$ and in $R$.
\end{lemma}

\begin{proof} It suffices to show that if $P$ is a constructible poset,
then an element of $P$ that is covered by a maximal element is not covered by
a non-maximal element.

This is clearly so if $P$ is a cone.
Suppose that $P=Q\cup R$ is a construction step, where $Q$ and $R$ satisfy
the property in question.
Suppose that a $p\in P$ is covered by a maximal element $q$ of $P$ and
by a non-maximal element $r$ of $P$.
If $q,r\in Q$ or $q,r\in R$, this would contradict our hypothesis since $Q$ and
$R$ are closed in $P$.
Hence we may assume by symmetry that $q\in Q$ and $r\in R$.
Then $q$ is maximal in $Q$, $r$ is non-maximal in $R$, and $p\in Q\cap R$.

If $p$ is non-maximal in $Q\cap R$, then $p<q'$ for some maximal element
$q'$ of $Q\cap R$, which is covered by a maximal element of $Q$ since $Q\cap R$
is of codimension one in $Q$.
In particular, $q'$ is non-maximal in $Q$, and then any $q''\le q'$ that covers
$p$ is non-maximal in $Q$.
Since $p$ is also covered by $q$, this contradicts our assumption that $Q$
satisfies the property in question.

If $p$ is maximal in $Q\cap R$, then $p$ is covered by a maximal element
$r'$ of $R$ since $Q\cap R$ is of codimension one in $R$.
Since $p$ is also covered by $r$, this contradicts our assumption that $R$
satisfies the property in question.
\end{proof}

\begin{lemma} \label{Hochster-lemma} Let $P$ be a poset with pure cones.
Then $C^*P$ is constructible if and only if either $P$ is a cone or the empty set,
or $P=Q\cup R$, where $Q$ and $R$ are closed subposets of $P$ such that
$\dim|Q|=\dim|R|=\dim|Q\cap R|+1$, and each of $C^*Q$, $C^*R$ and $C^*(Q\cap R)$
is constructible.
\end{lemma}

This implies that if $K$ is an affine polytopal complex, then $C^*K$ is
constructible if and only if $K$ is ``constructible'' in
the sense of Hochster \cite[p.\ 328]{Ho}, \cite{Ha0}.
Note that $C^*K$ is precisely the ``face poset'' of $K$ in sense of Topological
Combinatorics (where the empty set is regarded a face).

\begin{proof}
If $C^*P$ is constructible, then it is easy to see by induction that $P$ is
pure (using that $C^*P$ has pure cones for the induction base).
If $P$ satisfies the property in question, then again $P$ is pure by a similar
reasoning.
Now if $P$ is pure, then the condition $\dim|Q|=\dim|R|=\dim|Q\cap R|+1$ is
equivalent to saying that $Q\cap R$ is of codimension one in $Q$ and in $R$.
\end{proof}

\begin{lemma} \label{Mayer-Vietoris}
Let $P$ be a poset.

(a) If $P$ is constructible, then $|P|$ is contractible.

(b) If $C^*P$ is constructible and $P$ has pure cones, then $|P|$ is acyclic
in dimensions $<\dim|P|$.

(c) If $P$ is [transversely] constructible, then so is $C^*P$; the converse holds
if $P$ has pure cones and $|P|$ is acyclic.
\end{lemma}

Part (b) is well-known in a special case (see \cite[2.22]{Ha0}).

\begin{proof} The first assertion of (c) is obvious.
Part (a) follows from the Mayer--Vietoris sequence and
the Seifert--van Kampen theorem.

Next suppose that $P$ has pure cones and $C^*P$ is constructible, and let
$n=\dim |P|$ and let $r$ be the number of the intermediate stages $C^*P_i$
in the construction process such that $P_i=\emptyset$.
Then it is easy to see from the Mayer--Vietoris exact sequence that
$\tilde H_k(|P|)=0$ for $i<n$ and $\tilde H_n(|P|)=\Z^k$ (integer coefficients).
This implies (b) and the second assertion of (c).
\end{proof}

\begin{lemma}\label{Zeeman} (a) If $P$ is a constructible pseudo-manifold,
then $|P|$ is a PL ball.

(b) If $P$ is a transversely constructible cell complex, then $|P|$ is a PL ball.
\end{lemma}

Part (a) is well-known in a special case (see \cite[2.16]{Ha0}).

\begin{proof}
Let $P$ be a cell complex with coboundary and suppose that $P=Q\cup R$ is
either a transverse construction step or a construction step where $P$ is
a pseudo-manifold.
Then every maximal element of $Q\cap R$ is covered
by precisely one element of $Q$, which is maximal in $Q$, and
by precisely one element of $R$, which is maximal in $R$.
Arguing by induction, we may assume that $|Q|$, $|R|$ and $|Q\cap R|$ are PL balls.
Thus $|Q|$ and $|R|$ are balls of the same dimension $n$, and $|Q\cap R|$ is an
$(n-1)$-ball lying in their boundaries.
Then $|Q\cup R|$ is a ball by basic PL topology \cite[Corollary to Theorem 2]{Ze}.
\end{proof}

\begin{example} Let $M$ be a triangulation of a contractible $4$-manifold
that is distinct from the $4$-ball but has a $2$-dimensional spine that
$3$-deforms to a point (e.g.\ Mazur's manifold works).
Then $M\x I$ is a non-constructible cellulation of the $5$-ball.

Indeed, if $M\x I$ were constructible, then each construction step
would have to be of the form $P\x I=Q\x I\cup R\x I$.
Then $M$ itself would be constructible, contradicting Lemma \ref{Zeeman}.
\end{example}

The non-transverse part of the following lemma is essentially known
(in a special case), see \cite[2.14]{Ha0}.

\begin{lemma}\label{dual cones constructible}
If $P$ is a [transversely] constructible poset, then $\cel p\cer$ is
[transversely] constructible for each $p\in P$.
\end{lemma}

\begin{proof}
If $P$ is a cone, then so is $\cel p\cer$.
Assume that the assertion holds if $P$ is [transversely] constructible in
at most $n$ steps.
Suppose that $P$ is [transversely] constructible in $n+1$ steps.
Then there is a [transverse] construction step $P=Q\cup R$, where $Q$, $R$ and
$Q\cap R$ are [transversely] constructible in at most $n$ steps.
If $p\notin Q$, then $\cel p\cer_P=\cel p\cer_R$ and the assertion follows
by the inductive hypothesis.
The case $p\notin R$ is similar.
Note that the inclusion $\cel p\cer\subset P$ is a stratification map by
Lemma \ref{inclusion lemma}(a).
If $p\in Q\cap R$, then by Theorem \ref{pseudo-amalgamation} [resp.\ Lemma
\ref{amalgamation}] $\cel p\cer_{Q\cap R}$ is of codimension one
[resp.\ collared] in $\cel p\cer_Q$ and in $\cel p\cer_R$, and all three are
[transversely] constructible by the inductive hypothesis.
\end{proof}

\subsection{Constructible map}\label{constructible map-def}
A monotone map of posets $f\:P\to Q$ will be
called {\it [transversely] constructible} if it is a filtration
[resp.\ stratification] map, and
$f^{-1}(\fll q\flr)$ is [transversely] constructible for each $q\in Q$.

\begin{lemma} \label{constructible map}
Let $f\:P\to Q$ a [transversely] constructible map of posets, where $Q$ is
[transversely] constructible.
Then $P$ is [transversely] constructible.
\end{lemma}

\begin{proof}
Let $Q=R\cup S$ be a [transverse] construction step.
Then the restrictions of $f$ to $f^{-1}(R)$, to $f^{-1}(S)$ and to $f^{-1}(R\cap S)$
are [transversely] constructible maps.
Arguing by induction, we may then assume that $f^{-1}(R)$, $f^{-1}(S)$ and
$f^{-1}(R\cap S)$ are [transversely] constructible.
Since $R\cap S$ is is of codimension one [resp.\ collared] in $R$ and in $S$,
the monotone map $g\:Q\to I^*$ defined by $g(q)=\{0\}^*$ for $q\in R\but S$,
$g(q)=\{1\}^*$ for $q\in S\but R$  and $g(q)=\{0,1\}^*$ for $q\in R\cap S$ is
a filtration [resp.\ stratification] map.
By Corollary \ref{composition} [resp.\ \ref{pseudo-composition}],
the composition $P\xr{f} Q\xr{g}I^*$ is a filtration [resp.\ stratification] map.
Hence $f^{-1}(R\cap S)$ is is of codimension one [resp.\ collared]
in $f^{-1}(R)$ and in $f^{-1}(S)$.
Thus $P=f^{-1}(R)\cup f^{-1}(S)$ is a [transverse] construction step.
\end{proof}

Lemma \ref{constructible map} along with Corollary \ref{composition}
[resp.\ \ref{pure pseudo-composition}] imply:

\begin{theorem} \label{constructible composition} Composition of
[transversely] constructible maps is a [transversely] constructible map.
\end{theorem}

\begin{lemma}\label{neighborhood} Let $K$ be a poset and $L$ a full subposet
of $K$ such that $L^*$ is dense in $K^*$.

(a) If $L$ is closed in $K$, and $K^*$ is [transversely] constructible,
then so is $L^*$.

(b) If $L^*$ is constructible, then so is $K^*$.
\end{lemma}

That $L^*$ is dense in $K^*$ is of course understood with respect to
the Alexandroff topology, i.e., it is equivalent to $\fll L^*\flr=K^*$.

\begin{proof}[Proof. (a)]
Clearly, $L^*$ is a cone if and only if so is $K^*$.

Suppose that $K^*=Q\cup R$ is a [transverse] construction step.
Since $L^*$ is open in $K^*$, the inclusion $\cel p\cer\subset P$ is
a stratification map by Lemma \ref{inclusion lemma}(a).
Hence by Theorem \ref{amalgamation} [resp.\ Lemma
\ref{pseudo-amalgamation}] $L^*\cap Q\cap R$ is of codimension $\ge 1$
[resp.\ collared] in $L^*\cap Q$ and in $L^*\cap R$.
Since $L$ is a closed subposet of $K$, its intersection with any subposet
of $K$ is closed in that subposet.
Since $L$ is a full subposet of $K$, its intersection with any open subposet
of $K$ is full in that subposet by Lemma \ref{full in open}(b).

Since maximal elements of $Q$ and $R$ are maximal in $K^*$, they lie in $L^*$.
Hence $L^*\cap Q$ is dense in $Q$, and $L^*\cap R$ is dense in $R$.
If $p$ is a maximal element of $Q\cap R$, then since $Q\cap R$ is
of codimension $\ge 1$ in $Q$ and in $R$, $p$ is covered by a maximal
element $q$ of $Q$ and by a maximal element $r$ of $R$.
Then $q,r\in L^*$, and since $L$ is full in $K$, we get that $p\in L^*$.
Therefore maximal elements of $Q\cap R$ lie in $L^*$.
Hence $L^*\cap Q\cap R$ is dense in $Q\cap R$.
Thus arguing by induction, we may assume that $L^*\cap Q$, $L^*\cap R$ and
$L^*\cap Q\cap R$ are [transversely] constructible.
\end{proof}

\begin{proof}[(b)]
Suppose that $L^*=Q\cup R$ is a construction step.
Clearly $K^*=\fll Q\cup R\flr$ is the union of $\fll Q\flr$ and
$\fll R\flr$.
Maximal elements of $Q$, $R$, or $Q\cap R$ coincide with those of $\fll Q\flr$,
$\fll R\flr$, or $\fll Q\cap R\flr$, respectively.
Hence $\fll Q\cap R\flr$ is of codimension one in $\fll Q\flr$ and
in $\fll R\flr$.

To prove that $\fll Q\flr\cap\fll R\flr$ is of codimension one in
$\fll Q\flr$ and
in $\fll R\flr$ it suffices to show that $\fll Q\cap R\flr=\fll Q\flr\cap\fll R\flr$.
The inclusion $\fll Q\cap R\flr\subset\fll Q\flr\cap\fll R\flr$ is trivial.
Given a $p\in\fll Q\flr\cap\fll R\flr$, we have $p\le q$ and $p\le r$ for some
$q\in Q$ and some $r\in R$.
Since $L$ is full in $K$, the dual cone $\cel p\cer$ meets $L^*$ in a dual cone
$\cel s\cer$.
Since $q,r\in L^*$, we have found an $s\in L^*$ such that $s\le q$, $s\le r$ and
$p\le s$.
Hence $s\in Q\cap R$ and therefore $p\in\fll Q\cap R\flr$.

Arguing by induction, we would be able to assume that $\fll Q\flr$, $\fll R\flr$ and
$\fll Q\flr\cap \fll R\flr$ are constructible once we show that $Q^*$, $R^*$ and
$Q^*\cap R^*$ are full respectively in $\fll Q\flr^*$, $\fll R\flr^*$ and
$\fll Q\flr^*\cap \fll R\flr^*$.
This follows from Lemma \ref{full in open}(b), since the former ones are
the respective intersections of $L$ with the latter ones.
\end{proof}

\begin{theorem} \label{constructible point-inverses}
If $f\:K\to L$ is a simplicial map between simplicial complexes, or more generally
a full map between posets such that $f^*$ is a filtration map,
then $f^*\:K^*\to L^*$ is constructible if and only if $f^{-1}(\sigma)^*$
is constructible for each $\sigma\in L$.
\end{theorem}

That the first hypothesis is a specialization of the second follows from
Theorem \ref{partition}, Remark \ref{collared coboundary} and
Lemma \ref{point-inverses full}.

\begin{proof}
By Lemma \ref{full in open}(b), $f^{-1}(\sigma)$
is full in $f^{-1}(\cel\sigma\cer)$ for each $\sigma\in L$.
Also $f^{-1}(\sigma)$ is closed in $f^{-1}(\cel\sigma\cer)$ since
$\{\sigma\}$ is closed in $\cel\sigma\cer$.
Then by Lemma \ref{neighborhood}, $f^{-1}(\sigma)^*$ is constructible
if and only if $\fll f^{-1}(\sigma)^*\flr_{K^*}$ is constructible.
\end{proof}

%If $P$ is a cell complex and $f\:P\to Q$ is a transversely constructible map,
%then $|P|$ is a manifold collapsing onto a copy of $|Q|$ (or something like that)

\begin{theorem}\label{Cohen-Homma} If $M$ is a poset such that $|M|$ is a closed
PL manifold, and $f\:M\to K$ is a constructible map such that $f^*$ is cellular,
then $f$ is a subdivision map.
\end{theorem}

This is a variation of the Cohen--Homma theorem (Theorem \ref{Cohen-Homma2} below).

% cellular - needed?

\begin{proof} By Theorem \ref{manifolds}, $M$ is a cell complex.
Let $\sigma\in K$.
Since $M_\sigma:=f^{-1}(\fll\sigma\flr)$ is closed in $M$, it is a cell complex.
Since $f$ is a stratification, $\partial M_\sigma=f^{-1}(\partial\fll\sigma\flr)$.
Since $f^*$ is cellular, $M_\sigma^*$ is a cell complex with coboundary
$(\partial M_\sigma)^*$.
Hence $M_\sigma$ is a manifold with boundary $\partial M_\sigma$.
Since $M_\sigma$ is constructible, by Lemma \ref{Zeeman}(a) it is a ball.
In particular, $|M_\sigma|$ is homeomorphic to $|C(\partial M_\sigma)|$
keeping $|\partial M_\sigma|$ fixed.
\end{proof}

\section{Zipping}\label{zipping-definition}

Let $I$ denote the $1$-simplex.
Given a poset $P$ and a $p\in P$, suppose that $\fll p\flr$ is isomorphic to
$Q+I$ for some $Q$ by an isomorphism $h$ such that $\{p,q,r\}:=h^{-1}(I)$
is full in $P$ (in other words, $p$ is the least upper bound of $q$ and $r$
in $P$).
In this situation we say that $P$ {\it elementarily zips} onto the quotient poset
$P/\{p,q,r\}$ (see \S\ref{quotient poset}) {\it along} $p$.
If additionally $\fll p^*\flr$ is collared in $\fll q^*\flr$ and in $\fll r^*\flr$,
then we say that $P$ {\it elementarily transversely zips} onto
$P/h^{-1}(I)$ {\it along} $p$.
A {\it [transverse] zipping} is a sequence of elementary [transverse] zippings.

Non-transverse zipping was introduced by N. Reading \cite{Re}; the author has been
aware of the notion for a few years before he learned of Reading's paper
from E. Nevo.

We recall that we have called a poset $P$ {\it nonsingular} if it contains no
interval of cardinality three.

\begin{lemma} \label{constructible if zips} If a nonsingular poset $P$ zips
onto $Q$, then

(a) $Q$ is nonsingular, and

(b) the quotient map $f^*\:P^*\to Q^*$ is constructible.
\end{lemma}

\begin{proof}[Proof. (a)]
We may assume without loss of generality that there is an elementary
zipping $f\:P\to Q$ along a $p\in P$.
Suppose that $[a,c]=\{a,b,c\}$ for some $a,c\in Q$.
If $x\in Q$, $x\ne f(p)$, let $\hat x$ denote the unique preimage of $x$.
Let $q$, $r$ be the two preimages of $f(p)$ other than $p$.
If $f(p)\notin\{a,b,c\}$, then $\{\hat a,\hat b,\hat c\}=[\hat a,\hat c]$.
If $f(p)=c$, then $\{\hat a,\hat b,q\}=[\hat a,q]$.
If $f(p)=a$, then $\hat c$ covers some $s\in\{p,q,r\}$, and we have
$\{s,\hat b,\hat c\}=[s,\hat c]$.
If $f(p)=b$, then $\hat c$ covers either $p$ or some $s\in\{q,r\}$;
accordingly, we have either $\{q,p,\hat c\}=[q,\hat c]$, or
$\{\hat a,s,\hat c\}=[\hat a,\hat c]$.
In all the cases $P$ contains an interval of cardinality three, which is
a contradiction.
\end{proof}

\begin{proof}[(b)] By (a) it suffices to consider the case
where $P$ elementarily zips onto $Q$ along a $p\in P$.

Since the three-point poset $J:=f^{-1}(f(p))$ is full in $P$ and $J^*$ is
constructible, by Theorem \ref{constructible point-inverses} it suffices
to show that $f^*$ is a filtration map.

If $y\ne f(p)$, then $f^{-1}(y)$ is a singleton.
Then $(f^*)^{-1}(\partial\fll y^*\flr)$ the boundary of the cone
$(f^*)^{-1}(\fll y^*\flr)$ and so is of codimension one in it.

It remains to show that $\fll J^*\flr\but J^*$ is of codimension one
in $\fll J^*\flr$.
Write $J=\{p,q,r\}$.
Let $m$ be a maximal element of $\fll J^*\flr\but J^*$.
Suppose that $m<p^*$.
Since $P$ is nonsingular, there exists an $n\ne p^*$ such that $m<n<q^*$.
This contradicts the maximality of $m$.
Hence $m\nless p^*$.
Since $m<q^*$ or $m<r^*$, it follows that $m$ is covered by $q^*$ or $r^*$.
\end{proof}

\begin{lemma} \label{zipping-subdivision}
If $P^*$ is a cell complex and $P$ zips onto $Q$, then $Q^*$ is a cell complex,
and the quotient map $f^*\:P^*\to Q^*$ is a subdivision map.
\end{lemma}

For a direct construction of the homeomorphism between $|P|$ and $|Q|$ in
the case where $P$ is a manifold see \cite[1.4]{Ne} or, in the case where
$P$ is a sphere, \cite[4.7]{Re}.

\begin{proof}
Arguing by induction, it suffices to consider the case of an elementary zipping
along $p$.
If $J=f^{-1}(f(p))$, we claim that $\fll J^*\flr$ is a ball with boundary
$\fll J^*\flr\but J^*$.

Indeed, let $J=\{p,q,r\}$.
Since $P^*$ is a cell complex, $\fll x^*\flr$ is a ball with boundary
$\partial\fll x^*\flr$ for each $x\in P$.
Since $J$ is full in $P$, $\fll q^*\flr\cap\fll r^*\flr=\fll p^*\flr$.
Obviously $\fll p^*\flr$ is of codimension one in $\fll q^*\flr$ and in
$\fll r^*\flr$.
Hence by basic PL topology \cite[Corollary to Theorem 2]{Ze},
$\fll q^*\flr\cup\fll r^*\flr=\fll J^*\flr$ is a ball with boundary
$(\partial\fll q^*\flr\but\{p\})\cup(\partial\fll r^*\flr\but\{p\})
=\fll J^*\flr\but J^*$.

Since $P^*$ is a cell complex, it follows that $f^{-1}(\fll y^*\flr)$ is
a ball with boundary $f^{-1}(\partial\fll y^*\flr)$ for each $y\in Q$.
This implies the second assertion of the lemma.
Then the first one follows from Theorem \ref{subdivision map}.
\end{proof}

\begin{lemma}\label{zipping preserves cells}
If a cell complex $P$ zips onto $Q$, then $Q$ is a cell complex.
\end{lemma}

\begin{proof} We may assume without loss of generality that $P$ zips
elementarily onto $Q$ along a $p\in P$.
Let $f\:P\to Q$ be the quotient map.
Then $\partial\fll p\flr\simeq(\partial\fll f(p)\flr)+S^0$.
Since $\partial\fll p\flr$ is a sphere by the hypothesis, and
$|\partial\fll f(p)\flr+S^0|\cong |\partial\fll f(p)\flr|*|S^0|$,
by Morton's theorem \cite{Mo} $\partial\fll f(p)\flr$ is a sphere.

Given any $y\in Q$ other than $f(p)$, its preimage $f^{-1}(y)$
is a singleton $\{x\}$.
The restriction $f_0\:\partial\fll x\flr\to\partial\fll y\flr$ of $f$
is either an isomorphism (if $p\nless x$) or an elementary
zipping along $p$ (if $p<x$).
Now $\partial\fll x\flr$ is a sphere, hence its dual is a cell complex by
Theorem \ref{manifolds}.
Then by Lemma \ref{zipping-subdivision} $f_0$ is a subdivision.
Hence by Theorem \ref{subdivision map} $\partial\fll y\flr$ is a sphere.
\end{proof}

\begin{theorem}\label{transverse zipping vs zipping} If $P$ is a poset such that
$P^\#$ is a cell complex, and $P$ zips onto $Q$, then

(a) $Q^\#$ is a cell complex, and

(b) $P$ transversely zips onto $Q$.
\end{theorem}

We recall that $P^\#$ is a cell complex if and only if open intervals
$\partial\partial^*[p,q]$ in $P$ are spheres (Lemma \ref{semi-cell complex}).

Of course, $P^\#$ is a cell complex if either $P$ or $P^*$ is a cell
complex (by Corollary \ref{cell complex subdivision} or alternatively
by Lemma \ref{semi-cell complex} and Theorem \ref{manifolds}).

\begin{proof}[(a)]
It suffices to show that every open interval in $Q$ is a sphere, or equivalently
that $\partial^*\cel y\cer$ is a cell complex for every $y\in Q$.
Let $X=f^{-1}(\partial^*\cel y\cer)$.

If $y\ne f(p)$, then $f^{-1}(y)$ is a singleton $\{x\}$, and
$X=\partial^*\cel x\cer$, which is a cell complex.
Also, $f$ restricts either to an isomorphism $X\simeq\partial^*\cel y\cer$
or to an elementary zipping $X\to\partial^*\cel y\cer$ along $p$.
In the latter case $\partial^*\cel y\cer$ is a cell complex by
Lemma \ref{zipping preserves cells}.

Suppose that $y=f(p)$, and let $x\in\cel\{p,q,r\}\cer$.
If $x\ngtr p$, then since $\{p,q,r\}$ is full, either $x\ngtr q$
or $x\ngtr r$.
By symmetry we may assume the latter.
Then $f$ restricts to an isomorphism between $[q,x]$ and $[y,f(x)]$.
Since $\partial\partial^*[q,x]$ is a sphere, so is
$\partial\partial^*[y,f(x)]$.

It remains to consider the case $x>p$.
Then $f$ restricts to an elementary zipping
$\partial\fll x\flr\to\partial\fll f(x)\flr$ along $p$.
Since $\partial\fll x\flr$ is dual to a cell complex, by Lemma
\ref{zipping-subdivision} so is $\partial\fll f(x)\flr$.
In particular, $\partial\partial^*[y,f(x)]$ is a sphere.
\end{proof}

\begin{proof}[(b)]
By (a) it suffices to consider the case of an elementary zipping $f\:P\to Q$
along a $p\in P$.
Let $q$ and $r$ be the two preimages of $f(p)$ other than $p$, and pick
some $x>p$.
The open interval $\partial\partial^*[q,x]$ is a sphere of some dimension $n$,
hence its dual $\partial\partial^*[x^*,q^*]$ is a cell complex by
Theorem \ref{manifolds}.
Its cell $\partial^*[x^*,p^*]$ is maximal, hence is of the same dimension $n$.
On the other hand, $\partial^*[x^*,q^*]$ is an $(n+1)$-cell.
Thus $|\lk(x^*,\fll q^*\flr)|\cong C|\lk(x^*,\fll p^*\flr)$.
Similarly $|\lk(x^*,\fll r^*\flr)|\cong C|\lk(x^*,\fll p^*\flr)$.
Thus $\fll p^*\flr$ is collared in $\fll q^*\flr$ and in $\fll r^*\flr$.
\end{proof}

\begin{lemma}\label{transversely constructible if zips}
If $P^\#$ is a cell complex, and $P$ transversely zips onto $Q$, then the quotient
map $f^*\:P^*\to Q^*$ is transversely constructible.

%(b) If a nonsingular poset $P$ transversely zips onto $Q$, then the quotient
%map $f^*\:P^*\to Q^*$ is transversely constructible.
\end{lemma}

\begin{proof}
It suffices to consider a transverse elementarily zipping along $p$.
Then it is clear from the definition that $(f^*)^{-1}(\fll y^*\flr)$ is
constructible for each $y\in Q$.

Let $M=f^{-1}(\cel f(p)\cer)$ and $N=f^{-1}(\partial^*\cel f(p)\cer)$.
Let $q$ and $r$ be the two preimages of $f(p)$ other than $p$.

If $p\in N$ and $x\ngtr p$, then either
$\partial\fll x\flr_M=\{q\}+\partial\fll x\flr_N$ or
$\partial\fll x\flr_M=\{r\}+\partial\fll x\flr_N$.
Thus $\lk(x^*,M^*)\simeq C\lk(x^*,N^*)$.

Suppose that $x>p$.
Then $\partial\fll x\flr$ is dual to a cell complex, and $f$ restricts to
an elementary zipping $f_x\:\partial\fll x\flr\to\partial\fll f(x)\flr$
along $p$.
By Lemma \ref{zipping-subdivision}, $f_x^*$ is a subdivision.
Then the pair $(|\lk(x^*,M^*)|,|\lk(x^*,N^*)|)$ is homeomorphic, via the underlying
homeomorphism of $f_x^*$, to the pair
$(|\lk(f(x)^*,\fll f(p)^*\flr)|,|\lk(f(x)^*,\partial\fll f(p)^*\flr)|)$,
which in turn is of the form $(|CX|,|X|)$.

We have thus proved that $N^*$ is collared in $N^*$.
It follows that $f^*$ is a stratification map.
\end{proof}

\begin{comment}
\begin{proof}[(b) (sketch of proof)]
By Lemma \ref{constructible if zips}(a), we may assume that $P$ transversely
elementarily zips onto $Q$ along $p$.

Like in the proof of (a), it suffices to verify that $f^*$ is a stratification map.
Let $M$, $N$, $q$, $r$ be as in the proof of (a).
Like in the proof of (a), it suffices to consider the case $x>p$.

Consider the order intervals $Q=\partial[q,x]$ and $R=\partial[r,x]$;
we have $Q\cap R=\partial[p,x]$.

Let $Q'=Q\cap N$ and $R'=R\cap N$, so that $Q'\cap R'=\partial^*(Q\cap R)$.
We note that $p$ is covered by some $y\le x$, which by the non-singularity of $P$
must cover some $q'\in Q'\but R'$ and some $r'\in R'\but Q'$.
Hence $Q'\ne Q'\cap R'$ and $R'\ne Q'\cap R'$.
We have $Q\simeq C^*(\partial^*Q)$, so $|Q|\cong C|\partial^*Q|$ keeping
$|\partial^*Q|$ fixed.
On the other hand, by the hypothesis
$|\lk(x^*,\fll q^*\flr)|\cong C|\lk(x^*,\fll p^*\flr)|$ keeping
$|\lk(x^*,\fll p^*\flr)|$ fixed, in other words $|Q|\cong C|Q\cap R|$
keeping $|Q\cap R|$ fixed.
Since $Q'\ne Q'\cap R'$, by a relative version of \cite{Mo},
$|Q'|\cong C|Q'\cap R'|$ keeping $|Q'\cap R'|$ fixed.
Then $(|Q|,|Q\cap R|)\cong C(|Q'|,|Q'\cap R'|)$ keeping $|Q'|$ fixed.
Similarly $(|R|,|Q\cap R|)\cong C(|R'|,|Q'\cap R'|)$ keeping $|R'|$ fixed.
Combining this homeomorphism with the previous one, we get
$|Q\cup R|\cong C|Q'\cup R'|$.
Thus $N$ is collared in $M$.
\end{proof}
\end{comment}

From Lemmas \ref{transversely constructible if zips} and
\ref{constructible if zips} and Theorem \ref{constructible map} we obtain

\begin{theorem}\label{transversely constructible if zips2}
(a) If $P^\#$ is a cell complex, and $P$ transversely zips onto a dual cone, then $P^*$
is transversely constructible.

(b) If a nonsingular poset $P$ zips onto a dual cone, then
$P^*$ is constructible.

%(c) If a nonsingular poset $P$ transversely zips onto a dual cone, then
%$P^*$ is transversely constructible.
\end{theorem}

We now turn to a converse type of implication.

\begin{theorem}\label{zips if constructible}
(a) If $K$ is a poset such that $K^*$ is transversely constructible, then $K$ zips
onto a dual cone.

(b) If $K$ is a cell complex such that $K^*$ is constructible, then $K$
zips onto a singleton.

(c) If $K$ is a nonsingular
poset such that $K^*$ is transversely constructible, then $K$ transversely zips
onto a dual cone.
\end{theorem}

It is worth noting that each elementary zipping produced by the proof goes along
a $\sigma$ such that $\partial\fll\sigma\flr$ is a prejoin of copies of $S^0$
(the latter denotes the poset consisting of two incomparable elements).

\begin{proof} We first note that if $K$ is a cell complex and a dual cone,
then $K$ is a singleton.
Hence by Lemma \ref{zipping preserves cells}, in order to prove (b) it suffices
to show that $K$ transversely zips onto a dual cone.

If $P$ is a [transversely] constructible poset, we define its subposet
$M_P$ as follows.
If $P$ is a singleton, we set $M_P=\emptyset$, and if $P=Q\cup R$ is a
[transverse] construction step, we let
$M_P=M_Q\cup M_R\cup M_{Q\cap R}\cup M(Q\cap R)$,
where $M(P)$ denotes the set of all maximal elements of $P$.
We claim that in the transversely constructible case, $\lk(p,P)$ is a sphere
for every $p\in M_P$.
To see this, suppose that $P=Q\cup R$ is a transverse construction step, and
for $T$ running over $Q$, $R$ and $Q\cap R$, and for every $t\in M_T$,
$\lk(t,T)$ is a sphere.
If $p\in M_{Q\cap R}$ or $p\in M(Q\cap R)$, then $\lk(p,Q\cap R)$ is a sphere
(either by the assumption or because it is empty), whence
$\lk(p,P)\simeq S^0*\lk(p,Q\cap R)$ is also a sphere.
If $p\in Q\cap R$, then $\lk(p,Q)$ is a cone (over $\lk(p,Q\cap R)$), so it cannot
be a sphere; therefore $p\notin M_Q$, and similarly $p\notin M_R$.
Hence if $p\in M_Q$, then $\lk(p,P)=\lk(p,Q)$, which is a sphere by the
assumption.
Similarly, $\lk(p,P)$ is a sphere for each $p\in M_R$, which completes the
proof of the claim.

Now in each of (a), (b) we know that $K^*$ is constructible and
that $\partial\fll p\flr_K$ is a sphere for each $p^*\in M_{K^*}$.
Therefore from now on we treat (a) and (b) simultaneously.

For a [transversely] constructible poset $P$ we record its {\it scheme} of
[transverse] construction in the form of a poset $S_P$, along with a monotone map
$s_P\:P^*\to S_P$, defined as follows.
If $P$ is a cone, then $S_P$ is a singleton poset, so $s_P$ is uniquely defined.
If $P=Q\cup R$ is a [transverse] construction step, we pick some
[transverse] construction
schemes $S_Q$, $S_R$ and $S_{Q\cap R}$ and set $S_P=(S_Q\sqcup S_R)+S_{Q\cap R}$
(disjoint union in the concrete category of posets) and
$s_P|_{Q\cap R}=s_{Q\cap R}$, $s_P|_{Q\but R}=s_Q|_{Q\but R}$ and
$s_P|_{R\but Q}=s_R|_{R\but Q}$.
It follows by induction that each point-inverse of $s_P$ is a dual cone.
(In fact, if $P^*$ happens to be a cell complex, then $s_P$ can be seen to be
a bijection -- though normally not an isomorphism.)
We say that $S_T$ is a {\it (principal) subscheme} of $S_P$ if either $S_T=S_P$ or
$S_T$ is a subscheme of $S_Q$, $S_R$ or $S_{Q\cap R}$ (respectively, of $S_Q$
or $S_R$), where $P=Q\cup R$ is a [transverse] construction step.

Conversely, let us call a poset $S$ a {\it scheme poset} if $S$ is a singleton or
$S=(S_+\sqcup S_-)+S_0$, where each $S_*$ is a scheme poset.
To each non-atom $p$ of a scheme poset $S$ we associate two subsets $A^+(p,S)$ and
$A^-(p,S)$ of $S$ as follows.
In the event that $S=(S_+\sqcup S_-)+S_0$ and $p$ is a non-atom of some $S_*$
(i.e.\ of $S_+$, $S_-$ or $S_0$), we set $A^+(p,S)=A^+(p,S_*)$ and
$A^-(p,S)=A^-(p,S_*)$; and if $p$ is an atom of $S_0$, then we define $A^+(p,S)$
to be the set of atoms of $S_+$ and $A^-(p,S)$ to be the set of atoms of $S_-$.
Given a monotone map $s\:P^*\to S$ whose point inverses are dual cones,
let us write $s^{-1}(p)=\cel p_s\cer$.
Then $s$ corresponds to a construction scheme of $P$ if and only if
for each non-atom $p$ of $S$ there exist a $p^+\in A^+(p,S)$ and a
$p^-\in A^-(p,S)$ such that $p_s$ covers both $(p^+)_s$ and $(p^-)_s$.
(This is equivalent to the codimension one conditions.)

Now suppose that $X_0:=K^*$ is [transversely] constructible.
We may assume that $K$ is not a dual cone; then $X_0$ is not a cone.
If $X_0=Q\cup R$ is a [transverse] construction step, then either both $Q$ and $R$
are cones, or $S_{X_0}$ contains a proper principal subscheme $S_P$ (namely
$P=Q$ or $R$) such that $P$ is not a cone.
By iterating this observation, we can find a principal subscheme $S_{P_0}$
of $S_{X_0}$ such that there is a [transverse] construction step $P_0=Q_0\cup R_0$
where both $Q_0$ and $R_0$ are cones.
Since $X_1:=Q_0\cap R_0$ is [transversely] constructible, we can similarly find
a principal
subscheme $S_{P_1}$ of $S_{X_1}$ such that there is a [transverse] construction
step $P_1=Q_1\cup R_1$ where both $Q_1$ and $R_1$ are cones.
By iterating, we can find an $n$ such that for each $k=1,\dots,n$ there is a principal
subscheme $S_{P_k}$ of $S_{X_k}$ and a [transverse] construction step
$P_k=Q_k\cup R_k$,
where both $Q_k$ and $R_k$ are cones, $X_{k+1}=Q_k\cap R_k$, and $X_{n+1}$ is a cone.

The union of the singleton subschemes $S_{Q_i}$ and $S_{R_i}$, $i=0,\dots,n$; and
$S_{X_{n+1}}$ is a closed subposet $\Sigma$ of $S_{K^*}$, isomorphic to
$S^0+\dots+S^0+pt$ ($n+1$ copies of $S^0$).
We may write out the cones $Q_i=\fll q_i^*\flr_{X_i}$ and $R_i=\fll r_i^*\flr_{X_i}$
for $i=0,\dots,n$; and $Q_n\cap R_n=\fll p^*\flr_{X_{n+1}}$.
The monotone map $s_{K^*}\:K\to S_{K^*}$ sends each $q_i$ onto the unique element
of $S_{Q_i}$, each $r_i$ onto the unique element of $S_{R_i}$, and $p$ onto
the unique element of $S_{X_{n+1}}$.
In particular, $s_{K^*}$ restricts to a bijection between the subposet $\Pi$ of $K$
formed by $p$ and the $q_i$ and the $r_i$, $i=0,\dots,n$, and the subposet $\Sigma$
of $S_{K^*}$.
Since each $X_{i+1}$ is of codimension one in both $Q_i$ and $R_i$, every
atom of $X_{i+1}^*$ covers (in $P_i^*$, and hence in $K$) the only atom $q_i$ of
$Q_i^*$ as well as the only atom $r_i$ of $R_i^*$.
In particular, $p$ covers both $q_n$ and $r_n$.
Moreover, since each $S_{P_i}$ is a principal subscheme of $S_{X_i}$, and
$P_i=Q_i\cup R_i$ is a construction step, both $q_i$ and $r_i$ are atoms of
$X_i^*$.
Hence each of $q_i$, $r_i$ covers both $q_{i-1}$ and $r_{i-1}$, for each
$i=1,\dots,n$.
Therefore the bijection $\Pi\to\Sigma$ is an isomorphism, and $(\partial\Pi)^\flat$
contains an $n$-simplex that is maximal as a simplex of
$(\partial\fll p\flr)^\flat$.

Now $p\in M_{K^*}$ so $\partial\fll p\flr$ is a sphere, which as we have just
shown must be of dimension $n$.
Then it coincides with the $n$-sphere $\partial\Pi\subset\partial\fll p\flr$.
In other words, $\Pi$ is a closed subposet of $K$.
Then $\fll p\flr\simeq Q+I$, where $Q$ is an $(n-1)$-sphere.
Finally, since $\fll q^*_n\flr\cap\fll r^*_n\flr=\fll p^*\flr$,
the subposet $\{p,q_n,r_n\}$ of $K$ is full in $K$.
Thus we obtain an elementary [transverse] zipping $z\:K\to L$ along $p$.

Since $L$ has fewer elements than $K$, it remains to show that $L$ satisfies
the hypothesis, namely that $L^*$ is [transversely] constructible and
$\partial\fll p\flr_L$ is a sphere for each $p^*\in M_{L^*}$
[resp.\ $L$ is nonsingular].
The latter follows by the proof of Lemma \ref{zipping preserves cells}
[resp.\ from Lemma \ref{constructible if zips}(a)].

To see that $L^*$ is transversely constructible in (c), we note that $z\:K\to L$ is
a stratification map by the proof of Lemma \ref{transversely constructible if zips}.
Let $P=Q\cup R$ be a transverse construction step, where $S_P$ is any subscheme of
$S_{K^*}$ and $P\ne P_n$.
Then $S_P$ is a principal subscheme of $S_{X_i}$ for some $i<n+1$.
If $P=P_i$, then $X:=Q\cap R$ contains $\{p,q_n,r_n\}$.
If $P\ne P_i$, we claim that $X$ is disjoint from $\{p,q_n,r_n\}$.
Indeed, either $Q$ or $R$ contains $P_i$; by symmetry we may assume that
$P_i\subset Q$.
By the above $\lk(p,P_i)\simeq\lk(p,X_i)\simeq S^0+\dots+S^0$
($n+1-i$ copies).
Hence $\lk(p,Q)$ is a sphere, and in particular not a cone.
Therefore $p\notin X$; by a similar argument, $q_n\notin X$ and $r_n\notin X$.
Thus $z^{-1}(z(X))=X$, and it follows from Theorem \ref{amalgamation}(b)
that $z(P)=z(Q)\cup z(R)$ is a transverse construction step.

To see that $L^*$ is constructible in (a) and (b), let
$S_{L^*}=S_{K^*}/s_{K^*}(I)$, let $z\:S_{K^*}\to S_{L^*}$
be the quotient map, and let $s_{L^*}\:L\to S_{L^*}$ be induced by $s_{K^*}$.
If $z(p)$ is not an atom (which is the case precisely when $n>0$), then
$A^\pm(s_{L^*}z(p),S_{L^*})$ equals the one-to-one image of
$A^\pm(s_{K^*}(q_n),S_{K^*})$ or equivalently of
$A^\pm(s_{K^*}(r_n),S_{K^*})$, and we may choose $s_{L^*}z(p)^\pm$ to be
the images of $q_0$ and $r_0$.
If $s_{L^*}z(p)\in A^\pm(S_{L^*}z(x),S_{L^*})$ for some $x\in K$, then
$A^\pm(S_{K^*}(x),S_{K^*})$ contains $s_{K^*}(q_n)$ and $s_{K^*}(r_n)$,
and does not contain $s_{K^*}(p)$; if $x$ covers $q_n$ or $r_n$ in $K$,
then $z(x)$ covers $z(p)$ in $L$.
It follows that $L^*$ is constructible, with scheme $S_{L^*}$.
\end{proof}

\begin{corollary} \label{transversely constructible vs constructible}
If $K$ is a cell complex and $K^*$ is constructible, then
$K^*$ is transversely constructible.
\end{corollary}

\section{Collapsing}

\subsection{Collapsing of polyhedra}
We recall that a polyhedron $P$ is said to {\it elementarily collapse} onto
a subpolyhedron $Q$ if $P=Q\cup R$, where
$(R,R\cap Q)\cong (|\Delta^n|,|\Delta^{n-1}|)$ for some $n$.
A {\it collapse} is a sequence of elementary collapses; $P$ is
{\it collapsible} if it collapses onto a point.

\begin{lemma} \label{constructible-collapsible}
If $P$ is a transversely constructible poset, then $|P|$ is collapsible.
\end{lemma}

\begin{proof}
If $P$ is a single cone, then $|P|$ is a cone and hence collapsible.
Suppose that $P=Q\cup R$ is a transverse construction step, and let $S=Q\cap R$.
Arguing by induction, we may assume that $|Q|$, $|R|$ and $|S|$
are collapsible, say, onto points $\{q\}$, $\{r\}$ and $\{s\}$.
Let $J_{qs}$ and $J_{rs}$ be the traces of $\{s\}$ under the collapses
$|Q|\searrow\{q\}$ and $|R|\searrow\{r\}$.
This yields collapses $|Q|\searrow J_{qs}$ and $|R|\searrow J_{rs}$.
Now $J_{qs}$ and $J_{rs}$ are arcs, to they collapse onto $\{s\}$.
Thus each of $|Q|$, $|R|$ and $|S|$ collapses onto $\{s\}$.

Since $S$ is collared in $Q$ and in $R$, there are homeomorphisms
$$|Q|\cong |Q|\underset{|S|=|S|\x\{-1\}}\cup |S|\x [-1,0]\text{ and }
|R|\cong |R|\underset{|S|=|S|\x\{1\}}\cup |S|\x [0,1],$$ both sending $|S|$
onto $|S|\x\{0\}$ \cite[4.2]{Co2}.
These combine into a homeomorphism
$$|P|\cong |Q|\underset{|S|=|S|\x\{-1\}}\cup |S|\x[-1,1]\underset{|S|\x\{1\}=|S|}
\cup |R|.$$
The given collapse $|S|\searrow\{s\}$ yields a collapse
$|S|\x [-1,1]\searrow\{s\}\x[-1,1]\cup |S|\x\{-1,1\}$, so by excision, the right hand
side collapses onto
$$|Q|\underset{\{s\}=\{(s,-1)\}}\cup \{s\}\x[-1,1]\underset{\{(s,1)\}=\{s\}}\cup |R|.$$
Since $|Q|$ and $|R|$ each collapse onto $\{s\}$, the latter collapses onto
$\{s\}\x [-1,1]$ --- which is collapsible.
\end{proof}

\begin{proposition} If $P$ is a constructible $2$-dimensional simplicial
complex, then $|P|\x [0,1]$ is collapsible.
\end{proposition}

\begin{proof} Let $P=Q\cup R$ be a construction step.
Arguing by induction, we may assume that $|Q|\x [0,1]$ and $|R|\x [0,1]$ are
collapsible.
Clearly, $|P|\x I$ collapses onto
$X:=|Q|\x [0,\frac13]\cup |Q\cap R|\x [\frac13,\frac23]\cup |R|\x [\frac23,1]$.
Since $Q\cap R$ is a transversely constructible $1$-dimensional simplicial complex,
$|Q\cap R|$ is a tree, so it is collapsible.
Hence $X$ collapses onto
$Y:=|Q|\x[0,\frac13]\cup pt\x[\frac13,\frac23]\cup |R|\x[\frac23,1]$.
Let $T_1$ be the trace of $(pt,\frac13)$ under the collapse of
$|Q|\x[0,\frac13]$ onto a point, and let $T_2$ be the trace of $(pt,\frac23)$ under
the collapse of $|R|\x[\frac23,1]$ onto a point.
Then $Y$ collapses onto the tree $T_1\cup pt\x[\frac13,\frac23]\cup T_2$.
\end{proof}

The preceding proposition implies that if $P$ is a constructible
$2$-dimensional simplicial complex, then $|P|$ $3$-deforms to a point;
it follows that $|P|$ embeds in the $4$-ball (see \cite{2DH}).

See \cite{Be} concerning implications between variants of constructibility and
collapsibility in the case of triangulated balls.

\subsection{Simplicial collapsing}
We recall that a simplicial complex $K$ is said to {\it elementarily simplicially
collapse} onto a proper subcomplex $L$ if $K=L\cup\fll\sigma\flr$ for some
$\sigma\in K$ such that $L\cap\fll\sigma\flr=\{v\}*\partial\fll\tau\flr$, where
$\sigma=\{v\}*\fll\tau\flr$ for some vertex $v\in\sigma$.
The simplicial complex $K$ is said to {\it simplicially collapse} onto
a subcomplex $L$ if either $K=L$ or $K$ elementarily simplicially collapses onto
a subcomplex that simplicially collapses onto $L$.
Finally, $K$ is defined to be {\it simplicially collapsible} if it simplicially
collapses onto a singleton.

It is well-known  (see \cite{Wh}, \cite{Ze}, \cite{RS}) that a polyhedron $P$
collapses onto a
subpolyhedron $Q$ if and only if $(P,Q)$ admits a triangulation by a pair $(K,L)$
of simplicial complexes such that $K$ simplicially collapses onto $L$.

\subsection{Collapsing of posets} \label{collapsing of posets}
Let us give an inductive definition of a collapsible poset.
A poset $P$ is said to {\it elementarily collapse} onto a proper
closed subposet $Q$ if $P=Q\cup\fll\sigma\flr$ for some $\sigma\in P$ and
$Q\cap\fll\sigma\flr$ is collapsible.
$P$ is said to {\it collapse} onto a closed subposet $Q$ if either $P=Q$ or
$P$ elementarily collapses onto a closed subposet $R$ that collapses onto $Q$.
Finally, $P$ is defined to be {\it collapsible} if it collapses onto a singleton
poset.

Note that a cone elementarily collapses onto any its atom.
We have the following obvious excision rule for collapsing: if $P$ is a poset
and $Q$ is a closed subposet of $P$, then $P$ collapses onto $Q$ if and only if
$P\but R$ collapses onto $Q\but R$, where $R=P\but\fll P\but Q\flr$.

By definition, simplicial collapsibility implies collapsibility;
the converse implication holds (not only
for simplicial complexes) after passing to the barycentric subdivision:

\begin{lemma} If a poset $P$ collapses onto a closed subposet $Q$, then $P^\flat$
simplicially collapses onto $Q^\flat$.
\end{lemma}

We use the notation $\hat\sigma$ for the barycenter of a $\sigma\in P$, that is,
the vertex of $P^\flat$ formed by the singleton chain $\{\sigma\}$.

\begin{proof} Without loss of generality the given collapse is elementary,
with $P=Q\cup\fll\sigma\flr$ and $Q\cap\fll\sigma\flr$ collapsible onto
a singleton poset $\{p\}$.
Arguing by induction, we may assume that $(Q\cap\fll\sigma\flr)^\flat$
collapses simplicially onto the vertex $\{\hat p\}$.

Now $\fll\sigma\flr^\flat=\{\hat\sigma\}*(\partial\fll\sigma\flr)^\flat$
collapses simplicially onto the subjoin
$\{\hat\sigma\}*(Q\cap\fll\sigma\flr)^\flat$,
which in turn collapses simplicially onto
$\{\hat\sigma\}*\{\hat p\}\cup(Q\cap\fll\sigma\flr)^\flat$ using the
previously constructed simplicial collapse.
The latter collapses simplicially onto $(Q\cap\fll\sigma\flr)^\flat$,
so by excision,
$P^\flat$ simplicially collapses onto $Q^\flat$.
\end{proof}

\subsection{Shelling} \label{shelling}
A poset $P$ is inductively defined to be {\it elementarily [transversely]
shellable} onto a proper closed subposet $Q$ if $P=Q\cup CX$, where $CX$ is
a cone of $P$, such that $R:=Q\cap CX$ is [transversely] shellable and is of
codimension one [resp.\ collared] in $Q$, and $R\cap\fll X\but R\flr$ is
of codimension one [resp.\ collared] in $R$ and in $\fll X\but R\flr$.
(Clearly [by Lemma \ref{cojoin collars}(a)], this implies that $R$ is also
of codimension one [collared] in $CX$.)
$P$ is said to {\it [transversely] shell} onto a closed subposet $Q$ if
either $P=Q$
or $P$ elementarily [transversely] shells onto a closed subposet $R$ that
[transversely] shells onto $Q$.
Finally, $P$ is defined to be {\it [transversely] shellable} if it
[transversely] shells onto a cone.

[Transversely] shellable posets are both collapsible (onto a cone, and hence onto
a point) and [transversely] constructible.
In particular, $\emptyset$ is not shellable.

The following is proved similarly to Lemma \ref{Hochster-lemma}:

\begin{lemma} \label{shelling-lemma} Let $P$ be a poset with pure cones.
Then $C^*P$ is shellable if and only if either $P$ is a cone or the empty set,
or $P=Q\cup CX$, where $Q$ is a proper closed subposet of $P$ and $CX$ is
a cone of $P$, such that $Q$ and $R:=Q\cap CX$ are shellable, and if $CX$
is $n$-dimensional, then $Q\cap\cel CX\cer$ is also $n$-dimensional, 
$R$ is $(n-1)$-dimensional, $\fll X\but R\flr$ is purely $(n-1)$-dimensional, 
and $R\cap\fll X\but R\flr$ is purely $(n-2)$-dimensional.
\end{lemma}

It follows that if $K$ is an affine polytopal complex, then the shellability of
$C^*K$ implies ``shellability'' of $K$ in the sense of Bruggesser and Mani
\cite{BM} and is implied by ``shellability'' of $K$ in the sense of
Bj\"orner and Wachs (see \cite{Zi}, \cite{Ha0}).
Note that $C^*K$ is precisely the ``face poset'' of $K$ in sense of Topological
Combinatorics (where the empty set is regarded a face).

\begin{example}
If $P$ is an $n$-dimensional convex polytope in $\R^n$, let us consider its
boundary $dP$ and the part $dP_v$ of $dP$ that is visible from some
$v\in\R^m\but P$.
Then the face poset of $dP$ (including the empty face), and the poset
of nonempty faces of $dP_v$ for every $v\in\R^m\but P$ are shellable
(see \cite[Theorem 8.12]{Zi}).
\end{example}

\begin{remark} If $K$ is an $n$-dimensional simplicial complex,
then there are two well-known characterizations of shellability of its
face poset $C^*K$ (in this case the Bruggesser--Mani and Bj\"orner--Wachs
definitions are equivalent; see \cite{Zi}, \cite{Bj2}):

\begin{roster}
\item $K=K_0$ can be reduced to a single $n$-simplex by elementary reductions
of the form $K_{i+1}=K_i\cup\fll\sigma_i\flr$, where
$\dim|K_i\cap\fll\sigma_i\flr|=n-1$ for each $i$, and
$(\fll\sigma_i\flr,K_i\cap\fll\sigma_i\flr)\simeq
(\Delta^k*\Delta^{n-k-1},\partial\Delta^k*\Delta^{n-k-1})$ for some
$k=k(i)$ such that $0\le k\le n$.

\item $K=K_0$ is purely $n$-dimensional and can be reduced to a single
$n$-simplex by elementary reductions of the form
$K_{i+1}=K_i\cup\fll\sigma_i\flr$, where each $\dim|\fll\sigma_i\flr|=n$
and each $K_i\cap\fll\sigma\flr$ is purely $(n-1)$-dimensional (and nonempty).
\end{roster}

We note that the shellability of $K$ itself can be characterized as in (i),
but with $k\ne n$ (see also the proof of Lemma \ref{Mayer-Vietoris}(c));
this definition of shellability for simplicial complexes is found
in \cite[5.2]{Li}.
\end{remark}

Similarly to Lemma \ref{Zeeman}, one has (using \cite[Theorem 6]{Ze}):

\begin{lemma} (a) If a pseudo-manifold $N$ shells onto a manifold $M$,
then $|N|$ is homeomorphic to $|M|$.

(b) If a cell complex $N$ transversely shells onto a manifold $M$,
then $|N|$ is homeomorphic to $|M|$.
\end{lemma}

Similarly to Lemma \ref{dual cones constructible}, one has

\begin{lemma}
If $P$ is a [transversely] shellable poset, then $\cel p\cer$ is
[transversely] shellable for each $p\in P$.
\end{lemma}

The proof of Lemma \ref{neighborhood} also works to establish the following
lemma (whose $\fll L^*\flr_{K^*}$ corresponds to $K^*$ of Lemma
\ref{neighborhood}).

\begin{lemma}\label{neighborhood2}
Let $K$ be a poset, $L$ a full subposet of $K$, and $M$ a closed subposet
of $L$.

(a) If $L$ is closed in $K$, and $\fll L^*\flr_{K^*}$ [transversely] shells
onto $\fll M^*\flr_K$, then $L^*$ [transversely] shells onto $M^*$.

(b) If $L^*$ is shells onto $M^*$, then $\fll L^*\flr_{K^*}$ shells onto
$\fll M^*\flr_{K^*}$.
\end{lemma}

\subsection{Excision}
We have the following obvious excision rule for shelling: if $P$ is a poset
and $Q$ is a closed subposet of $P$, then $P$ [transversely] shells onto $Q$
if and only if $P\but R$ [transversely] shells onto $Q\but R$, where
$R=P\but\Cel\fll P\but Q\flr\Cer$.
Since $P\but R$ is open in $P$, by Lemma \ref{inclusion lemma} the inclusion
$P\but R\incl P$ is a stratification map.
Theorem \ref{amalgamation} [resp.\ \ref{pure pseudo-amalgamation}] implies
the following stronger ``map excision'' rule:

\begin{lemma}\label{map excision}
Suppose that $f\:P_1\to P_2$ is a stratification [resp.\ filtration] map,
$Q_i$ is a closed
subposet respectively of $P_i$ for $i=1,2$ such that $f^{-1}(Q_2)=Q_1$ and
$f$ restricts to an isomorphism between $P_1\but\Int Q_1$ and $P_2\but\Int Q_2$.
Then $P_1$ [transversely] shells onto $Q_1$ if and only if $P_2$
[transversely] shells onto $Q_2$.
\end{lemma}

\begin{lemma}\label{cojoin-excision}
Let $P$ and $Z$ be posets, $Q$ a closed subposet of $P$ such that
$Q\cap\fll P\but Q\flr$ is collared in $Q$ and in $\fll P\but Q\flr$, and $R$
a closed subposet of $Q$.
If $Q\x CZ\cup CQ\x Z$ transversely shells onto $R\x CZ\cup CQ\x Z$,
then $Q\x CZ\cup CP\x Z$ transversely shells onto $R\x CZ\cup CP\x Z$.
\end{lemma}

\begin{proof}
Define a retraction $\phi\:CP\to CQ$ by sending $P\but Q$ onto the maximal element
of $CQ$.
By Lemma \ref{cojoin collars}(a), $\phi$ is a stratification map.
Define a retraction $f\:Q\x CZ\cup CP\x Z\to Q\x CZ\cup CQ\x Z$ by
$f=\id_{Q\x CZ}\cup\phi\x\id_Z$.
Then $f$ is a stratification map, and the assertion follows from
Lemma \ref{map excision}.
\end{proof}

\subsection{$\emptyset$-shelling}
A poset $P$ is defined to be [transversely] $\emptyset$-shellable if it
$\emptyset$-shells onto the empty set, where an (elementary) [transverse]
$\emptyset$-shelling is defined similarly to an (elementary) [transverse] shelling.

We note that every cone elementarily transversely $\emptyset$-shells
onto $\emptyset$, so [transverse] $\emptyset$-shellability is strictly weaker
than [transverse] shellability.

It is easy to see that the product of transversely shellable posets is transversely
shellable.
The following lemma is concerned with the cojoin,
$P\cojoin Q\simeq P\x CQ\cup CP\x Q$.

\begin{lemma}\label{cojoin} Suppose that a poset $P$ transversely $\emptyset$-shells
onto a closed subposet $R$, and let $Z$ be a transversely shellable poset.
Then $P\x CZ\cup CP\x Z$ transversely shells onto $R\x CZ\cup CP\x Z$.
\end{lemma}

In particular, the cojoin of a transversely $\emptyset$-shellable poset and a
transversely shellable poset is transversely shellable.

We will only need the case $Z=pt$ of this lemma.

\begin{proof}
Suppose that $P$ elementarily transversely $\emptyset$-shells onto $Q$, where $Q$
transversely $\emptyset$-shells onto $R$.
Arguing by induction, we may assume that $Q\x CZ\cup CQ\x Z$
transversely shells onto $R\x CZ\cup CQ\x Z$.
Then by Lemma \ref{cojoin-excision}, $Q\x CZ\cup CP\x Z$ transversely shells
onto $R\x CZ\cup CP\x Z$.
Thus it suffices to show that $P\x CZ\cup CP\x Z$ transversely shells onto
$Q\x CZ\cup CP\x Z$.

We have $P=Q\cup CY$, where $S:=Q\cap CY=Q\cap Y$ is transversely
$\emptyset$-shellable and collared in $Q$, and $S\cap\fll Y\but S\flr$
is collared in $S$ and in $\fll Y\but S\flr$ (and so by
Lemma \ref{cojoin collars}(a), $S$ is also collared in $CY$).
Then $CY\x CZ$ is a single cone, which meets $Q\x CZ\cup CP\x Z$
in $S\x CZ\cup CY\x Z$.
By the second assertion of Lemma \ref{cojoin collars}(b), $S\x CZ\cup CY\x Z$
is collared in $Q\x CZ\cup CP\x Z$.
Next, we have
$\fll (Y\x CZ\cup CY\x Z)\but (S\x CZ\cup CY\x Z)\flr=\fll Y\but S\flr\x CZ$.
By the above $(\fll Y\but S\flr\cap S)\x CZ$ is collared in
$\fll Y\but S\flr\x CZ$ and in $S\x CZ$.
Then by the first assertion of Lemma \ref{cojoin collars}(b) it is also collared
in $S\x CZ\cup CY\x Z$.

It remains to show that $S\x CZ\cup CY\x Z$ is shellable.
Since $S$ is transversely $\emptyset$-shellable, we may assume by induction
that $S\x CZ\cup CS\x Z$ transversely shells onto $CS\x Z$.
Since $S\cap\fll Y\but S\flr$ is collared in $S$ and in $\fll Y\but S\flr$,
by Lemma \ref{cojoin-excision}, $S\x CZ\cup CY\x Z$ transversely shells
onto $CY\x Z$.
Since $Z$ is transversely shellable, so is $CY\x Z$.
\end{proof}

\begin{lemma}\label{quasi-shelling}
Let $X$ be a poset, $Y$ a closed subposet of $X$, and $Z$ a closed subposet of $Y$.
Then $\fll(Y^\flat)^*\flr_{(X^\flat)^*}$ transversely $\emptyset$-shells onto
$\fll(Z^\flat)^*\flr_{(X^\flat)^*}$.
\end{lemma}

\begin{proof} By arranging the elements of $Y\but Z$ in a total order
extending the original partial order, it suffices to consider the case where
$Y=Z\cup\fll\sigma\flr_X$, with $\partial\fll\sigma\flr_X\subset Z$.
Then $H_\sigma=CA\x CB$, where $A\simeq H(\partial\fll\sigma\flr)$ and
$B\simeq H(\partial\fll\sigma^*\flr)$.
We have
$H_\sigma\cup\fll(Z^\flat)^*\flr_{(X^\flat)^*}=\fll(Y^\flat)^*\flr_{(X^\flat)^*}$
and $H_\sigma\cap\fll(Z^\flat)^*\flr_{(X^\flat)^*}=A\x CB$.
Arguing by induction, we may assume that $A$ is transversely $\emptyset$-shellable,
and hence so is $A\x CB$.
\end{proof}

\begin{theorem} \label{collapsible-shellable} Let $M$ be a poset and $P$ a closed
subposet of $M$.
If $P$ collapses onto a closed subposet $Q$, then $N(P,M):=\fll H(P)\flr_{H(M)}$
shells onto $N(Q,M)$.
\end{theorem}

In particular, if $P$ is collapsible, then $N(P,M)$ is shellable
(and hence constructible).

\begin{proof} It suffices to consider the case of an elementary collapse.
Thus suppose that $P=Q\cup\fll\sigma\flr$, where $R:=Q\cap\fll\sigma\flr$ is
collapsible, $R\subset\partial\fll\sigma\flr$.
Arguing by induction, we may assume that $N(R,\partial\fll\sigma\flr)$ is
shellable.

Since $\sigma$ is maximal in $P$, $\partial H_\sigma^P$ is isomorphic to
$H(\partial\fll\sigma\flr)$.
This isomorphism takes $N(Q,P)\cap H_\sigma^P$ onto $N(R,\partial\fll\sigma\flr)$.
Since the latter is shellable, it follows that $N(Q,P)\cup H_\sigma^P$
shells onto $N(Q,P)$.

On the other hand, $H(\fll\sigma\flr)$ is isomorphic to the cojoin
$H(\partial\fll\sigma\flr)\x C(pt)\cup CH(\partial\fll\sigma\flr)\x pt$.
By Lemma \ref{quasi-shelling}, $H(\partial\fll\sigma\flr)$ transversely
$\emptyset$-shells onto $N(R,\partial\fll\sigma\flr)$.
Hence by Lemma \ref{cojoin}, the cojoin
$H(\partial\fll\sigma\flr)\x C(pt)\cup CH(\partial\fll\sigma\flr)\x pt$
transversely shells onto
$N(R,\partial\fll\sigma\flr)\x C(pt)\cup CH(\partial\fll\sigma\flr)\x pt$.
The isomorphism sends the latter onto $N(R,\fll\sigma\flr)\cup H_\sigma^P$.
Thus $H(\fll\sigma\flr)$ transversely shells onto $N(R,\fll\sigma\flr)\cup H_\sigma^P$.
By excision (using that $\Fr (N(Q,P)\cup H_\sigma^P)$ is disjoint from $Q$),
we get that $N(P,P)$ transversely shells onto $N(Q,P)\cup H_\sigma^P$.

Thus $N(P,P)$ shells onto $N(Q,P)$.
Hence by Lemma \ref{neighborhood2}, $N(P,M)$ shells onto $N(Q,M)$.
\end{proof}

\begin{corollary} \label{collapsible-shellable2}
If a poset $P$ is collapsible, then $H(P)=(P^\flat)^*$ is shellable.
\end{corollary}

\begin{corollary} \label{collapsible-constructible}
If $P$ collapses onto a closed subposet $Q$ such that $H(Q)$ is
constructible, then $H(P)$ is constructible.
\end{corollary}

\begin{proof}
Since $N(Q,Q)=H(Q)$ is constructible, by Lemma \ref{neighborhood}, so is $N(Q,P)$.
Hence $H(P)=N(P,P)$ is constructible by Theorem \ref{collapsible-shellable}.
\end{proof}

\section{Collapsible maps}\label{collapsing}

\begin{lemma}\label{Wh-collapse}
If $X$ is a simplicial complex such that $|X|$ is collapsible, then $X$ admits
a simplicially realizable subdivision map $\beta\:X'\to X$ such that $X'$ is
simplicially collapsible and $\beta^{-1}(Y)$ is simplicially collapsible
for each simplicially collapsible subcomplex $Y$ of $X$.
\end{lemma}

\begin{proof} We use Whitehead's theorem \cite[Theorem 7]{Wh} (see also
\cite[Theorem III.6]{Gl}) implying that $X$ admits a stellar subdivision map
$\beta\:X'\to X$ such that $X'$ is a simplicially collapsible.
It is easy to see by considering elementary stellar subdivisions that a
stellar subdivision of a simplicially collapsible simplicial complex is
simplicially collapsible \cite[Theorem 4]{Wh}, and the assertion follows.
\end{proof}

\begin{remark}
An alternative proof could be based on a recent result by Adiprasito and
Benedetti \cite[Corollary 3.5]{AB} implying that if $\beta\:X'\to X$ is a
simplicially realizable subdivision map, then there exists an $m$ such that
$\gamma_m^{-1}(Y)$ is simplicially collapsible for each simplicially collapsible
subcomplex $Y$ of $X$, where $\gamma_m\:(X')^{\flat m}\xr{\flat m}X'\xr{\beta}X$
is the $m$th barycentric subdivision map.
\end{remark}

\begin{theorem}\label{collapsible-constructible map}
 Let $f\:K\to L$ be a simplicial map such that $|f|$ is collapsible.
Then there exists a simplicially realizable subdivision map $\alpha\:K'\to K^\flat$
such that the composition
$g\:K'\xr{\alpha}K^\flat\xr{f^\flat}L^\flat$ is simplicial
and dual to a transversely constructible map.
\end{theorem}

That $g$ be simplicial is equivalent to saying that
$\alpha((K')^{(0)})\subset (f^\flat)^{-1}((L^\flat)^{(0)})$.

\begin{proof} Let $\sigma_1,\dots,\sigma_r$ be the simplices of $L$ arranged in an
order of increasing dimension.
Let $F_i=f^{-1}(\sigma_i)$.
The simplicially realizable subdivision map $\beta_i\:F_i^\beta\to F_i^\flat$
given by Lemma \ref{Wh-collapse} followed by the barycentric subdivision map
$\flat\:(F_i^\beta)^\flat\to F_i^\beta$ extends by taking joins with the links to
a simplicially realizable subdivision map $K_i\to K_{i-1}$, where $K_0=K^\flat$.
Let $\alpha\:K'\to K^\flat$ be the resulting subdivision map, where $K'=K_r$.
It remains to show that $g^*$ is transversely constructible.

It is easy to see that the barycentric subdivision of a simplex is
collapsible.
Hence by Lemma \ref{Wh-collapse},
$\fll\rho\flr^\beta:=\beta_i^{-1}(\fll\rho\flr^\flat)$ is collapsible for
each $\rho\in F_i$ and each $i$.

Let $K=K_0\xr{f_1}\dots\xr{f_r}K_r=L$ be the factorization of $f$ given by
Lemma \ref{factorization}, where $f_i$ shrinks $F_i$ onto an element $p_i\in K_i$.
Let $\hat p_i=(p_i)\in K_i^\flat$.
It is not hard to see that the fibers $\Phi_\tau=(f_i^\flat)^{-1}(\tau)$,
$\tau\in K_i^\flat$ of the simplicial map $f_i^\flat\:K_{i-1}^\flat\to K_i^\flat$
are $\Phi_{\hat p_i}=F_i^\flat$ and $\Phi_\tau\simeq\fll\rho_i\flr^\flat$,
where $\tau=(\dots>\rho_j>p_i)$, necessarily with $\rho_j\in F_j$, where
$\sigma_j$ covers $\sigma_i$ in $K$, where $\rho_i$ is the image of
$\rho_j$ under the Hatcher map $f_{\sigma_j\sigma_i}\:F_j\to F_i$.

Let $K'=K_0'\xr{g_1}\dots\xr{g_r}K_r'=L^\flat$ be the obvious simplicial
factorization of $g$ given by the simplicial factorization of $f^\flat$.
Thus $g_i$ shrinks $(F_i^\beta)^\flat$ onto an element $\hat p_i'\in K_i'$.
Clearly, the fibers $\Phi'_\tau=g_i^{-1}(\tau)$, $\tau\in K_i'$ of
the simplicial map $g_i$ are $\Phi'_{\hat p_i'}=(F_i^\beta)^\flat$ and
$\Phi'_\tau\simeq(\fll\rho_i\flr^\beta)^\flat$,
where $\tau=(\dots>\rho_j'>p_i)$, necessarily with $\rho_j'\in\fll\rho_j\flr^\beta$,
$\rho_j\in F_j$, where $\sigma_j$ covers $\sigma_i$ in $K$, where $\rho_i$
is the image of $\rho_j$ under the Hatcher map $f_{\sigma_j\sigma_i}\:F_j\to F_i$.

Since $F_i^\beta$ and $\fll\rho_i\flr^\beta$ are collapsible, by Corollaries
\ref{collapsible-constructible} and \ref{transversely constructible vs constructible},
$H(F_i^\beta)$ and $H(\fll\rho_i\flr^\beta)$ are transversely constructible.
Since $g_i$ is simplicial, by Corollary \ref{simplicial dual to stratification map},
$g_i^*$ is a stratification map.
Hence $g_i^*$ is transversely constructible.
Thus by Theorem \ref{constructible composition} $g^*$ is transversely
constructible.
\end{proof}

\begin{theorem} \label{collapse-composition}
Let $P$ be a polyhedron and $f\:P\to Q$ a collapsible map onto a collapsible
polyhedron.
Then $P$ is collapsible.
\end{theorem}

\begin{proof}
Let $\phi\:K\to L$ be a simplicial map triangulating $f$.
Let $L'$ be a collapsible (e.g.\ simplicially collapsible) affine simplicial
subdivision of $L$.
By Corollaries \ref{collapsible-constructible} and
\ref{transversely constructible vs constructible}, $H(L')$ is transversely
constructible.

On the other hand, by basic PL topology there is a simplicial affine subdivision
$K'$ of $K$ and a simplicial map $f\:K'\to L'$ triangulating $f$.
By Theorem \ref{collapsible-constructible map}, there exists a
simplicially realizable
subdivision map $\alpha\:K''\to (K')^\flat$ such that the composition
$K''\xr{\alpha}(K')^\flat\xr{f^\flat}(L')^\flat$ is simplicial, and its dual
is transversely constructible.
Then by Lemma \ref{constructible map}, $(K'')^*$ is transversely constructible.
Hence by Lemma \ref{constructible-collapsible} $P$ is collapsible.
\end{proof}

\begin{remark}\label{trivial-remark}
Let us also sketch a much easier (but not necessarily combinatorially clearer)
proof of Theorem \ref{collapse-composition}.
Let $K$ be a simplicially collapsible subdivision of a triangulation
of $Q$ that makes $f$ simplicial.
Then $f$ is triangulated by a simplicial map $L\to K$.
Given an elementary simplicial collapse of $K$ onto $K\but\{\sigma,\tau\}$,
where $\sigma$ is a maximal element of $K$ and $\tau$ a maximal element of
$K\but\{\sigma\}$, we can collapse $P$ onto $f^{-1}(|K\but\{\sigma\}|)\cup X$,
where $X$ is an appropriate ``section'' of $f$ over $|\fll\sigma\flr|$,
using the Hatcher decomposition (Theorem \ref{hocolim}).
The latter in turn collapses onto $f^{-1}(|K\but\{\sigma,\tau\}|)\cup X$,
which finally collapses onto $f^{-1}(|K\but\{\sigma,\tau\}|)$.
\end{remark}

\begin{corollary} \label{collapsible-composition}
Composition of collapsible maps is collapsible.
\end{corollary}

M. M. Cohen showed that composition of collapsible retractions is a collapsible
retraction \cite[8.6]{Co2}.

The following corollary of Theorems \ref{Cohen-Homma}, \ref{subdivision map}
and \ref{collapse-composition} is originally due to T. Homma
(see \cite[Lemma 5.4.1]{Ru}) and M. M. Cohen \cite{Co1}.

\begin{corollary}[Cohen--Homma]\label{Cohen-Homma2} If $M$ is a closed manifold,
$X$ is a polyhedron, and there exists a collapsible map $M\to X$, then $X$ is
homeomorphic to $M$.
\end{corollary}

A much easier (but not necessarily combinatorially clearer) proof of
Corollary \ref{Cohen-Homma2} is given by Theorem \ref{subdivision map},
the proof of Theorem \ref{Cohen-Homma} and classical results about
collapsing.

\end{document}